\documentclass[11pt]{article}
\usepackage[utf8]{inputenc}
\usepackage{amsfonts}
\usepackage{graphicx}
\usepackage{amsmath}
\usepackage{amssymb}
\usepackage{bbm}
\usepackage{hyperref}
\usepackage{cleveref}
\usepackage{mathtools}
\usepackage[affil-it]{authblk}
\usepackage[dvipsnames]{xcolor}
\usepackage{caption}
\usepackage{placeins}
\usepackage{diagbox}
\usepackage{tabularx}
\usepackage{cancel}
\usepackage[symbol]{footmisc}
\usepackage{ntheorem}
\usepackage[T1]{fontenc}
\usepackage{textcomp}
\usepackage[
backend=biber,
style=numeric,
sorting=nyt
]{biblatex}

\usepackage{scalerel}
\usepackage{stackengine,wasysym}

\newtheorem{definition}{Definition}

\newtheorem{theorem}{Theorem}
\newtheorem{lemma}{Lemma}[section]
\newtheorem{proposition}{Proposition}
\newtheorem{remark}{Remark}[section]

\usepackage{geometry}
\numberwithin{equation}{section}%

\usepackage{setspace}
\hypersetup{hidelinks}

\newcommand\reallywidetilde[1]{\ThisStyle{%
  \setbox0=\hbox{$\SavedStyle#1$}%
  \stackengine{-.1\LMpt}{$\SavedStyle#1$}{%
    \stretchto{\scaleto{\SavedStyle\mkern.2mu\AC}{.5150\wd0}}{.6\ht0}%
  }{O}{c}{F}{T}{S}%
}}

\newcommand{\E}{\mathbbm{E}}
\newcommand{\R}{\mathbbm{R}}
\newcommand{\bcdot}{\ensuremath{\boldsymbol{\cdot}}}

\newcommand{\tr}{{\scriptscriptstyle \rm T}}
\DeclareMathOperator*{\esssup}{ess\,sup}

\newcommand\warning[1]{\color{black} #1 \color{black} }

\newcommand{\CQFD}{\begin{flushright}     $\square$ \end{flushright}}

\usepackage{amsmath,array}

\title{Interpretation of stochastic primitive equations with relaxed hydrostatic assumption as a higher order approximation of 3D stochastic Navier-Stokes}

\author[1]{Arnaud Debussche}
\author[2]{Étienne Mémin}
\author[,2]{Antoine Moneyron \thanks{Corresponding author: antoine.moneyron@inria.fr}}

\affil[1]{Univ Rennes, CNRS, IRMAR UMR 6625, F35000, Rennes, France.}
\affil[2]{Univ Rennes, INRIA, IRMAR UMR 6625, F35000, Rennes, France.}

\date{\today}

\begin{document}

\maketitle
\thispagestyle{empty}

\begin{abstract}

\begin{center}
    \rule{0.5\linewidth}{1pt}
\end{center}

\vspace{2mm}

In this paper, we investigate the convergence of solutions of a stochastic representation of the three-dimensional Navier-Stokes equations to those of their primitive equations counterpart. Our analysis covers both weak and strong convergence regimes, corresponding respectively to rigid-lid and "fully periodic" boundary conditions. Furthermore, we explore the impact of relaxing the hydrostatic assumption in the stochastic primitive equations by retaining martingale terms as deviations from hydrostatic equilibrium. This modified model, obtained through a specific asymptotic scaling accessible only within the stochastic framework, captures non-hydrostatic effects while remaining within the primitive equations formalism. The resulting generalized hydrostatic model has been shown to be well-posed when the additional terms are regularized using a suitable filter \warning{for divergence-free noises under suitable assumptions.} Within this setting, we demonstrate that the model provides a higher-order approximation of the 3D Navier-Stokes equations for appropriately scaled noises.

\end{abstract}

%%%%%% Main Text %%%%%%

\newpage

\setcounter{page}{1} %Start the actual document on page 1

%\renewcommand*\contentsname{Table of contents}
%\tableofcontents

\newpage

\section{Introduction}

Representing ocean fluid flows through stochastic modeling is a dynamic and evolving area of research. Due to the vast range of interacting scales, from large-scale solar heating and atmospheric forcing to small-scale turbulent dissipation, it is computationally infeasible to resolve turbulent fluid motion explicitly with high accuracy. Consequently, only approximate models are tractable for practical applications. Furthermore, three-dimensional turbulent flows exhibit strong chaotic behavior and can give rise to spontaneous stochasticity due to the non-uniqueness of solutions. Stochastic modeling has emerged as a powerful and flexible framework for capturing unresolved variability and uncertainty in such systems \cite{Berner-Al_2017, FO_2017, FOBWL_2015, MTV_1999}. Societal needs related to climate change, economic and strategic concerns, as well as the monitoring of catastrophic events, also call for a transition toward probabilistic forecasting. Such approaches aim to provide physically relevant ensembles of realistic realizations, enabling, in particular, the quantification of uncertainties associated with specific events.

Stochastic modelling and calculus provide a rigorous framework for addressing these challenges. Early models in this domain were largely phenomenological, grounded in turbulence studies that often involved energy backscattering across scales \cite{Leith_1990, MT_1992}, and introduced stochastic parameterizations through multiplicative random forcing \cite{BMP_1999, Shutts_2005}. However, the noise variance in such models is typically uncontrolled \emph{a priori}, necessitating the addition of an eddy viscosity term to absorb excess energy. While the precise form of this viscosity remains uncertain, it is commonly derived using the debated Boussinesq turbulence assumption to obtain plausible approximations \cite{Schmitt_2007}. Moreover, when random forcing is introduced without any connection to underlying conservation laws, the resulting systems often suffer from limited accuracy and reduced interpretability \cite{CDMR_2018}.

Over the past decade, the location uncertainty (LU) approach has been developed and studied to propose physically consistent stochastic models \cite{Mémin_2014, TML_2023} based on the work of \cite{MR_2005}. This approach is based on a stochastic version of the Reynolds transport theorem, which is applied to the conservation of mass, momentum, and energy \cite{Mémin_2014}. This framework has been applied to various classical geophysical models \cite{BCCLM_2020, RMC_2017_Pt1, RMC_2017_Pt2, RMC_2017_Pt3}, reduced-order models \cite{RMHC_2017, RPMC_2021, TCM_2021}, and large eddy simulation models \cite{CMH_2020, CHLM_2018, HM_2017}. Its physical relevance has also been tested on prototypical flow models \cite{BCCLM_2020, BLBM_2021, CDMR_2018}. In addition, the well-posedness of a stochastic version of the LU 2D Navier-Stokes equation has been established, as well as the existence of a martingale solution in the 3D case \cite{DHM_2023}. Additionally, the study revealed that the stochastic model remains continuous as the noise vanishes, establishing a strong consistency with the deterministic Navier-Stokes equations. The noise considered in the LU framework is referred to as transport noise, which has been the subject of significant research within the mathematics community due to the need for well-posedness in fluid dynamics models \cite{AHHS_2022, AHHS_2022_preprint, BS_2021, CFH_2019, DHM_2023, FGL_2021, FGP_10, FL_2021, GCL_2023, LCM_2023, MR_2005}. Moreover, this type of noise has been shown to be associated with enhanced dissipation and mixing phenomena \cite{FGL_2022, FR_2023}.

In the deterministic framework, the primitive equations are commonly used to model geophysical flows \cite{Vallis2017}. These equations are derived from the 3D Navier-Stokes equations, assuming that vertical acceleration is negligible. This leads to the classical hydrostatic equilibrium, where the vertical derivative of pressure is related to density fluctuations. Although, this balance is physically valid in large-scale ocean dynamics, it breaks down outside the shallow water regime or in the presence of thermodynamic effects, such as deep convection. This assumption is further referred to as the \emph{strong hydrostatic hypothesis}. Notably, this deterministic model is known to be well-posed \cite{CT_2007} under appropriate boundary conditions (rigid-lid), and its stochastic versions have been largely explored by the mathematical community. As a matter of fact, the well-posedness of the stochastic primitive equations with multiplicative noise was demonstrated in \cite{DGHT_2011, DGHTZ_2012}, and more recently, with a specific class of regular transport noise in \cite{BS_2021}. In a more recent study, it was shown that the stochastic primitive equations with transport noise, similar to the LU framework, are well-posed under the strong hydrostatic hypothesis, using water world boundary conditions \cite{AHHS_2022}. However, this work assumes that the horizontal noise is independent of the vertical axis, which simplifies the treatment of the barotropic and baroclinic noises. This implies that the noise is bidimensional if divergence-free using rigid-lid boundary condition.

As pointed out previously, the strong hydrostatic balance does not allow to represent deep convection phenomena with strong up or down-welling of water. The LU setting allows to relax easily the strong hydrostatic balance, by considering the martingale terms of the vertical acceleration as deviation terms. This yields a generalisation of the primitive equations, which has been studied in \cite{DMM2025}. There, the authors proposed various models to interpret the relaxed (or weak) hydrostatic assumption. This modification, which captures non-hydrostatic effects while remaining within the framework of the primitive equations, arises from a specific scaling that is only accessible in the stochastic setting. In particular, they proved the well-posedness of a low-pass filtered version of the model with relaxed hydrostatic assumption, using the fact that the horizontal noise is independent of the vertical axis as in \cite{AHHS_2022}. %\warning{We anticipate that this assumption could be relaxed by considering small enough baroclinic component of the noises, e.g. compared to the (molecular) viscosity.}

In addition, it has been shown that the solutions to the weak solutions to the 3D Navier-Stokes equations converge to the ones of the primitive equations for a vanishing aspect ratio \cite{LT2019}, \warning{for a domain that is periodic in the three directions.} A similar result about the convergence of the strong solutions is also established, proving that the (deterministic) 3D Navier-Stokes are well-posed for thin enough domains. 

Following a similar method, we analyse the asymptotic error between the LU 3D Navier-Stokes equations to the low-pass filtered LU primitive equations, in the stochastic setting. One noticeable difference is that we use rigid-lid boundary conditions to study the $L^2-H^1$ convergence. However, we commute back to the "fully periodic" domain in the study of the $H^1-H^2$ convergence. In particular, we show that \warning{there exist regimes where the regularised tridimensional LU Navier-Stokes equations can be efficiently approximated by the LU primitive equations with (regularised) weak hydrostatic assumption, but not by the strongly hydrostatic ones.} \warning{The parameters of interest are the aspect ratio $\epsilon$ and a vertical noise coefficient $\alpha_\sigma$. In turn, we show that the regime mentioned above corresponds to $\alpha_\sigma = o(\epsilon^{-1})$ but $\alpha_\sigma \neq o(\epsilon^{-1/2})$. When $\alpha_\sigma = o(\epsilon^{-1/2})$, we also show the convergence towards the strongly hydrostatic primitive equations.}

Our study goes beyond extending deterministic results of \cite{LT2019} to a stochastic setting, the main challenge being to find a suitable framework in which we can use the tools of stochastic calculus. Specifically, we introduce a scaled gradient, to express the pressure terms arising in the scaled Navier-Stokes equations, as well as a scaled inner product. Then, we define a ``modified'' Leray projector to cancel these scaled gradient terms, in order to use the tools of the classical Leray theory. This is close in spirit to what was done in \cite{BS_2021} for the primitive equations. In particular, this ``modified'' projector is crucial in our stochastic setting, since It\={o}'s lemma yields covariation contributions. Thus, our projector allows to cancel the terms stemming from the scaled gradient of the stochastic pressure. In this framework, we infer that the blow-up time of the scaled Navier-Stokes equations tends to infinity in probability for a vanishing aspect ratio. This stands as a generalisation of the well-posedness result of the Navier-Stokes equations for thin enough domains in the stochastic context. Notice that, to obtain a result similar as that of the deterministic setting, we would typically need almost sure energy estimates to use the arguments of \cite{LT2019}, while we only have second order moment estimate in our study. This difficulty and type of result are common for SPDEs.

In the following we introduce the LU framework and precise our assumptions. Then we describe the abstract setting of this study and state our main results. After this, we prove the convergence of the weak solution of LU 3D Navier-Stokes equations to the ones of the LU primitive equations, with rigid-lid boundary conditions. Thus, the solutions of the latter equations with strong hydrostatic hypothesis stand as zeroth order approximations of the solutions of the former, when the primitive equations noise is bidimensional. Moreover, our estimates allow to interpret the weakly hydrostatic LU primitive equations as a higher order approximation of the LU Navier-Stokes equations, using a specific scaling on the vertical velocity noise contribution. We finally establish a similar result for strong solutions, in "fully periodic" boundary conditions.

\section{The LU formalism}\label{sec-LU-formalism}

In this section, $\mathcal{S}_0$ denotes a 3D bounded spatial domain. The LU formulation is based on the following time-scale separation of the Lagrangian displacement $X$ of flow:
\begin{equation}
    dX_t = u(X_t,t) dt + \sigma(X_t,t) dW_t.
\end{equation}
Importantly, this decomposition must be interpreted in the It\={o} sense. Moreover, the velocity component $u(X_t,t)$ denotes the large-scale Eulerian velocity (correlated in both space and time), and $\sigma(X_t,t) dW$ is a highly oscillating unresolved velocity (uncorrelated in time but correlated in space). The latter term is referred to as the (displacement) noise term.

Let us define this noise term more precisely. Consider a cylindrical Wiener process $W$ in the space of square integrable functions $\mathcal{W} := L^2(\mathcal{S}_0,\R^3)$. Then there exists a Hilbert orthonormal basis $(e_i)_{i \in \mathbbm{N}}$ of $\mathcal{W}$, and a sequence of independent standard Brownian motions $(\hat{\beta}^i)_{i \in \mathbbm{N}}$ on a filtered probability space $(\Omega, \mathcal{F}, (\mathcal{F}_t)_t, \mathbbm{P})$ such that,
\begin{equation*}
    W=\sum_{i \in \mathbbm{N}} \hat{\beta}^i e_i.
\end{equation*}
Notice, however, that the sum $\sum_{i \in \mathbbm{N}} \hat{\beta}^i e_i$ does \emph{not} converge in $\mathcal{W}$. Thus, the above identity only makes sense in a larger space $\mathcal{U} \supset \mathcal{W}$, such that the embedding $\mathcal{W} \hookrightarrow \mathcal{U}$ is Hilbert-Schmidt. For instance, $\mathcal{U}$ may be taken as the dual  of a reproducing kernel Hilbert subspace
of $\mathcal{W}$ for the inner product $(\bcdot, \bcdot)_{\mathcal{W}}$, e.g. $H^{-s}(\mathcal{S}_0, \R^3)$ with $s > \frac{3}{2}$. We then define the noise through a deterministic, time dependent correlation operator $\sigma_t$. Let $\hat{\sigma} : [0,T] \rightarrow L^2(\mathcal{S}_0^2,\R^3)$ be a bounded symmetric kernel and define
\begin{equation*}
    (\sigma_t f)(x) = \int_{\mathcal{S}_0} \hat{\sigma}(x,y,t) f(y) dy, \quad \forall f \in \mathcal{W}.
\end{equation*}
With this definition, $\sigma_t$ is a Hilbert-Schmidt operator mapping $\mathcal{W}$ into itself, so that the noise can be written as
\begin{equation*}
    \sigma_t W_t = \sum_{i \in \mathbbm{N}} \hat{\beta}_t^i \sigma_t e_i,
\end{equation*}
where the series converges in $\mathcal{W}$ almost surely, and in $L^p(\Omega, \mathcal{W})$ for all $p \in \mathbbm{N}$. We interpret $\mathcal{W}$ as the space carrying the Wiener process $W_t$, while we retain the  notation $L^2(\mathcal{S}_0,\R^3)$ to denote the space of tridimensional velocities.

Moreover, there exists a Hilbert basis $(\phi_n)_n$ of  $\mathcal{W}$ consisting of eigenfunctions of the correlation operator $\sigma_t$, scaled  by their corresponding eigenvalues. By a change of basis, there exists a sequence of independent standard Brownian motions $(\beta_t^k)_k$, defined on the same filtered probability space, such that
\begin{equation*}
    \sigma_t W_t = \sum_{k \in \mathbbm{N}} \beta_t^k \phi_k.
\end{equation*}
As such, $(\Omega, \mathcal{F}, (\mathcal{F}_t)_t, \mathbbm{P}, W)$ forms a stochastic basis. We may associate a covariance tensor to the  random field $\sigma W_t$. Given two points  $x,y \in \mathcal{S}_0$ , and two times $t,s \in \R^+$, define the covariance tensor $Q$ formally by 
\begin{equation*}
    Q(x,y,t,s) \delta_{t,s} =  \E[(\sigma_t dW_t)(x) (\sigma_s dW_s)(y)] = \int_{\mathcal{S}_0} \hat{\sigma}(x,z,t) \hat{\sigma}(z,y,s) dz \delta(t-s).
\end{equation*}
The diagonal part of this covariance tensor (in the sense of a 4x4 tensor) is referred to as the variance tensor, and is denoted
\begin{equation}
    a(x,t) = \int_{\mathcal{S}_0} \hat{\sigma}(x,y,t) \hat{\sigma}(y,x,t) dy = \sum_{k=0}^\infty \phi_k(x,t) \phi_k(x,t)^\tr \in \R^{3 \times 3}.
\end{equation}
We assume in addition that the variance tensor satisfies the integrability condition  $a \in L^1([0,T], L^2(\mathcal{S}_0, \R^{ 3\times 3}))$. In the more general setting -- when $\hat{\sigma}_t$ is itself a random function -- the random matrix-valued process $a$ must satisfy the following integrability condition,
\begin{equation*}
    \E \int_0^T \|a(\bcdot,t)\|_{L^2(\mathcal{S}_0, \R^{ 3\times 3})} < \infty,
\end{equation*}
where the norm $\| \bcdot \|_{L^2(\mathcal{S}_0, \R^{ 3\times 3}) }$ is the Hilbert norm on the space  $L^2(\mathcal{S}_0, \R^{ 3\times 3})$, equipped with  the Frobenius matrix norm. Under these conditions, the stochastic integral $\int_0^t \sigma_s dW_s $ defines a $\mathcal{W}$-valued Gaussian process with zero mean and finite variance: $\E\big[\|\int_0^t \sigma_s dW_s\|_{L^2}^2\big]<\infty$. Finally, the quadratic variation of $\int_0^t (\sigma_t d W_s)(x)$ is given by the bounded variation process $\int_0^t a(x,s) ds.$

Similarly to the classical derivation of the Navier-Stokes equations,  the LU Navier-Stokes equations are derived using a stochastic version of the Reynolds transport theorem (SRTT) \cite{Mémin_2014}. Let $q$ be a random scalar quantity defined within a volume $\mathcal{V}(t)$ transported by the flow. For incompressible unresolved flows -- i.e., when $\nabla \bcdot \sigma_t=0$ -- the SRTT reads
\begin{gather}
    d\Big( \int_{\mathcal{V}(t)} q(x,t) dx \Big) = \int_{\mathcal{V}(t)} \big(\mathbbm{D}_t q + q \nabla \bcdot (u-u_s) dt\big) dx,\\
    \mathbbm{D}_t q = d_t q + (u - u_s) \bcdot \nabla q \: dt + \sigma dW_t \bcdot \nabla q - \frac{1}{2} \nabla \bcdot (a \nabla q) dt, \label{transport-operator}
\end{gather}
where an additional drift $u_s = \frac{1}{2} \nabla \bcdot a$, termed the It\={o}-Stokes drift in \cite{BCCLM_2020}, appears. Here, $d_t q(x,t) = q(x,t+dt) - q(x,t)$ is the forward time increment at a fixed spatial point $x$, and $\mathbbm{D}_t q$ is a stochastic transport operator introduced in \cite{Mémin_2014, RMC_2017_Pt1}, playing the role of the material derivative. Moreover, the It\={o}-Stokes drift is directly related to the divergence of the variance tensor $a$, which accounts for the effects of noise inhomogeneity on large scale dynamics. Such advection terms are commonly added as corrective terms in large-scale simulations, to represent the impact of surface waves and Langmuir turbulence \cite{CL_1976, MSM_1997}. As shown in \cite{BCCLM_2020}, the LU framework exhibits similar features, generalizing the effects of the small-scale inhomogeneity on the large-scale flow.

Furthermore, the stochastic transport operator features physically interpretable terms that contribute to a large-scale representation of fluid flows. The last term on the right-hand side of \eqref{transport-operator} is an inhomogeneous diffusion, representing the mixing effects induced by unresolved small-scale variability. This stochastic diffusion is entirely determined by the variance tensor of the noise and may be viewed as a matrix-valued generalization of the classical Boussinesq eddy viscosity assumption. The third term on the right-hand side corresponds to the transport of the large-scale quantity, 
$q$, by the unresolved (random) velocity fluctuations. Remarkably, the energy injected by this backscattering term is exactly balanced by the dissipation due to the stochastic diffusion term, as shown in \cite{RMC_2017_Pt1}. This energetic balance can be interpreted as a direct manifestation of a fluctuation-dissipation relation at the level of the coarse-grained equations.

\section{Scaled Navier-Stokes equations and primitive equations in the LU formalism} \label{sec:derivation-NS-PE}

To investigate the convergence of the Navier-Stokes equations to the primitive equations within the LU framework, we begin by deriving a scaled version of the former, which formally converges to the latter in the limit of vanishing aspect ratio. Consider the LU Navier-Stokes equations posed on a thin domain $\mathcal{S_\epsilon} = \mathcal{S}_H \times (-\epsilon, \epsilon) \subset \R^3$, where $\mathcal{S}_H\subset \R^2$ is a bounded domain with smooth boundary $\partial \mathcal{S}_H$. The parameter $\epsilon >0$ represents the aspect ratio, defined as $\epsilon = \frac{H}{L}$, with  $H$ and $L$ denoting the characteristic vertical and horizontal length scales, respectively.

Throughout, we adopt the shorthand notations  $\nabla_H = (\partial_x \: \partial_y)^\tr$ and $\nabla_3 = (\partial_x \: \partial_y \: \partial_z)^\tr$, to denote the horizontal and full spatial gradients. Accordingly,  for a scalar or vector-valued field $f$, we define the horizontal and three-dimensional Laplace operators as
\begin{equation}
    \Delta_H f := \nabla_H \bcdot (\nabla_H f) = (\partial_{xx} + \partial_{yy})f, \quad \Delta_3 f := \nabla_3 \bcdot (\nabla_3 f) = (\partial_{xx} + \partial_{yy} + \partial_{zz})f.
\end{equation}
We now state the three-dimensional LU Navier-Stokes equations, which are derived by applying  the SRTT to the mass and momentum equations. The velocity field is decomposed as $u := (v \: w)^\tr$, where $v$ denotes the horizontal component and $w$ the vertical component of the velocity. Additionally, $\theta$ denotes a passive tracer, typically representing temperature or  salinity.

Within the LU formalism, we refer to $\Upsilon^{1/2} \sigma dW_t$ as the noise term of the LU Navier-Stokes equation, and distinguish its horizontal and vertical components so that $\Upsilon^{1/2} \sigma dW_t = \Upsilon^{1/2} (\sigma^H dW_t \: \sigma^z dW_t)$. The prefactor $\Upsilon^{1/2}$ encapsulates the scaling of the noise. The corresponding variance tensor is then $\Upsilon a$, implying  that $\Upsilon$ carries the physical unit of a viscosity, e.g. $m^2 s^{-1}$.  Moreover, we decompose the variance tensor as follows, 
$$a := \begin{pmatrix} a_H & a_{Hz} \\ a_{zH} & a_{zz} \end{pmatrix}$$
The It\={o}-Stokes drift is given by $u_s := \frac{1}{2} \nabla \bcdot a$ and we introduce the notations
\begin{align*}
    (v_s \: w_s)^\tr := u_s,\quad (v^* \: w^*)^\tr := u^* := u-u_s = (v - v_s \: \: \:  w - w_s)^\tr.
\end{align*}
Now we present the LU Navier-Stokes equations. In this formulation, the pressure is assumed to be a semimartingale process, denoted by  $dp_t^{SM} := p dt + dp_t^\sigma$, where $p dt$ is the bounded variation part and $dp_t^\sigma$ is the martingale component. In addition, we introduce a scaling parameter $\alpha_\sigma$, \warning{to emphasize the difference between the vertical and horizontal components of the noise. It allows for a direct comparison with the setting presented in \cite{DMM2025}. In that work, the stochastic contribution in  the vertical momentum equation are retained to derive a relaxed, or weak, form of the hydrostatic balance. As we shall see, in the regime $\alpha_\sigma^2$ large and $\alpha_\sigma \epsilon$ small,  the weak hydrostatic hypothesis provides a better approximation than the strong one.}

\paragraph{Navier-Stokes equations with linear law of state (NS)} \hfill\\
For $(x,t) \in \mathcal{S}_\epsilon \times \R_+$, the (LU) Navier-Stokes equations with linear law of state (NS) read,
\begin{align}
    d_t v + \Big((u^* dt + \Upsilon^{1/2} &\sigma dW_t) \bcdot \nabla_3\Big) v - \frac{\Upsilon}{2} \nabla_3 \bcdot(a \nabla_3 v) dt \nonumber\\
    &- \Delta_{\mu,\nu} (v dt + \Upsilon^{1/2} \sigma^H dW_t) + \frac{1}{\rho_0} \nabla_H (p\:dt + dp_t^{\sigma}) = 0, \label{NS-basic-v}\\
    d_t w + \Big((u - \alpha_\sigma^2 u_s&) dt + \alpha_\sigma \Upsilon^{1/2} \sigma dW_t) \bcdot \nabla_3\Big) w - \frac{\alpha_\sigma^2 \Upsilon}{2} \nabla_3 \bcdot(a \nabla_3 w) dt \nonumber\\
    &- \Delta_{\mu,\nu} (w dt + \alpha_\sigma \Upsilon^{1/2} \sigma^z dW_t) + \frac{1}{\rho_0} \partial_z (p\:dt + dp_t^{\sigma}) + \frac{\rho}{\rho_0} g = 0, \label{NS-basic-w}\\
    d_t \theta + \Big((u^* dt &+ \Upsilon^{1/2} \sigma dW_t) \bcdot \nabla_3\Big) \theta - \frac{\Upsilon}{2} \nabla_3 \bcdot(a \nabla_3 \theta) dt - \Delta_{\mu,\nu} \theta dt = 0 \label{NS-basic-theta},\\
    \nabla_H &\bcdot v + \partial_z w = 0, \qquad \text{ and }  \qquad \rho = \rho_0(1+\theta). \label{NS-div-and-rho}
\end{align}
Here, the anisotropic viscosity operator is given by  $\Delta_{\mu,\nu} := \mu \Delta_H + \nu \partial_{zz}$, where $\mu$ and $\nu$ are the horizontal and vertical viscosities, respectively. As is standard in the thin-domain literature \cite{PF2001, LT2019}, we assume different scalings for horizontal and vertical viscosities, namely $\mu = O(1)$ and $\nu = O(\epsilon^2)$. For simplicity, we follow \cite{LT2019} and choose $\mu = 1$ and $\nu = \epsilon^2$.

The noise term $\sigma dW_t$ is also assumed to be divergence-free:
\begin{equation}
    \nabla_H \bcdot \sigma^H dW_t + \partial_z \sigma^z dW_t = 0.
\end{equation}
 Note that equations \eqref{NS-basic-v} and \eqref{NS-basic-theta} can be interpreted as instances of the SRTT with a noise term $\sigma dW_t$, applied to the horizontal momentum and the tracer. Similar yet different, equation \eqref{NS-basic-w} is obtained by applying the SRTT to the vertical momentum, with  a rescaled noise term $\alpha_\sigma \sigma dW_t$. This distinction is crucial for capturing the weak hydrostatic regime discussed earlier.

 In the deterministic setting, the approach developed in \cite{LT2019} is to consider scaled variables which take values in $\mathcal{S} = \mathcal{S}_H \times [-1,1]$. Regarding the bounded variation terms involved in (NS), we essentially use the same scalings. However, the presence of martingale terms calls for a noise structure given a priori, so that we recover the dynamics of the LU primitive equations at the limit $\epsilon \rightarrow 0$.

\warning{
\subsection{Discussion about the noise structure}

\subsubsection{The bidimensional case}

Before presenting the rescaled version of the LU Navier-Stokes equations, we discuss the structure of the noise $\sigma dW_t$. Since the existence of global solutions to the LU primitive equations has been showed in \cite{DMM2025} for bidimensional noises, these are the type of noises we consider in our proofs. We also assume that they are divergence-free. Let $W^{2D}$ bidimensional cylindrical Wiener process over $\mathcal{S}_H$, and $\sigma_0^{2D} dW_t^{2D}$ the noise to appear in the primitive equations. Then, the (non-scaled) Navier-Stokes noise is chosen as $\sigma^{2D} dW_t^{2D} = \sigma_0^{2D} dW_t^{2D}$.

Equivalently, it is possible to define this type of noises through tridimensional Wiener processes. Consider now $W$ a cylindrical Wiener process over $\mathcal{S}_\epsilon$, and Hilbert-Schmidt operators $\sigma$ acting on $L^2(\mathcal{S}_\epsilon,\R^3)$ or $L^2(\mathcal{S},\R^3)$. We also denote by $\hat{\sigma}$ their associated kernels. Moreover, we denote by $\tilde{W}(x,y,z) := \frac{1}{\sqrt{\epsilon}} W(x,y,\epsilon z)$, which is a cylindrical Wiener process on $\mathcal{S}$ by the scaling property. Thus, assuming that the primitive equation noise $\sigma_0 d\tilde{W}_t$ is bidimensional is equivalent to
\begin{gather}
    \hat{\sigma}_0^{H} (x,y,z,x',y',z') = \hat{\sigma}_{0}^{H} (x,y,0, x',y',0), \\
    \hat{\sigma}_0^{z} (x,y,z,x',y',z') = 0.
\end{gather}
Then, we define the Navier-Stokes noise $\sigma_{NS} dW_t = (\sigma_{NS}^{H} dW_t, \sigma_{NS}^{z} dW_t)^\tr$ through its kernel: for all $(x,y,z), (x',y',z') \in \mathcal{S}_\epsilon$,
\begin{gather}
    \hat{\sigma}_{NS}^{H} (x,y,z,x',y',z') = \frac{1}{\epsilon\sqrt{\epsilon}} \hat{\sigma}_{0}^{H} (x,y,0, x',y',0), \\
    \hat{\sigma}_{NS}^{z} (x,y,z,x',y',z') = 0.
\end{gather}
The velocity $u_\epsilon = (v_\epsilon, w_\epsilon)^\tr$ and noise $\sigma_\epsilon dW_t = (\sigma_\epsilon^H dW_t, \sigma_\epsilon^z dW_t)^\tr$ appearing in the scaled Navier-Stokes equations are defined by scaling the Navier-Stokes quantities (see \cite{LT2019} for an analogous derivation in the deterministic case),
\begin{align*}
    v_\epsilon(x,y,z,t) = v(x,y,\epsilon z,t), && \sigma_\epsilon^H dW_t(x,y,z,t) = \sigma^{H} dW_t(x,y,\epsilon z,t),\\
    w_\epsilon(x,y,z,t) = \frac{1}{\epsilon} w(x,y,\epsilon z,t), && \sigma_\epsilon^z dW_t(x,y,z,t) = \frac{1}{\epsilon} \sigma^{z} dW_t^z(x,y,\epsilon z,t).
\end{align*}
Additionally, reminding that $a := \begin{pmatrix} a_H & a_{Hz} \\ a_{zH} & a_{zz} \end{pmatrix}$ denotes the \emph{non-rescaled} variance tensor, we define the \emph{rescaled} variance tensor as $a^\epsilon := \begin{pmatrix} a_H^\epsilon & a_{Hz}^\epsilon \\ a_{zH}^\epsilon & a_{zz}^\epsilon \end{pmatrix}$, where
\begin{gather*}
    a_H^\epsilon(x,y,z,t) = a_H(x,y,\epsilon z,t), \quad a_{zH}^\epsilon(x,y,z,t) = \frac{1}{\epsilon} a_{zH}(x,y,\epsilon z,t), \\
    a_{Hz}^\epsilon(x,y,z,t) = \frac{1}{\epsilon} a_{Hz}(x,y,\epsilon z,t), \quad
    a_{zz}^\epsilon(x,y,z,t) = \frac{1}{\epsilon^2} a_{zz}(x,y,\epsilon z,t).
\end{gather*}
Hence, we can also define the rescaled It\={o}-Stokes drift $u_s^\epsilon = (v_s^\epsilon \: w_s^\epsilon)^\tr$, by setting
$$v_s^\epsilon(x,y,z,t) = v_s(x,y,\epsilon z,t), \quad w_s^\epsilon(x,y,z,t) = \frac{1}{\epsilon} w_s(x,y,\epsilon z,t),$$
where $u_s = (v_s \: w_s)^\tr$ denote the non-rescaled It\={o}-Stokes drift. Then, we define $u_\epsilon^* = (v_\epsilon^* \: w_\epsilon^*)^\tr$ with $v_\epsilon^* = v_\epsilon - v_s^\epsilon$ and $w_\epsilon^* = w_\epsilon - w_s^\epsilon$.

The noise defined above can be rewritten in terms of $\tilde{W}$, which is a cylindrical Wiener process on $\mathcal{S}$. It yields ultimately $\sigma_\epsilon dW_t = \sigma_0 d\tilde{W}_t$. The derivation is performed below in the more general context of ``true'' tridimensional noises.

\subsubsection{The tridimensional case}

Even if our main results only cover bidimensional noises, we present the tridimensional noise setting below, as it gives some insights about the rescaling arguments. Considering a general primitive equations noise $\sigma_0 d\tilde{W}_t$, we define the (non-scaled) Navier-Stokes noise $\sigma dW_t = (\sigma^{H} dW_t, \sigma^{z} dW_t)^\tr$, for all $(x,y,z), (x',y',z') \in \mathcal{S}_\epsilon$, through
\begin{gather}
    \hat{\sigma}^{H} (x,y,z,x',y',z') = \frac{1}{\epsilon\sqrt{\epsilon}} \hat{\sigma}_{0}^{H} (x,y, \epsilon^{-1}z, x',y',\epsilon^{-1}z'), \label{eq:noise-struct-H}\\
    \hat{\sigma}^{z} (x,y,z,x',y',z') = \frac{1}{\sqrt{\epsilon}} \hat{\sigma}_{0}^{z} (x,y, \epsilon^{-1} z,x',y', \epsilon^{-1} z'). \label{eq:noise-struct-z}
\end{gather}
Notice that in this case, $\sigma dW_t$ is divergence-free whenever $\sigma d\tilde{W}_t$ is. Moreover, this choice leads to a scaled Navier-Stokes noise that is equal to $\sigma_0 d\tilde{W}_t$. To see this, we perform first a change of variable on $\sigma_\epsilon^H dW_t$,
\begin{align*}
    \sigma_\epsilon^H (x,y,z) dW_t &= \int_{\mathcal{S_\epsilon}} \hat{\sigma}^{H}(x,y,\epsilon z,x',y',z') dW_t(x',y',z') dx' dy'dz'\\
    &=\epsilon \int_{\mathcal{S}} \hat{\sigma}^{H} (x,y,\epsilon z,x',y',\epsilon z'') dW_t(x',y',\epsilon z'') dx' dy'dz''\\
    &= \epsilon \sqrt{\epsilon} \int_{\mathcal{S}} \hat{\sigma}^{H} (x,y,\epsilon z,x',y',\epsilon z'') d\tilde{W}_t(x',y', z'') dx' dy'dz'' = \sigma_{0}^H(x,y, \epsilon z) d\tilde{W}_t.
\end{align*}
Similarly,
\begin{align*}
    \sigma_\epsilon^z (x,y,z) dW_t &= \frac{1}{\epsilon} \int_{\mathcal{S_\epsilon}} \hat{\sigma}^{z}(x,y,\epsilon z,x',y',z') dW_t(x',y',z') dx' dy'dz'\\
    &= \int_{\mathcal{S}} \hat{\sigma}^{z} (x,y,\epsilon z,x',y',\epsilon z'') dW_t(x',y',\epsilon z'') dx' dy'dz''\\
    &= \sqrt{\epsilon} \int_{\mathcal{S}} \hat{\sigma}^{z} (x,y,\epsilon z,x',y',\epsilon z'') d\tilde{W}_t(x',y', z'') dx' dy'dz'' = \sigma_{0}^z(x,y, z) d\tilde{W}_t.
\end{align*}
Using the decomposition $\sigma_0 d\tilde{W}_t = \sum_k \phi_k^0 d\beta_k$, we define $a^0 = \sum_k \phi_k^0 (\phi_k^0)^\tr$ and $u_s^0 = \frac{1}{2} \nabla \bcdot a_0$. Thus, it holds that $a^\epsilon = a^0$ and $u_s^\epsilon = u_s^0$.

Eventually, we can remark that, upon assuming $\sigma_0^H dW_t$ is independent of $z$, we recover the type of noise proposed in \cite{AHHS_2022}. Another possibility to recover this type of noises is to consider a general reference kernel $\hat{\sigma}_{ref}$, and a real number $\gamma \in [0,1)$. Then, we define the following (non-rescaled) Navier-Stokes noise,
\begin{gather}
    \hat{\sigma}^{H} (x,y,z,x',y',z') = \frac{1}{\epsilon\sqrt{\epsilon}} \hat{\sigma}_{ref}^{H} (x,y, \epsilon^{-\gamma}z, x',y', \epsilon^{-\gamma}z'),\\
    \hat{\sigma}^{z} (x,y,z,x',y',z') = \frac{1}{\sqrt{\epsilon}} \hat{\sigma}_{ref}^{z} (x,y, \epsilon^{-1} z,x',y', \epsilon^{-1} z').
\end{gather}
Thus, the kernel of the scaled noise would be,
\begin{gather}
    \hat{\sigma}_\epsilon^{H} (x,y,z,x',y',z') = \hat{\sigma}_{ref}^{H} (x,y, \epsilon^{1-\gamma} z, x',y', \epsilon^{1-\gamma} z') \\
    \hat{\sigma}^{z} (x,y,z,x',y',z') = \hat{\sigma}_{ref}^{z} (x,y, z,x',y', z').
    \end{gather}
As such, the scaled noise defined above converges in law to a noise of the desired form. However, if we impose a divergence-free condition on $\sigma_0 dW_t$, together with the rigid-lid boundary conditions, then $\sigma_0 dW_t$ must be bidimensional. Consequently, the noises above are more general only for divergent noises. This setting is typically relevant in the context of weakly compressible dynamics, which emerges traditionally from the Boussinesq assumption.
}

\subsection{Scaled Navier-Stokes equations and primitive equations} \label{subsec:Navier-Stokes-and-primitive}

We derive below the scaled Navier-Stokes equations, on the scaled domain $\mathcal{S} = \mathcal{S}_H \times (-1, 1)$. For this purpose, we define the following scaled variables,
\begin{align*}
    v_\epsilon(x,y,z,t) = v(x,y,\epsilon z,t), && \sigma_\epsilon^H dW_t(x,y,z,t) = \sigma^{H} dW_t(x,y,\epsilon z,t),\\
    w_\epsilon(x,y,z,t) = \frac{1}{\epsilon} w(x,y,\epsilon z,t), && \sigma_\epsilon^z dW_t(x,y,z,t) = \frac{1}{\epsilon} \sigma^{z} dW_t^z(x,y,\epsilon z,t),\\
    p_\epsilon(x,y,z,t) = p(x,y,\epsilon z,t), && dp_t^{\sigma,\epsilon}(x,y,z,t) = dp_t^{\sigma}(x,y,\epsilon z,t),\\
    \rho_\epsilon - \rho_0 = \epsilon (\rho - \rho_0)(x,y,\epsilon z,t), && \theta_\epsilon(x,y,z,t) = \epsilon \theta(x,y,\epsilon z,t).
\end{align*}

\begin{remark}\hfill
    Thus far, $\epsilon$ has been interpreted as the aspect ratio $\epsilon := \frac{H}{L}$, where $H$ and $L$ are the typical vertical and horizontal lengths, respectively. However, more general scalings may be considered. For instance, one could introduce the scaled parameter $\check{\epsilon}$, defined by
    $$\check{\epsilon}^2 = \frac{Fr^2 Ro}{Bu}\epsilon^2 = Ro^3 Bu^{-2} \epsilon^2$$
    where $Ro = \frac{U}{f_0 L}$, $Bu = \frac{R_d^2}{L^2}$ and $Fr^2 = \frac{U^2}{gH} = Ro^2 Bu^{-1}$, $R_d = \frac{(gH)^{1/2}}{f_0}$ being the Rossby deformation radius, $f_0$ the Coriolis parameter and $U$ the typical velocity \cite{pedlosky2013geophysical}.
\end{remark}
Following the derivation of \cite{LT2019}, we find the scaled Navier-Stokes equations under location uncertainty (SNS). For consistency, we formulate our equations with the cylindrical Wiener process $\tilde{W}$ on $\mathcal{S}$.

\paragraph{Scaled Navier-Stokes equation (SNS$_{\alpha_\sigma}^\epsilon$)}
For $(x,t) \in \mathcal{S} \times \R_+$, 
\begin{align}
    d_t v_\epsilon + \Big((u_\epsilon^* dt &+ \Upsilon^{1/2} \sigma_0 d\tilde{W}_t) \bcdot \nabla_3\Big) v_\epsilon - \frac{\Upsilon}{2} \nabla_3 \bcdot(a^0 \nabla_3 v_\epsilon) dt \nonumber\\
    &- \Delta_3 (v_\epsilon dt + \Upsilon^{1/2} \sigma_0^H d\tilde{W}_t) + \frac{1}{\rho_0} \nabla_H (p_\epsilon\:dt + dp_t^{\sigma, \epsilon}) = 0, \\
    \epsilon^2 \Bigg(d_t w_\epsilon + \Big((u_\epsilon^* dt &+ \alpha_\sigma \Upsilon^{1/2} \sigma_0 d\tilde{W}_t) \bcdot \nabla_3\Big) w_\epsilon - \frac{\alpha_\sigma^2 \Upsilon}{2} \nabla_3 \bcdot(a^0 \nabla_3 w_\epsilon) dt \nonumber\\
    &- \Delta_3 (w_\epsilon dt + \alpha_\sigma \Upsilon^{1/2} \sigma_0^z d\tilde{W}_t)\Bigg) + \frac{1}{\rho_0} \partial_z (p_\epsilon\:dt + dp_t^{\sigma, \epsilon}) + \frac{\rho_\epsilon}{\rho_0} g = 0, \label{eq:SNS-w}\\
    d_t \theta_\epsilon + \Big((u_\epsilon^* dt + &\Upsilon^{1/2} \sigma_0 d\tilde{W}_t) \bcdot \nabla_3\Big) \theta_\epsilon - \frac{\Upsilon}{2} \nabla_3 \bcdot(a^0 \nabla_3 \theta_\epsilon) dt - \Delta_3 \theta_\epsilon dt = 0,\\
    \nabla_H &\bcdot v_\epsilon + \partial_z w_\epsilon = 0, \qquad \text{ and } \qquad \rho_\epsilon = \rho_0(1+\theta_\epsilon),
\end{align}
with again, a divergence-free assumption on the noise,
\begin{equation}
    \nabla_3 \bcdot \sigma_0 d\tilde{W}_t = \nabla_H \bcdot \sigma_0^H d\tilde{W}_t + \partial_z \sigma_0^z d\tilde{W}_t = 0. \label{eq-NS-scaled-div-free-noise}
\end{equation}
As mentioned earlier, $\sigma_0 d\tilde{W}_t$ is chosen to be the noise involved in the primitive equations below. Importantly, the Wiener processes $\tilde{W}$ and $W$ are not supported by the same spaces: $\tilde{W}$ is defined on the ``scaled'' domain $\mathcal{S} = \mathcal{S}_H \times (-1,1)$, while $W$ is defined on the ``physical'' domain $\mathcal{S}_H \times (-\epsilon,\epsilon)$. Nevertheless, the divergence-free assumption on $\sigma_\epsilon dW_t$ implies that its limit $\sigma d\tilde{W}_t$ is also divergence-free. 

To cancel the pressure terms, we introduce hereafter the scaled gradient operator $\nabla_\epsilon$, defined, for any $f \in H^1(\mathcal{S}, \R)$, as
\begin{equation}
    \nabla_\epsilon f = \begin{pmatrix}
        \nabla_H f \\
        \frac{1}{\epsilon^2} \partial_{z} f
    \end{pmatrix}.
\end{equation}
Hence, the system (SNS$_{\alpha_\sigma}^\epsilon$) can be written more compactly as
\begin{align}
    d_t u_\epsilon + (u_\epsilon^* \bcdot \nabla_3) u_\epsilon dt &+ \Upsilon^{1/2} (\sigma_0 d\tilde{W}_t \bcdot \nabla_3) I_{\alpha_\sigma} u_\epsilon - \frac{\Upsilon}{2} \nabla_3 \bcdot(a^0 \nabla_3 I_{\alpha_\sigma^2} u_\epsilon) dt \nonumber\\
    - \Delta_3 (u_\epsilon dt &+ \Upsilon^{1/2} I_{\alpha_\sigma} \sigma_0 d\tilde{W}_t) + \frac{1}{\rho_0} \nabla_\epsilon (p_\epsilon\:dt + dp_t^{\sigma, \epsilon}) + \frac{\rho_\epsilon}{\epsilon^2 \rho_0} g \bold{e}_z = 0, \label{eq:SNS-noreg-u}\\
    d_t \theta_\epsilon + \Big((u_\epsilon^* dt + &\Upsilon^{1/2} \sigma_0 d\tilde{W}_t) \bcdot \nabla_3\Big) \theta_\epsilon - \frac{\Upsilon}{2} \nabla_3 \bcdot(a^0 \nabla_3 \theta_\epsilon) dt - \Delta_3 \theta_\epsilon dt = 0, \label{eq:SNS-noreg-theta}\\
    \nabla_H &\bcdot v_\epsilon + \partial_z w_\epsilon = 0, \qquad \text{ and } \qquad \rho_\epsilon = \rho_0(1+\theta_\epsilon), \label{eq:SNS-noreg-div-density}
\end{align}
where we define $I_{\alpha} = Diag(1, 1, \alpha)$, for all $\alpha > 0$. Again, we assume that the incompressibility condition \eqref{eq-NS-scaled-div-free-noise} holds.

Notice that the first term of equation \eqref{eq:SNS-w} involves terms of order $\epsilon^2$, $\alpha_\sigma \epsilon^2$ and $\alpha_\sigma^2 \epsilon^2$. By analogy with the deterministic case, where the primitive equations are formally recovered by disregarding the terms scaling with $\epsilon^2$, we infer the expression of the primitive equations under location uncertainty with strong hydrostatic hypothesis (PE). They are found by disregarding all the terms mentioned above.
\paragraph{Primitive equations with strong hydrostatic hypothesis (PE)} \hfill \\
For $(x,t) \in \mathcal{S} \times \R_+$,
\begin{align}
    d_t v + \Big((u^* dt + \Upsilon^{1/2} \sigma_0 d\tilde{W}_t) \bcdot \nabla_3\Big) v &- \frac{\Upsilon}{2} \nabla_3 \bcdot(a \nabla_3 v) dt - \Delta_3 (v dt + \Upsilon^{1/2} \sigma_0 d\tilde{W}_t^H) \nonumber \\
    &\qquad \qquad \qquad \qquad + \frac{1}{\rho_0} \nabla_H (p\:dt + dp_t^{\sigma}) = 0,\\
    d_t \theta + \Big((u^* dt + \Upsilon^{1/2} \sigma_0 d\tilde{W}_t) \bcdot \nabla_H\Big) \theta &- \frac{\Upsilon}{2} \nabla_3 \bcdot(a \nabla_3 \theta) dt - \Delta_3 \theta dt = 0,\\
    \partial_z dp_t^{\sigma} = 0, \quad &\text{ and } \quad \partial_z p = -\rho g, \\
    \nabla_H \bcdot v + \partial_z w = 0, \quad &\text{ and } \quad \rho = \rho_0(1+\theta).
\end{align}
However, when $\alpha_\sigma \gg 1$, the terms scaling like $\alpha_\sigma \epsilon^2$ and $\alpha_\sigma^2 \epsilon^2$ become much larger than those of scaling $\epsilon^2$. Reasoning similarly as in \cite{DMM2025}, we retain only the former, so that we obtain the non-regularised primitive equations with weak hydrostatic assumptions, which are defined below. Thus, this model is intermediate between the LU primitive equations (PE) and the LU scaled Navier-Stokes equations (SNS$_{\alpha_\sigma}^\epsilon$).
\paragraph{Non-regularised primitive equations with weak hydrostatic assumption} \hfill \\
For $(x,t) \in \mathcal{S} \times \R_+$,
\begin{align}
    d_t v + \Big((u^* dt + \Upsilon^{1/2} \sigma_0 d\tilde{W}_t) \bcdot &\nabla_3\Big) v - \Delta_3 (v dt + \Upsilon^{1/2} \sigma_0^H d\tilde{W}_t) \nonumber \\
    &+ \frac{1}{\rho_0} \nabla_H (p\:dt + dp_t^{\sigma}) - \frac{\Upsilon}{2} \nabla_3 \bcdot(a^0 \nabla_3 v) dt= 0,\\
    d_t \theta + \Big((u^* dt + \Upsilon^{1/2} \sigma_0 &d\tilde{W}_t) \bcdot \nabla_3\Big) \theta - \frac{\Upsilon}{2} \nabla_3 \bcdot(a^0 \nabla_3 \theta) dt - \Delta \theta dt = 0,\\
    \frac{1}{\rho_0} \partial_z p = -\frac{\rho}{\rho_0} &g + \alpha_\sigma^2 \epsilon^2 \frac{\Upsilon}{2} \nabla_3 \bcdot (a^0 \nabla_3 w),\\
    \frac{1}{\rho_0} \partial_z dp_t^{\sigma} = - \alpha_\sigma \epsilon^2 [(\Upsilon^{1/2} \sigma_0 &d\tilde{W}_t \bcdot \nabla) w] + \alpha_\sigma \epsilon^2 \Delta (\Upsilon^{1/2} \sigma_0^z d\tilde{W}_t), \\
    \nabla_H \bcdot v + \partial_z & w = 0, \quad \text{ and } \quad \rho = \rho_0(1+\theta).
\end{align}
Regarding the martingale pressure term, we observe that the horizontal momentum equation includes the additional contribution
\begin{equation}
    \alpha_\sigma \epsilon^2 \nabla_H \int_z^1 [(\Upsilon^{1/2} \sigma_0 d\tilde{W}_t \bcdot \nabla) w] dz'.
\end{equation}
This expression involves third-order derivatives in $v$, since $w$ is determined through the divergence-free condition as $w = \int_z^1 (\nabla_H \bcdot v) dz'$. To regularise this term, we introduce a convolution kernel $K$, and define a smoothed  approximation of the above system,  (PE$_{\alpha_\sigma}^\epsilon$). By construction  the systems (PE$_{\alpha_\sigma}^\epsilon$) and (PE) coincide whenever $\epsilon = 0$ or $\alpha_\sigma =0$.

\paragraph{(Regularised) primitive equations with weak hydrostatic assumption (PE$_{\alpha_\sigma}^\epsilon$)} \hfill \\
For $(x,t) \in \mathcal{S} \times \R_+$,
\begin{align}
    d_t v + \Big((u^* dt + \Upsilon^{1/2} \sigma_0 d\tilde{W}_t) \bcdot &\nabla_3\Big) v - \Delta_3 (v dt + \Upsilon^{1/2} \sigma_0^H d\tilde{W}_t) \nonumber \\
    &+ \nabla_H (p\:dt + dp_t^{\sigma}) - \frac{\Upsilon}{2} \nabla_3 \bcdot(a^0 \nabla_3 v) dt= 0,\\
    \partial_z p = -\frac{\rho}{\rho_0} &g + \alpha_\sigma^2 \epsilon^2 \frac{\Upsilon}{2} \nabla_3 \bcdot (a^K \nabla_3 w),\\
    \partial_z dp_t^{\sigma} = - \alpha_\sigma \epsilon^2 K*[(\Upsilon^{1/2} \sigma_0 &d\tilde{W}_t \bcdot \nabla) w] + \alpha_\sigma \epsilon^2 \Delta (\Upsilon^{1/2} \sigma_0^z d\tilde{W}_t), \label{eq:weak-PE-dzP} \\
    \nabla_H \bcdot v + \partial_z w = 0, \text{ and } \: \nabla_H &\bcdot \sigma_0^H d\tilde{W}_t + \partial_{z} \sigma_0^z d\tilde{W}_t = 0, \text{ and } \: \rho = \rho_0(1 + \theta),\\
    d_t \theta + \Big((u^* dt + \Upsilon^{1/2} \sigma_0 &d\tilde{W}_t) \bcdot \nabla_3\Big) \theta - \frac{\Upsilon}{2} \nabla_3 \bcdot(a^0 \nabla_3 \theta) dt - \Delta \theta dt = 0,
\end{align}
where we denote by $a^K[\bcdot] := \sum_{k=1}^\infty \phi_k \mathcal{C}_K^* \mathcal{C}_K [\phi_k^\tr \bcdot]$, so that stands the term $\nabla_3 \bcdot (a^K \nabla_3 w)$ stands as a regularisation of $\nabla_3 \bcdot (a \nabla_3 w)$. If $K$ is regular enough -- e.g. $K \in H^3(\mathcal{S}, \R^3)$ -- the well-posedness of (PE$_{\alpha_\sigma}^\epsilon$) under suitable boundary conditions is ensured -- see \cite{DMM2025}. Moreover, we recover the primitive equations system (PE) by setting $K=0$ and disregarding the additive noise term in equation \eqref{eq:weak-PE-dzP}. This regularisation by filtering can be interpreted as cutting the highest frequencies of the noise term $(\sigma d\tilde{W}_t \bcdot \nabla) w$ in the vertical momentum equation. Do so typically increases the spatial length of this term, yet tempers the irregular dynamics arising from the higher third-order derivatives. Nevertheless, the deterministic-like non-linear terms remain unchanged by this regularisation. 

\paragraph{Scaled LU Navier-Stokes equations with regularised noise term (rSNS$_{\alpha_\sigma}^\epsilon$)}\hfill\\
\warning{It is natural that we consider a Navier-Stokes equation with the same filtered noise as in the primitive equations. Hence, we introduce a regularised version of the LU Navier-Stokes system for consistency,}
\begin{align}
    d_t v_\epsilon + \Big((u_\epsilon^* dt &+ \Upsilon^{1/2} \sigma d\tilde{W}_t) \bcdot \nabla_3\Big) v_\epsilon - \frac{\Upsilon}{2} \nabla_3 \bcdot(a \nabla_3 v_\epsilon) dt \nonumber\\
    &- \Delta_3 (v_\epsilon dt + \Upsilon^{1/2} \sigma^H d\tilde{W}_t) + \frac{1}{\rho_0} \nabla_H (p_\epsilon\:dt + dp_t^{\sigma, \epsilon}) = 0, \\
    \epsilon^2 \Bigg(d_t w_\epsilon &+ (u_\epsilon^* dt \bcdot \nabla_3) w_\epsilon + \alpha_\sigma \Upsilon^{1/2} K*[(\sigma d\tilde{W}_t \bcdot \nabla_3) w_\epsilon] - \frac{\alpha_\sigma^2 \Upsilon}{2} \nabla_3 \bcdot(a^K \nabla_3 w_\epsilon) dt \nonumber\\
    &- \Delta_3 (w_\epsilon dt + \alpha_\sigma \Upsilon^{1/2} \sigma^z d\tilde{W}_t)\Bigg) + \frac{1}{\rho_0} \partial_z (p_\epsilon\:dt + dp_t^{\sigma, \epsilon}) + \frac{\rho_\epsilon}{\rho_0} g = 0, \label{eq:SNS-w-reg}\\
    d_t \theta_\epsilon + \Big((u_\epsilon^* dt + &\Upsilon^{1/2} \sigma d\tilde{W}_t) \bcdot \nabla_3\Big) \theta_\epsilon - \frac{\Upsilon}{2} \nabla_3 \bcdot(a \nabla_3 \theta_\epsilon) dt - \Delta_3 \theta_\epsilon dt = 0,\\
    \nabla_H &\bcdot v_\epsilon + \partial_z w_\epsilon = 0, \qquad \text{ and } \qquad \rho_\epsilon = \rho_0(1+\theta_\epsilon).
\end{align}
Equivalently, this can be rewritten as,
\begin{align}
    d_t u_\epsilon + (u_\epsilon^* \bcdot \nabla_3) u_\epsilon dt &+ \Upsilon^{1/2} \mathcal{I}_K (\sigma d\tilde{W}_t \bcdot \nabla_3) I_{\alpha_\sigma} u_\epsilon - \frac{\Upsilon}{2} \nabla_3 \bcdot(\mathfrak{a}^K \nabla_3 I_{\alpha_\sigma^2} u_\epsilon) dt \nonumber\\
    - \Delta_3 (u_\epsilon dt &+ \Upsilon^{1/2} I_{\alpha_\sigma} \sigma d\tilde{W}_t) + \frac{1}{\rho_0} \nabla_\epsilon (p_\epsilon\:dt + dp_t^{\sigma, \epsilon}) + \frac{\rho_\epsilon}{\epsilon^2 \rho_0} g \bold{e}_z = 0, \\
    d_t \theta_\epsilon + \Big((u_\epsilon^* dt + &\Upsilon^{1/2} \sigma d\tilde{W}_t) \bcdot \nabla_3\Big) \theta_\epsilon - \frac{\Upsilon}{2} \nabla_3 \bcdot(a \nabla_3 \theta_\epsilon) dt - \Delta_3 \theta_\epsilon dt = 0,\\
    \nabla_H &\bcdot v_\epsilon + \partial_z w_\epsilon = 0, \qquad \text{ and } \qquad \rho_\epsilon = \rho_0(1+\theta_\epsilon),
\end{align}
where use the following notations: we denote the following operator by $\mathcal{I}_K = Diag(1, 1, K*)$ and, again, $I_{\alpha} = Diag(1, 1, \alpha)$, for all $\alpha > 0$. In addition, the operator $\mathfrak{a}^K$ is defined such that
\begin{equation}
    \forall (v,w) \in L^2(\mathcal{S},\R^3), \nabla_3 \bcdot \Big(\mathfrak{a}^K \nabla_3 \begin{pmatrix}
        v \\ w
    \end{pmatrix}\Big)
    = \begin{pmatrix}
        \nabla_3 \bcdot (a \nabla_3 v) \\ \nabla_3 \bcdot (a^K \nabla_3 w)
    \end{pmatrix}.
\end{equation}
Once again, we assume that the incompressibility condition \eqref{eq-NS-scaled-div-free-noise} holds. \warning{In Section \ref{sec-abstract-formulation}, we compare  (PE$_{\alpha_\sigma}^\epsilon$) with (rSNS$_{\alpha_\sigma}^\epsilon$), to show a convergence result that holds when $\alpha_\sigma \epsilon \rightarrow 0$ as $\epsilon \rightarrow 0$. Next, in Section \ref{sec-non-reg}, we compare (PE$_{\alpha_\sigma}^\epsilon$) with the non-regularised Navier-Stokes system (SNS$_{\alpha_\sigma}^\epsilon$) and obtain a much weaker result.}

From now, we denote the solutions to (rSNS$_{\alpha_\sigma}^\epsilon$) and (PE$_{\alpha_\sigma}^\epsilon$) as $(u_\epsilon, \theta_\epsilon)$ and $(u,\theta)$ respectively, or equivalently $(v_\epsilon, w_\epsilon, \theta_\epsilon)$ and $(v, w, \theta)$, respectively. In the rest of the paper, we drop the tilde notation for readability, as well as the $0$ index and exponent notation. More explicitly, we write $\sigma dW_t$, $a$ and $u_s$ in place of $\sigma_0 d\tilde{W}_t$, $a^0$ and $u_s^0$.

\subsection{Boundary conditions}

In this subsection, we introduce the initial and boundary conditions we use with rigid-lid boundary conditions. Notice that ``fully periodic'' boundary conditions will be considered in Subsection \ref{subsec:main-result-strong}. First, we decompose the boundary as $\partial S = \Gamma_u \cup \Gamma_b \cup \Gamma_l$ -- respectively the upper, bottom and lateral boundaries -- and equip this problem with the following free-slip rigid-lid type boundary conditions \cite{BS_2021,DGHT_2011, DMM2025}, 
\begin{align}
    \partial_z v  = 0, && w = 0, && \partial_z \theta + \alpha_\theta \theta =0, && \text{on } \Gamma_u, \label{boundary-conditions}\\
    \partial_z v = 0, && w = 0, && \partial_z \theta =0,  && \text{on } \Gamma_b, \nonumber \\
    \partial_{\bold{n}_H} v \times \bold{n}_H = 0, && v \bcdot \bold{n}_H = 0, && \partial_{\bold{n}_H} \theta =0, && \text{on } \Gamma_l. \nonumber
\end{align}
In the following, we consider an It\={o}-Stokes drift $u_s$ which fulfil \eqref{boundary-conditions}. Both problems (SNS$_{\alpha_\sigma}^\epsilon$) and (PE$_{\alpha_\sigma}^\epsilon$) -- and thus (PE) -- are equipped with these boundary conditions.

\section{Abstract formulation of the problem} \label{sec-abstract-formulation}

\subsection{Function spaces}\label{subsec-def-spaces}

In this subsection, we define the function spaces to be used below. Remind that the scaled spatial domain is denoted by $\mathcal{S} = \mathcal{S}_H \times [-1,1] \subset \R^3$. Let $d \geq 1$. We define the following function spaces on the domain $\Omega_S \in \{\mathcal{S}, \mathcal{S}_\epsilon, \mathcal{S}_H\}$. First, denote by $L^0(\Omega_S, \R^d)$ (respectively, $C_c^\infty(\Omega_S, \R^d)$) the set of measurable functions (respectively, smooth functions with compact support) from $\Omega_S$ to $\R^d$. In addition, for any real number $1 \leq q < \infty$, $1 \leq p \leq \infty$ and $s \geq 0$, and any integer $m \geq 0$, we denote by
\begin{gather*}
    L^q(\Omega_S,\R^d) = \bigg\{u \in L^0(\Omega_S, \R^d) \bigg| \int_{\Omega_S} \|u\|^q < +\infty\bigg\},\\
    L^\infty(\Omega_S,\R^d) = \bigg\{u \in L^0(\Omega_S, \R^d) \bigg| \esssup_{\Omega_S} \|u\| < +\infty\bigg\},\\
    W^{m,p}(\Omega_S,\R^d) = \bigg\{u \in L^p(\Omega_S,\R^d) \bigg| \forall \alpha \text{ such that } |\alpha|\leq m, D^\alpha u \in L^p(\Omega_S) \bigg\}.
\end{gather*}
In addition, we write $W^{m,2}(\Omega_S,\R^d) = H^m(\Omega_S,\R^d)$ for any non negative integer $m$, and we define the spaces $H^s(\Omega_S,\R^d)$ by interpolation, for any positive real number $s$. Furthermore, for any Banach space $\mathcal{B}$, and for $I$ an interval of $\R_+$, we denote by $C(I, \mathcal{B})$ (respectively $C_w(I, \mathcal{B}), L^0(I, \mathcal{B})$) the set of continuous (respectively weakly continuous, measurable) functions from $I$ to $\mathcal{B}$, and the function spaces
\begin{gather*}
    L^p(I, \mathcal{B}) = \Big\{ f \in L^0(I, \mathcal{B}) \: \Big| \: \int_I \|f\|_{\mathcal{B}}^p < \infty \Big\}, \\
    W^{s,p}(I, \mathcal{B}) = \Big\{ f \in L^p(I, \mathcal{B}) \: \Big| \: \int_I \int_I \frac{\|f(t) - f(t')\|^p_{\mathcal{B}}}{|t-t'|^{1+s p}} dt dt'  < \infty \Big\},
\end{gather*}
with
\begin{equation*}
    \|\bcdot\|_{L^p(I, \mathcal{B})} = \Big(\int_I \|\bcdot\|_{\mathcal{B}}^p dt\Big)^{1/p}\!\!\!\!\!\!, \quad \text{ and } \quad \|\bcdot\|_{W^{s,p}(I, \mathcal{B})} = \Big(\int_I \|\bcdot\|_{\mathcal{B}}^p dt + \int_I \int_I \frac{\|f(t) - f(t')\|^p_{\mathcal{B}}}{|t-t'|^{1+s p}} dt dt'\Big)^{1/p}\!\!\!\!\!\!,
\end{equation*}
their respective associated norms -- see \cite{FG_1995}.
For any interval $I \subset \R_+$, we also define $W^{s,p}_{loc}(I,\mathcal{B}) := \Big\{ f : I \rightarrow \mathcal{B} \: \Big| \: \forall T >0, f \in W^{s,p}([0,T], \mathcal{B}) \Big\}$.
In addition, for any Hilbert spaces $\mathcal{H}_1$ and $\mathcal{H}_2$, we define $\mathcal{L}_2( \mathcal{H}_1,\mathcal{H}_2)$ the space of Hilbert-Schmidt operators from $\mathcal{H}_1$ to $\mathcal{H}_2$, and $\|\bcdot\|_{\mathcal{L}_2( \mathcal{H}_1,\mathcal{H}_2)}$ its associated norm.

Now, let $\phi_k$ the eigenfunctions of the operator $\sigma_0$ as defined in Section \ref{sec:derivation-NS-PE}. Then denote by $\phi_k = (\phi_k^H \quad \phi_k^z)^\tr$, where $\phi_k^H$ are 2D vectors and $\phi_k^z$ are scalars. Then, we assume that the noise fulfils the non-penetration condition on the boundary, that is
\begin{equation}
    \phi_k \bcdot \bold{n} = 0 \text{ on $\partial \mathcal{S}$}. \label{eq-noiseBC}
\end{equation}
Furthermore, the eigenfunctions $\phi_k$ are assumed to be regular enough in the following sense,
\begin{gather}
    \sup_{t\in [0,T]} \sum_{k=0}^\infty \|\phi_k\|_{H^4(\mathcal{S}, \R^3)}^2 < \infty, \text{ which implies } \sup_{t\in[0,T]} \sup_k \|\phi_k\|_{L^\infty(\mathcal{S},\R^3)} < \infty, \label{smoothness-noise}
\end{gather}
and 
\begin{gather*}
    u_s \in L^\infty\Big([0,T], H^4(\mathcal{S}, \R^3)\Big), \: a \in H^1\Big([0,T],H^3(\mathcal{S},\R^{3 \times 3})\Big), \: d_t u_s \in L^\infty\Big([0,T],H^3(\mathcal{S}, \R^3)\Big).
\end{gather*}
In addition, notice that the aforementioned conditions imply that
\begin{equation*}
    a \nabla u_s \in L^\infty\Big([0,T],H^2(\mathcal{S},\R^3)\Big).
\end{equation*}
These regularity assumptions are not limiting in practice since most models consider spatially smooth noises for ocean models \cite{TML_2023}, as they are the physically observed ones.

Moreover, we introduce the spaces associated to the boundary conditions \eqref{boundary-conditions}. Define the following inner products,
\begin{gather*}
    (v,v^\sharp)_{H_1} = (v,v^\sharp)_{L^2(\mathcal{S},\R^2)}, \quad (v,v^\sharp)_{V_1} = (\nabla v, \nabla v^\sharp)_{L^2(\mathcal{S},\R^{2 \times 2})}\\
    (\theta,\theta^\sharp)_{H_2} = (\theta,\theta^\sharp)_{L^2(\mathcal{S},\R)}, \quad (\theta,\theta^\sharp)_{V_2} = (\nabla \theta,\nabla \theta^\sharp)_{L^2(\mathcal{S},\R^{2 \times 2})} + \alpha_\theta (\theta,\theta^\sharp)_{L^2(\Gamma_u, \R)},
\end{gather*}
and let
\begin{align}
    (F,F^\sharp)_H = (v,v^\sharp)_{H_1} + (\theta,\theta^\sharp)_{H_2}, \quad (F,F^\sharp)_V = (v,v^\sharp)_{V_1} + (\theta,\theta^\sharp)_{V_2},
\end{align}
for all $F,F^\sharp \in L^2(\mathcal{S},\R^3)$, such that $F = (v,\theta)^\tr$ and $F^\sharp = (v^\sharp, \theta^\sharp)^\tr$. Also, we denote by $\|.\|_H, \|.\|_{H_i}$ and $\|.\|_V,\|.\|_{V_i}$ the associated norms. With a slight abuse of notation, we may write $\|.\|_H, \|.\|_V$ in place of $\|.\|_{H_i}, \|.\|_{V_i}$, respectively. Then, denote by $\mathcal{V}_1$ the space of functions of $C^\infty(\mathcal{S},\R^2)$, such that for all $v \in \mathcal{V}_1$, $\nabla_H \bcdot \int_{-1}^1 v dz = 0$ on $\mathcal{S}_H$ and $v \bcdot \bold{n} = 0$ on $\Gamma_l$. In addition, define $\mathcal{V}_2$ the space of functions of $C^\infty(\mathcal{S},\R)$ that average to zero over $\mathcal{S}$. Denote by $H_i$ the closure of $\mathcal{V}_i$ for the norm $\|.\|_{H_i}$, and $V_i$ its closure by $\|.\|_{V_i}$. Eventually, define $H=H_1 \times H_2$ and $V=V_1 \times V_2 $, which are also the closures of $\mathcal{V}_1 \times \mathcal{V}_2$ by $\|.\|_{H}$ and $\|.\|_{V}$, respectively. Often, by abuse of notation, we write $(\bcdot,\bcdot)_H$ instead of $(\bcdot,\bcdot)_{V' \times V}$. More generally, if $K$ is a subspace of $H$ and $K'$ its dual space, we write $(\bcdot,\bcdot)_H$ instead of $(\bcdot,\bcdot)_{K' \times K}$. In the following, $H_1$ and $V_1$ are interpreted as horizontal velocity spaces ($\R^2$-valued processes). Using this formalism, the vertical velocity $w$ is written as a functional of the horizontal velocity $v$ through the continuity equation, namely $w(v) = \int_z^1 \nabla_H \cdot v \: dz'$. In addition, we interpret $H_2$ and $V_2$ as tracer spaces. Moreover, we define $\mathcal{D}(A) = V \cap H^2(\mathcal{S},\R^3)$, where $A$ is defined as the operator $\mathbf{P}(-\Delta_3)$. Here, $\mathbf{P}$ is a projector defined in the next subsection. As such, $A : \mathcal{D}(A) \rightarrow H$ is an unbounded operator.

\subsection{Modified Leray projectors}

In order to disregard the pressure terms in the scaled momentum Navier-Stokes and primitive equations (SNS), we introduce modified Leray projectors. We define first the barotropic and baroclinic projectors 
\begin{gather}
    \mathcal{A}_2 :  L^2({\mathcal S},   \R^2)  \to  L^2({\mathcal S}_H,   \R^2), \quad \text{ and } \quad \mathcal{A},\mathcal{R} : L^2({\mathcal S}, \R^2)  \to  L^2({\mathcal S}, \R^2),
\end{gather}
of the velocity component as follows -- see \cite{BS_2021}. For $v \in H$, let
\begin{equation}
    \mathcal{A}_2 [v](x,y) = \frac{1}{2} \int_{-1}^1 v(x,y,z') dz', \quad \mathcal{A} [v](x,y,z) = \mathcal{A}_2 [v](x,y), \quad \mathcal{R} [v] = v - \mathcal{A} [v].
\end{equation}
Remark that $\mathcal{A}$ and $\mathcal{R}$ are orthogonal projectors with respect to the inner product $(\bcdot, \bcdot )_{H}$. To simplify notations, we will use $\bar{v}$ in place of $\mathcal{A}_2 [v]$ or $\mathcal{A} [v]$, and $\tilde{v}$ in place of $\mathcal{R} [v]$. We are now in position to define the projector $\mathbf{P}$, which acts on bi-dimensional vector fields on ${\mathcal S}$ by combining the classical 2D  Leray projection of the barotropic component with the (divergent) baroclinic component of the horizontal velocity,
\begin{equation}
    \mathbf{P}= \mathbf{P}^{2D} \mathcal{A} + \mathcal{R},
\end{equation}
where $\mathbf{P}^{2D}$ denotes the standard 2D Leray projector associated to the boundary condition $\bar{v} \bcdot \bold{n} = 0$ on $\partial \mathcal{S}_H$.

Furthermore, we introduce the following inner product, for all $u,u^\sharp \in L^2(\mathcal{S}, \R^3)$,
\begin{equation}
    (u, u^\sharp)_{L_{\epsilon}^2} = (v, v^\sharp)_{H_1} + \epsilon^2 (w, w^\sharp)_{L^2(\mathcal{S},\R)},
\end{equation}
and we denote by $\|\bcdot\|_{L_\epsilon^2}$ its associated norm. Similarly, for $v,v^\sharp \in V_1$, we define 
\begin{gather}
    (u, u^\sharp)_{V_\epsilon} = (v, v^\sharp)_{V_1} + \epsilon^2 (w, w^\sharp)_{H^1(\mathcal{S},\R)}, 
\end{gather}
and $\|\bcdot\|_{V_\epsilon}$ the associated norm. Moreover, we define the space
\begin{equation}
    \mathcal{F} = \Big\{(v, \int_z^1 \nabla_H \bcdot v dz')^\tr \: \Big| \: v \in H_1 \Big\}.
\end{equation}
Thus, we can introduce $\mathbf{P}_{\epsilon} : L^2(\mathcal{S}, \R^3) \rightarrow \mathcal{F}$, the orthogonal projection onto $\mathcal{F}$ with respect to the inner product $(\bcdot, \bcdot)_{L_\epsilon^2}$. In particular, all the terms of the form $\nabla_\epsilon f$ are cancelled by $\mathbf{P}_{\epsilon}$.

In this context, $\mathbf{P}$ and $\mathbf{P}_\epsilon$ can be interpreted as ``modified'' Leray projectors for (PE$_{\alpha_\sigma}^\epsilon$) and (SNS$_{\alpha_\sigma}^\epsilon$) respectively, since they cancel pressure gradient terms (i.e. the horizontal pressure gradient and the full rescaled pressure gradient, respectively). The main difference however is that $\mathbf{P}$ keeps the baroclinic pressure terms unchanged in (PE$_{\alpha_\sigma}^\epsilon$) -- see \cite{BS_2021}. The use of such projector is justified by a crucial remark mentioned in \cite{CT_2007}: in the (deterministic) primitive equations, the pressure naturally decomposes into barotropic and baroclinic components. The baroclinic mode is explicitly determined via the hydrostatic balance, while the barotropic mode influences only the barotropic dynamics through a vertically averaged pressure. Consequently, the barotropic dynamics resemble those of a two-dimensional Navier-Stokes system, whereas the baroclinic component follows a structure akin to a forced Burgers equation, with the two modes coupled through nonlinear interactions. In this context, the projector $\mathbf{P}$ acts as a horizontal Leray projector on the barotropic component, ensuring incompressibility in the vertically averaged, large-scale flow, while it reduces to the identity on the baroclinic dynamics.

\subsection{Main results}

\subsubsection{Convergence estimate with rigid-lid boundary conditions}

In the following, we assume that the noise is bidimensional in the following sense: we suppose that $\sigma^H dW_t$ is divergence-free and independent of the $z$-coordinate, and that
\begin{equation}
    \forall (x,y,z) \in \mathcal{S}, \quad \sigma dW_t (x,y,z) = (\sigma^H dW_t (x,y) \quad 0)^\tr, \label{eq:hyp-2D-noise}
\end{equation}
This is in the line of \cite{DMM2025}, and very connected to the $z$-independence assumption for global well-posedness given in \cite{AHHS_2022}. Furthermore, without loss of generality, we set the noise scaling factor to be $\Upsilon = 1$. Applying the projector $\mathbf{P}_\epsilon$ to the momentum equation of (rSNS$_{\alpha_\sigma}^\epsilon$) yields the following problem, which is equipped with the boundary conditions \eqref{boundary-conditions}, and that we denote by $(\mathcal{P}_\epsilon^{rSNS})$,
\begin{align}
    d_t u_\epsilon + \mathbf{P}_\epsilon (u_\epsilon^* dt \bcdot \nabla_3) u_\epsilon &+ \mathbf{P}_\epsilon \mathcal{I}_K (\sigma dW_t \bcdot \nabla_3) I_{\alpha_\sigma} u_\epsilon - \mathbf{P}_\epsilon \frac{1}{2} \nabla_3 \bcdot(\mathfrak{a}^K \nabla_3 I_{\alpha_\sigma^2} u_\epsilon) dt \nonumber\\
    &- \mathbf{P}_\epsilon \Delta_3 (u_\epsilon dt + I_{\alpha_\sigma} \sigma dW_t) +  \mathbf{P}_\epsilon \frac{\rho_\epsilon}{\epsilon^2 \rho_0} g \bold{e}_z dt = 0, \\
    d_t \theta_\epsilon + \Big((u_\epsilon^* dt + & \sigma dW_t) \bcdot \nabla_3\Big) \theta_\epsilon - \frac{1}{2} \nabla_3 \bcdot(a \nabla_3 \theta_\epsilon) dt - \Delta_3 \theta_\epsilon dt = 0,
\end{align}
with $\rho_\epsilon = \rho_0(1+\theta_\epsilon)$. Remind that we have dropped the tilde notation, as well as the $0$ index and exponent notation. Thus, in this equation, $\sigma dW_t$, $a$ and $u_s$ refer to what was previously denoted $\sigma_0 d\tilde{W}_t$, $a^0$ and $u_s^0$.

Applying similarly $\mathbf{P}$ to the horizontal momentum equation of (PE$_{\alpha_\sigma}^\epsilon$) yields the problem $(\mathcal{P}_\epsilon^{PE})$, which is equipped with the same boundary conditions \eqref{boundary-conditions},
\begin{align}
    d_t v + \mathbf{P}\Big((u^* dt + \sigma dW_t) &\bcdot \nabla_3\Big) v -  \mathbf{P} \Delta_3 (v dt + \sigma^H dW_t) \nonumber \\
    -  \mathbf{P} \nabla_H &\int_z^1 \partial_z (p\:dt + dp_t^{\sigma}) dz' - \frac{1}{2}  \mathbf{P} \nabla_3 \bcdot(a \nabla_3 v) dt= 0, \label{eq:abstract-PE-v}\\
    \partial_z p = -\frac{\rho}{\rho_0} &g + \alpha_\sigma^2 \epsilon^2 \frac{1}{2} \nabla_3 \bcdot (a^K \nabla_3 w),\\
    \partial_z dp_t^{\sigma} = - \alpha_\sigma \epsilon^2 K*[(\sigma &dW_t \bcdot \nabla) w] +  \alpha_\sigma \epsilon^2 \Delta (\sigma^z dW_t), \\
    d_t \theta + \Big((u^* dt + \sigma &dW_t) \bcdot \nabla_3\Big) \theta - \frac{1}{2} \nabla_3 \bcdot(a \nabla_3 \theta) dt - \Delta \theta dt = 0,
\end{align}
with $ \rho = \rho_0(1 + \theta)$. 

\bigskip

\noindent
Now, we introduce the following definition of global martingale Leray-Hopf weak solutions,
\begin{definition}[Global martingale Leray-Hopf weak solution] \label{def-GlobMartWeakSol}
    Let $(v_0,\theta_0) \in V$ and $\mathcal{B}$ be a stochastic basis. If $(v,\theta)$ is a pair of velocity and tracer stochastic fields defined on $\mathcal{B}$ such that $(v,\theta)_{t=0} = (v_0,\theta_0)$, $(\mathcal{B},v,\theta)$ is said to be a \emph{global martingale Leray-Hopf weak solution} of $(\mathcal{P}_\epsilon^{rSNS})$ with initial data $(v_0,\theta_0)$ if the propositions hereafter hold. In the following, we denote by $w = \int_z^1 \nabla_H \bcdot v$, $u = (v,w)^\tr$, $\rho = \rho_0(1+\theta)$.
    \begin{enumerate}
        \item for all $T > 0$, $(v,\theta) \in C_w([0,T], H) \cap L^2([0,T], V)$ almost surely in $\mathcal{B}$,
        \item for all $T>0$, for all integer $p \geq 1$, for all stopping times $0 < \eta < \zeta < T$, $(v,w,\theta)$ fulfils the following energy estimates almost surely in $\mathcal{B}$,
        \begin{multline}
            \frac{1}{p}\|u\|_{L_\epsilon^2}^{2p}(\zeta) + \int_\eta^{\zeta} \|u\|_{L_\epsilon^2}^{2p-2} \| \nabla_3 u\|_{L_\epsilon^2}^2 dt + \frac{1}{2} \int_\eta^{\zeta} \|u\|_{L_\epsilon^2}^{2p-2} (\mathfrak{a}^K \nabla I_{\alpha_\sigma} u,\nabla I_{\alpha_\sigma} u)_{L_\epsilon^2} dt \\
            \leq \frac{1}{p} \|u\|_{L_\epsilon^2}^{2p} (\eta) + \frac{1}{2} \int_\eta^{\zeta} \sum_k \|u\|_{L_\epsilon^2}^{2p-2} \|\mathbf{P}_\epsilon \mathcal{I}_K (\phi_k \bcdot \nabla I_{\alpha_\sigma} u) + \Delta I_{\alpha_\sigma} \phi_k\|^2_{L_\epsilon^2} dt\\
            + \int_\eta^{\zeta} \|u\|_{L_\epsilon^2}^{2p-2} \big(\Delta I_{\alpha_\sigma} \sigma dW_t, u\big)_{L_\epsilon^2} - \int_\eta^{\zeta} \|u\|_{L_\epsilon^2}^{2p-2} (\frac{\rho_\epsilon}{\rho_0} g, w)_{L^2} dt, \label{eq-defWeakSol-energ}
        \end{multline}
        \item for all $T > 0$, for all test functions $(u^\sharp, \theta^\sharp) \in C_c^\infty(\mathcal{S},\R^4)$ such that $\nabla_3 \bcdot u^\sharp = 0$, let $w^\sharp = u^\sharp \bcdot \bold{e}_z$. Then, the following relations hold almost surely in $\mathcal{B}$, 
        \begin{multline}
            \int_\mathcal{S} u^\sharp \bcdot (u(T) - u(0)) + \Bigg[\int_0^T - u \bcdot (u^* dt \bcdot \nabla_3) u^\sharp - I_{\alpha_\sigma} u \bcdot \mathcal{I}_K(\sigma dW_t \bcdot \nabla_3) u^\sharp\\
            -  \frac{1}{2} I_{\alpha_\sigma^2} u \bcdot \nabla_3 \bcdot(\mathfrak{a}^K \nabla_3 u^\sharp) dt - (u dt + I_{\alpha_\sigma} \sigma dW_t) \Delta_3 u^\sharp + w^\sharp \frac{g \rho_\epsilon}{\epsilon^2 \rho_0} dt\Bigg] dxdydz = 0
        \end{multline}
        and 
        \begin{multline}
            \int_\mathcal{S} \theta^\sharp \bcdot (\theta(T) - \theta(0)) + \Bigg[ \int_0^T - \theta \Big((u^* dt + \sigma dW_t) \bcdot \nabla_3\Big) \theta^\sharp \\
            - \frac{1}{2} \theta \nabla_3 \bcdot(a \nabla_3 \theta^\sharp) dt - \theta \Delta_3 \theta^\sharp dt \Bigg] dxdydz = 0.
        \end{multline}
    \end{enumerate}
\end{definition}
Before stating our main results, we introduce notations to be used in their statements and proofs. First, we gather some useful facts. For each $\epsilon>0$, adapting the proof of \cite{DHM_2023}, we infer that $(\mathcal{P}_\epsilon^{rSNS})$ admits at least one global martingale Leray-Hopf weak solution $(v_\epsilon, \theta_\epsilon)$, defined on an $\epsilon$-dependent stochastic basis which is not known \emph{a priori}. 
\begin{proposition} \label{thm:SNS-reg-weak}
    The problem $(\mathcal{P}_\epsilon^{rSNS})$, equipped with the initial condition $(u_0, \theta_0) \in H$, admits a least one global-in-time martingale solution.
\end{proposition}
The proof is not given in details, since the reasoning is a straightforward adaptation of the work performed for the ``direct'' LU Navier-Stokes equations in \cite{DHM_2023}. In addition, since it was shown in \cite{DMM2025} that $(\mathcal{P}_\epsilon^{PE})$ admits a unique global pathwise solution, we denote by $(v, \theta)$ such a solution on the aforementioned stochastic basis. As such, $(v, \theta)$ and $(v_\epsilon, \theta_\epsilon)$ are both defined on the same $\epsilon$-dependent stochastic basis. Yet, for readability, the $\epsilon$-dependence of $(v, \theta)$ does not appear explicitly in the notation. Furthermore, from \cite{DMM2025} again, almost surely, for all $T>0$,
\begin{equation*}
    \sup_{[0,T]} \frac{1}{2}\|(v,\theta)\|_{V}^2 + \int_0^{T} \|(v,\theta)\|_{\mathcal{D}(A)}^2 dt < \infty.
\end{equation*}
Reasoning similarly, we may prove that, almost surely,
\begin{equation*}
     \sup_{[0,T]} \frac{1}{4}\|(v,\theta)\|_{V}^4 + \int_0^{T} \|(v,\theta)\|_{V}^2 \|(v,\theta)\|_{\mathcal{D}(A)}^2 dt < \infty.
\end{equation*}
Additionally, denote by $u_\epsilon = (v_\epsilon, w_\epsilon)^\tr$ and $u = (v, w)^\tr$ the tridimensional velocity vectors, where we use the notations $w_\epsilon = w(v_\epsilon)$ and $w = w(v)$. Moreover, define $\Theta_\epsilon := \theta_\epsilon - \theta$ and $U_\epsilon = (V_\epsilon, W_\epsilon)^\tr$, where $V_\epsilon := v_\epsilon - v$ and $W_\epsilon := w_\epsilon - w$. Then, for two stopping times $\eta,\zeta$ such that $\eta<\zeta$, we define
    \begin{gather}
        \mathcal{E}_\epsilon^0(\eta,\zeta) : = \sup_{[\eta,\zeta]} \frac{1}{2}\|U_\epsilon\|_{L_\epsilon^2}^2 + \int_\eta^\zeta \|\nabla_3 U_\epsilon\|_{L_\epsilon^2}^2 dt + \sup_{[\eta,\zeta]} \frac{1}{2}\|\Theta_\epsilon\|_{L^2}^2 + \int_\eta^\zeta \|\nabla_3 \Theta_\epsilon\|_{L^2}^2 dt,\\
        \mathcal{E}^{1,4}(\eta,\zeta) : = \sup_{[\eta,\zeta]} \frac{1}{2}\|(v,\theta)\|_{V}^4 + \int_\eta^\zeta \|(v,\theta)\|_{V}^2 \|(v,\theta)\|_{\mathcal{D}(A)}^2 dt.
    \end{gather}
Then, the following theorem holds.
\begin{theorem}\label{theorem-cvg-weak-VREG}
    Assume that $K \in H^3$, and that the noise term $\sigma dW_t$ in the primitive equations is bidimensional and fulfils the regularity conditions \eqref{smoothness-noise}. Let $\epsilon >0$, and equip $(\mathcal{P}_\epsilon^{rSNS})$ and $(\mathcal{P}_\epsilon^{PE})$ with the same initial condition $(v_0, \theta_0)^\tr \in V$, so that $U_\epsilon\vert_{t=0} = \Theta_\epsilon\vert_{t=0} = 0$. Then, for all $T >0$, and $\delta > 0$, there exists a constant $C_\delta >0$ such that,
    \begin{equation}
        \E\Big[\mathcal{E}_\epsilon^0(0,T \wedge \tau_\delta) \Big] \leq C_\delta \epsilon^2  \E\Bigg[(1+\alpha_\sigma^2) + \int_0^{T \wedge \tau_\delta} \| v \|_{H^1}^4 + (1 + \|v\|_{H^1}^{2}) \|v\|_{H^2}^{2} dt \Bigg] \label{eq-thm-cvg-weak-VREG}
    \end{equation}
    where $\tau_{\delta} := \inf \Big\{t>0 \Big|\mathcal{E}^{1,4}(0,t) \geq \delta \Big\}$, so that $\tau_\delta \rightarrow \infty$ in probability as $\delta \rightarrow \infty$. In particular, $\mathcal{E}_\epsilon^0(0,T) \rightarrow 0$ in probability whenever $\alpha_\sigma = o(\epsilon^{-1})$.
\end{theorem}
\begin{remark} \label{remark:weak-thm} \hfill
    \begin{enumerate}
        \item The solutions to (rSNS$_\epsilon^{\alpha_\sigma}$) mentioned in Theorem \ref{theorem-cvg-weak} are not unique, and are defined on abstract stochastic bases, which are \emph{a priori} $\epsilon$-dependent. However, the theorem indicates that, whatever base we may choose, we can control the second moment of the error up to $\tau_\delta \wedge T$, for all $T,\delta>0$, namely in the space $L^2([0,T], V) \cap  L^\infty([0,T], H)$.
        \item Regarding the regularity of $K$ and the proofs given in \cite{DMM2025}, one can prove that the solutions to the problem with ``$K$-regularised weak hydrostatic hypothesis'' are continuous with respect to $K$ -- although it was not done in the original paper. In particular, this shows that the solutions to the (regularised) weak hydrostatic problem converge to the strong hydrostatic one when $K$ vanishes, the latter being equivalent to setting $K=0$. 
        \item If the scaled noise was $\epsilon$-dependent, then the aforementioned convergence result would be similar, yet subject to the following conditions: denote by $(\phi_k^\epsilon)_k$ the eigenfunctions of the scaled Navier-Stokes noise $\sigma_\epsilon dW_t$, and by $a^\epsilon := \sum_k \phi_k^\epsilon (\phi_k^\epsilon)^\tr$, $\Phi_k^\epsilon := \phi_k^\epsilon - \phi_k$, so that, whenever
        \begin{gather}
            \sum_k \|\Phi_k^\epsilon\|_{L^2([0,T], H^2)}^2 + \|a^\epsilon - a\|_{L^\infty([0,T], H^1)}^2= O(\alpha_\sigma^2 \epsilon^2), \text{ and } \sum_k \|\Phi_k^\epsilon\|_{L^\infty([0,T], L^\infty)}^2 = O(\epsilon^2), \label{eq-condition-thm-weak}
        \end{gather}
        then $\E\big[\mathcal{E}_\epsilon^0(0,T \wedge \tau_\delta) \big]$ converges to $0$ at order $\alpha_\sigma^2 \epsilon^2$ provided that $\alpha_\sigma = o(\epsilon^{-1})$. The condition \eqref{eq-condition-thm-weak} is fulfilled in particular when choosing $\phi_k^\epsilon = \phi_k$, i.e. $\Phi_k^\epsilon = 0$.
    \end{enumerate}
\end{remark}

\subsubsection{Improved estimates in "fully periodic" boundary conditions} \label{subsec:main-result-strong}

Assume now that the domains $\mathcal{S}_\epsilon$ and $\mathcal{S}$ are defined as \emph{"fully periodic"} domains, namely $\mathcal{S}_\epsilon = \mathbbm{T}_2 \times \epsilon \mathbbm{T}_1$ and $\mathcal{S} = \mathbbm{T}_3$, see \cite{LT2019}. Then, define the abstract problems ($\Pi_\epsilon^{rSNS}$) and ($\Pi_\epsilon^{PE}$) as the counterparts of ($\mathcal{P}_\epsilon^{rSNS}$) and ($\mathcal{P}_\epsilon^{PE}$) with tridimensionally periodic boundary conditions, that is
\begin{gather}
    \text{$v$ and $\theta$ are periodic with respect to $x,y,z$.}
\end{gather}
In such case, we may redefine $H_1$ and $V_1$ as the closures (by the same norms) of the space of horizontally periodic functions $v$ of regularity $C^\infty(\mathcal{S},\R^2)$ such that $\nabla_H \bcdot \int_{-1}^1 v dz = 0$ on $\mathcal{S}_H$. Similarly, we redefine $H_2$ and $V_2$ as the closures of the space of horizontally periodic functions of regularity $C^\infty(\mathcal{S},\R)$ that average to zero. Then we redefine $H$, $V$ and $\mathcal{D}(A)$ accordingly, and the projectors $\mathbf{P}$ and $\mathbf{P}_\epsilon$ as well. Furthermore, the noise boundary condition \eqref{eq-noiseBC} is replaced by a periodicity assumption on the functions $\phi_k$.

Furthermore, the tridimensional Navier-Stokes equations with transport noise have been shown well-posed in \cite{AV2024-NS-tranport-noise}, under weak regularity assumptions on the initial condition -- typically in Besov spaces. Adapting this proof of local well-posedness to the regularised LU interpretation of the 3D Navier-Stokes equations with initial condition ($u_0,\theta_0) \in H^1(\mathcal{S},\R^{3+1})$, $u_0$ being divergence-free, does not pose any major difficulty.

\begin{proposition}\label{thm:SNS-reg-strong}
    The problem $(\Pi_\epsilon^{rSNS})$, equipped with the initial condition $(u_0, \theta_0) \in V$, admits a unique maximal pathwise solution, which is a priori only local-in-time.
\end{proposition}

\noindent
Therefore, we denote by $(v_\epsilon, \theta_\epsilon)$ the maximal solution to $(\Pi_\epsilon^{rSNS})$, and $\tau_\epsilon^{rSNS}$ its associated stopping time. In this setting, the following theorem holds,
\begin{theorem}\label{theorem-cvg-strong-VREG}
    Assume that the assumptions of Theorem \ref{theorem-cvg-weak-VREG} hold. Let $(v_\epsilon, \theta_\epsilon)$ and $(v, \theta)$ be, respectively, the solutions to the problems $(\Pi_\epsilon^{rSNS})$ and $(\Pi_\epsilon^{PE})$ with the same initial condition $(v_0, \theta_0) \in \mathcal{D}(A)$, and on the same stochastic basis. Let $\tau_\epsilon^{rSNS}$ be the stopping time associated to $(\Pi_\epsilon^{rSNS})$, and define, for all stopping times $\eta,\zeta$ such that $\eta<\zeta$,
    \begin{equation}
    \mathcal{E}_\epsilon^1(\eta,\zeta) := \left\{
    \begin{array}{ll}
        \sup_{[\eta, \zeta]} \frac{1}{2}\Big(\|U_\epsilon\|_{V_\epsilon^2}^2+\|\Theta_\epsilon\|_{V}^2\Big) + \int_{\eta}^{\zeta} \Big(\|U_\epsilon\|_{D_\epsilon(A)}^2 + \|\Theta_{\epsilon}\|_{\mathcal{D}(A)}^2\Big) dt \text{ if $\zeta < \tau_\epsilon^{rSNS}$},\\
        \infty \text{ otherwise}.
    \end{array}
    \right.
    \end{equation}
    Then, using same notations, Theorem \ref{theorem-cvg-weak-VREG} holds, and $v \in L^2(\Omega, L^2([0,T],H^3(\mathcal{S},\R^2)))$. Moreover, for all $T > 0$, there exists a constant $C_1 > 0$ and a stopping time
    \begin{multline}
        \check{\tau}_\epsilon = \inf \Big\{t \in \R_+ \Big| \sup_{[0, t]} (\|U_\epsilon\|_{V_\epsilon}^2 + \|\Theta_\epsilon\|_{V}^2) + \int_0^t(\| U_\epsilon\|_{\mathcal{D}_\epsilon(A)}^2 + \|\Theta_\epsilon\|_{\mathcal{D}(A)}^2) dt > \frac{1}{4C_1} \Big\} \leq \tau_\epsilon^{rSNS}.
    \end{multline}
    such that, for all $\delta > 0$,
    \begin{equation}
        \E \big[\mathcal{E}_\epsilon^1(0, T \wedge \check{\tau}_\epsilon \wedge \tau_\delta)\big]
        \leq C_1\epsilon^2 (1+ \alpha_\sigma^2)\E\int_{0}^{\check{\tau}_{\epsilon} \wedge \tau_\delta \wedge T} (1+\|v\|_{H^3}^2) dt. \label{eq-bound-thm2-VREG}
    \end{equation}
    Furthermore, $\check{\tau}_\epsilon, \tau_\epsilon^{rSNS} \rightarrow \infty$ in probability. Consequently, $\E \big[\mathcal{E}_\epsilon^1(0, T \wedge \check{\tau}_\epsilon \wedge \tau_\delta)\big]$ converges to $0$ whenever $\alpha_\sigma = o(\epsilon^{-1})$, and in such case $\mathcal{E}_\epsilon^1(0, T) \rightarrow 0$ in probability as well.
\end{theorem}
\warning{
Theorems \ref{theorem-cvg-weak-VREG} and \ref{theorem-cvg-strong-VREG} show that the asymptotic error $\E \big[\mathcal{E}_\epsilon^1(0, T \wedge \check{\tau}_\epsilon \wedge \tau_\delta)\big]$ between the regularised Navier-Stokes equations (rSNS$_{\alpha_\sigma}^\epsilon$) and the weakly hydrostatic primitive equations (PE$_{\alpha_\sigma}^\epsilon$) is typically bounded by a term of order $O(\alpha_\sigma^2\epsilon^2)$. As we shall see below, the error between Navier-Stokes and the strongly hydrostatic primitive equations is of order $\alpha_\sigma^4 \epsilon^2$. Hence, the weak hydrostatic assumption does provide a better approximation. 

Thus, we have identified a ``grey zone'' where the weakly hydrostatic primitive equations (PE$_{\alpha_\sigma}^\epsilon$) constitute a suitable approximation of the regularised Navier-Stokes system (rSNS$_{\alpha_\sigma}^\epsilon$), while the strongly hydrostatic primitive equations (PE$_{0}^0$) do not \emph{a priori}. Namely, this regime corresponds to $\alpha_\sigma$ being typically of order $\epsilon^{-\gamma}$, where $\gamma \in [1/2, 1)$. Yet, finding an appropriate approximation of (rSNS$_{\alpha_\sigma}^\epsilon$) in the regime $\alpha_\sigma \sim \epsilon^{-\gamma}$ with $\gamma \geq 1$ remains an open question.
\begin{remark} \hfill\\
    Using the notations of Remark \ref{remark:weak-thm}.3, the LHS of \eqref{eq-bound-thm2-VREG} converges to $0$ under the assumption \eqref{eq-condition-thm-weak} and the additional one $\sum_k \|\Phi_k^\epsilon\|_{H^4}^2 \rightarrow 0$, provided that $\alpha_\sigma = o(\epsilon^{-1})$.
\end{remark}
}

\section{Comparison to the (non-regularised) LU Navier-Stokes equations} \label{sec-non-reg}

In this section, we consider the (non-regularised) LU Navier-Stokes equations, to study their convergence towards the LU primitive equations. Although seeking the convergence of the ``direct'' system seems natural, ultimately this setting does not allow to distinguish orders of convergence of the weak and strong hydrostatic assumptions with regard to the coefficient $\alpha_\sigma$, as $\epsilon$ tends to $0$. This difference comes from the presence of an extra term in the estimation, which is due  to the regularisation error between the two versions of the LU Navier-Stokes equations. For this reason, this setting is slightly more intricate. Hence, the proofs presented in Sections \ref{sec-thm-cvg-weak} and \ref{sec-thm-cvg-strong} are given in the setting presented here (see Theorems \ref{thm:SNS-weak} and \ref{thm:SNS-strong}), although Theorems \ref{thm:SNS-reg-weak} and \ref{thm:SNS-reg-strong} can be proven similarly. Notice that, since $K=0$ allows to recover the strong hydrostatic assumption, the results of this section provide a comparison between (PE$_{\alpha_\sigma}^\epsilon$) and (SNS$_{\alpha_\sigma}^\epsilon$).

\subsection{Abstract formulation of the (non-regularised) LU Navier-Stokes equations}

Let us derive first the abstract formulation of the scaled LU Navier-Stokes equations without regularistion, in the fashion of what was presented in Section \ref{sec:derivation-NS-PE}.

From now, for simplicity and without loss of generality, we assume that $\Upsilon = 1$. In the following, we apply the ``rigid-lid boundary conditions'' projector $\mathbf{P}_\epsilon$ to (SNS$_{\alpha_\sigma}^\epsilon$) -- see equations \eqref{eq:SNS-noreg-u} to \eqref{eq:SNS-noreg-div-density}. Doing so, we infer the expression of the following abstract problem $(\mathcal{P}_\epsilon^{SNS})$,
\begin{align}
    d_t u_\epsilon + \mathbf{P}_\epsilon (u_\epsilon^* dt \bcdot \nabla_3) u_\epsilon &+ \mathbf{P}_\epsilon (\sigma dW_t \bcdot \nabla_3) I_{\alpha_\sigma} u_\epsilon - \mathbf{P}_\epsilon \frac{1}{2} \nabla_3 \bcdot(a \nabla_3 I_{\alpha_\sigma^2} u_\epsilon) dt \nonumber\\
    &- \mathbf{P}_\epsilon \Delta_3 (u_\epsilon dt + I_{\alpha_\sigma} \sigma dW_t) +  \mathbf{P}_\epsilon \frac{\rho_\epsilon}{\epsilon^2 \rho_0} g \bold{e}_z dt = 0, \\
    d_t \theta_\epsilon + \Big((u_\epsilon^* dt + & \sigma dW_t) \bcdot \nabla_3\Big) \theta_\epsilon - \frac{1}{2} \nabla_3 \bcdot(a \nabla_3 \theta_\epsilon) dt - \Delta_3 \theta_\epsilon dt = 0,
\end{align}
with $\rho_\epsilon = \rho_0(1+\theta_\epsilon)$. Similarly, we define the problem $(\Pi_\epsilon^{SNS})$ and by using ``periodic boundary conditions'' projectors, as previously mentioned in Section \ref{sec-abstract-formulation}.

\subsection{Main results}

We state hereafter our main results. Essentially, the two first theorems adapt classical results from the (deterministic) Leray theory to the our stochastic setting. Notice that Definition \ref{def-GlobMartWeakSol} can be adapted to define martingale solutions of the problem $(\mathcal{P}_\epsilon^{SNS})$.
\begin{proposition} \label{thm:SNS-weak}
    The problem $(\mathcal{P}_\epsilon^{SNS})$, equipped with the initial condition $(u_0, \theta_0) \in H$, admits a least one global-in-time martingale solution.
\end{proposition}
 The proof of Proposition \ref{thm:SNS-weak} comes from the work performed for the ``direct'' LU Navier-Stokes equations in \cite{DHM_2023}. Hence, global martingale solutions exist whenever the initial condition $(u_0, \theta_0)$ is in $H$. To investigate the convergence such a solution for a vanishing aspect ratio, we denote it by $(\hat{u}_\epsilon^{SNS}, \hat{\theta}_\epsilon^{SNS})$, and introduce the quantity $\hat{\mathcal{E}}_\epsilon^0$ in the following. It is the counterpart of $\mathcal{E}_\epsilon^0$ that is adapted to the problem $(\mathcal{P}_\epsilon^{SNS})$. Let $T>0$, then, for all stopping times $\eta, \zeta$ such that $0<\eta<\zeta<T$, we define
\begin{gather}
    \hat{\mathcal{E}}_\epsilon^0(\eta,\zeta) : = \sup_{[\eta,\zeta]} \frac{1}{2}\|\hat{U}_\epsilon\|_{L_\epsilon^2}^2 + \int_\eta^\zeta \|\nabla_3 \hat{U}_\epsilon\|_{L_\epsilon^2}^2 dt + \sup_{[\eta,\zeta]} \frac{1}{2}\|\hat{\Theta}_\epsilon\|_{L^2}^2 + \int_\eta^\zeta \|\nabla_3 \hat{\Theta}_\epsilon\|_{L^2}^2 dt,
\end{gather}
where we denote by $\hat{U}_\epsilon := \hat{u}_\epsilon^{SNS} - u$ and $\hat{\Theta}_\epsilon := \hat{\theta}_\epsilon^{SNS} - \theta$. We remind that $(u,\theta)$ denote the solution to the weak hydrostatic primitive equation problem $(\mathcal{P}_\epsilon^{PE})$, and that, with the notation
\begin{equation}
    \mathcal{E}^{1,4}(\eta,\zeta) : = \sup_{[\eta,\zeta]} \frac{1}{2}\|(v,\theta)\|_{V}^4 + \int_\eta^\zeta \|(v,\theta)\|_{V}^2 \|(v,\theta)\|_{\mathcal{D}(A)}^2 dt,
\end{equation}
we have, for $T>0$, almost surely,
\begin{equation}
    \mathcal{E}^{1,4}(0,T) < \infty.
\end{equation}
Hence, the following theorem is in the line of Theorem \ref{theorem-cvg-weak}, and of which proof follows essentially the same reasoning.
\begin{theorem}\label{theorem-cvg-weak}
    Assume that $K \in H^3$, and that the noise term $\sigma dW_t$ in the primitive equations is bidimensional and fulfils the regularity conditions \eqref{smoothness-noise}. Let $\epsilon >0$, and equip $(\mathcal{P}_\epsilon^{rSNS})$ and $(\mathcal{P}_\epsilon^{PE})$ with the same initial condition $(v_0, \theta_0)^\tr \in V$, so that $U_\epsilon\vert_{t=0} = \Theta_\epsilon\vert_{t=0} = 0$. Then, for all $T >0$, and $\delta > 0$, there exists a constant $C_\delta >0$ such that,
    \warning{
    \begin{multline}
        \E\Big[\hat{\mathcal{E}}_\epsilon^0(0,T \wedge \tau_\delta) \Big] \leq C_\delta \epsilon^2  \E\Bigg[\int_0^{T \wedge \tau_\delta} \| v \|_{H^1}^4 + (1 + \|v\|_{H^1}^{2}) \|v\|_{H^2}^{2} dt \Bigg] \\
        + C_\delta (1+\alpha_\sigma^4) \epsilon^2 \E\Bigg[\int_0^{T \wedge \tau_\delta} \sum_k \|(I-K*)[ \phi_k \bcdot \nabla w ]\|_{L^2}^2 + \|(a - a^K)\nabla w\|_{L^2}^2 dt \Bigg], \label{eq-thm-cvg-weak}
    \end{multline}
    }
    where $\tau_{\delta} := \inf \Big\{t>0 \Big|\mathcal{E}^{1,4}(0,t) \geq \delta \Big\}$, so that $\tau_\delta \rightarrow \infty$ in probability as $\delta \rightarrow \infty$. In particular, $\hat{\mathcal{E}}_\epsilon^0(0,T) \rightarrow 0$ in probability whenever $\alpha_\sigma = o(\epsilon^{-1/2})$.
\end{theorem}
\begin{remark}
    The estimate \eqref{eq-thm-cvg-weak} suggests that the weak hydrostatic assumption improves the constant in front of the term of order $\alpha_\sigma^4 \epsilon^2$: choosing the regularising kernel $K$ such that $\|I - K *\|_{ \mathcal{L}_2(L^2(\mathcal{S},\R^3))}$ is small would reduce the value of the factor
        $$\E\Bigg[ \int_0^{T \wedge \tau_\delta} \sum_k \|(I-K*)[ \phi_k \bcdot \nabla w ]\|_{L^2}^2 + \|(a - a^K)\nabla w\|_{L^2}^2 dt \Bigg]$$
    if $w$ was $K$-independent. However, the asymptotics of $w$ when $K*$ approaches $I$ is not known precisely.   
\end{remark}
\warning{In turn, the estimate above shows that (PE$_{\alpha_\sigma}^\epsilon$) approximates (rSNS$_{\alpha_\sigma}^\epsilon$) much better than (SNS$_{\alpha_\sigma}^\epsilon$). This is natural regarding the way these problems were constructed. Notice that, by similar arguments, the error of convergence between the strongly hydrostatic primitive equations and the regularised Navier-Stokes can be shown of order $O(\alpha_\sigma^4 \epsilon^2)$ as well. The proof runs almost identically and is therefore omitted. As such, the error of convergence between the primitive equations and the regularised Navier-Stokes is asymptotically less with the weak hydrostatic hypothesis than with the strong one.}

Following the method of Section \ref{sec-abstract-formulation}, we define the problem $(\Pi_\epsilon^{SNS})$, which the counterpart of $(\mathcal{P}_\epsilon^{SNS})$ with periodic boundary conditions. Then the following proposition holds.
\begin{proposition}\label{thm:SNS-strong}
    The problem $(\Pi_\epsilon^{SNS})$, equipped with the initial condition $(u_0, \theta_0) \in V$, admits a unique maximal pathwise solution, which is a priori only local-in-time.
\end{proposition}
Reasoning as for Proposition \ref{thm:SNS-reg-strong}, the proof of Proposition \ref{thm:SNS-strong} can be adapted straightforwardly from \cite{AV2024-NS-tranport-noise}. From this result, the problem $(\Pi_\epsilon^{SNS})$ is ensured to have a unique maximal pathwise solution, provided that the initial condition is in $V$. Hence, we state a analogous of Theorem \ref{theorem-cvg-strong}.
\begin{theorem}\label{theorem-cvg-strong}
    Assume that the assumptions of Theorem \ref{theorem-cvg-weak} hold. Let $(v_\epsilon, \theta_\epsilon)$ and $(v, \theta)$ be, respectively, the solutions to the problems $(\Pi_\epsilon^{SNS})$ and $(\Pi_\epsilon^{PE})$ with the same initial condition $(v_0, \theta_0) \in \mathcal{D}(A)$, and on the same stochastic basis. Let $\tau_\epsilon^{SNS}$ be the stopping time associated to $(\Pi_\epsilon^{SNS})$, and define, for all stopping times $\eta,\zeta$ such that $\eta<\zeta$,
    \begin{equation}
    \hat{\mathcal{E}}_\epsilon^1(\eta,\zeta) := \left\{
    \begin{array}{ll}
        \sup_{[\eta, \zeta]} \frac{1}{2}\Big(\|\hat{U}_\epsilon\|_{V_\epsilon^2}^2+\|\hat{\Theta}_\epsilon\|_{V}^2\Big) + \int_{\eta}^{\zeta} \Big(\|\hat{U}_\epsilon\|_{D_\epsilon(A)^2}^2 + \|\hat{\Theta}_{\epsilon}\|_{\mathcal{D}(A)}^2\Big) dt \text{ if $\zeta < \tau_\epsilon^{SNS}$},\\
        \infty \text{ otherwise}.
    \end{array}
    \right.
    \end{equation}
    Then, using same notations, Theorem \ref{theorem-cvg-weak} holds, and $v \in L^2(\Omega, L^2([0,T],H^3(\mathcal{S},\R^2))$. Moreover, for all $T > 0$, there exists a constant $C_1 > 0$ and a stopping time
    \begin{multline}
        \check{\tau}_\epsilon = \inf \Big\{t \in \R_+ \Big| \sup_{[0, t]} (\|\hat{U}_\epsilon\|_{V_\epsilon}^2 + \|\hat{\Theta}_\epsilon\|_{V}^2) + \int_0^t(\| \hat{U}_\epsilon\|_{\mathcal{D}_\epsilon(A)}^2 + \|\hat{\Theta}_\epsilon\|_{\mathcal{D}(A)}^2) dt > \frac{1}{4C_1} \Big\} \leq \tau_\epsilon^{SNS}.
    \end{multline}
    such that, for all $\delta > 0$,
    \begin{multline}
        \E \big[\mathcal{E}_\epsilon^1(0, T \wedge \check{\tau}_\epsilon \wedge \tau_\delta)\big]
        \leq C_1\epsilon^2 (1+ \alpha_\sigma^4)\E\int_{0}^{\check{\tau}_{\epsilon} \wedge \tau_\delta \wedge T} \Bigg( (1+\|v\|_{H^3}^2)\\
        \times (1+\|(a - a^K)\nabla w\|_{H^1}^2 + \sum_k \|(I-K*)[(\phi_k \bcdot \nabla_3)w]\|_{H^1}^2) \Bigg) dt. \label{eq-bound-thm2}
    \end{multline}
    Furthermore, $\check{\tau}_\epsilon, \tau_\epsilon^{SNS} \rightarrow \infty$ in probability. Consequently, $\E \big[\mathcal{E}_\epsilon^1(0, T \wedge \check{\tau}_\epsilon \wedge \tau_\delta)\big]$ converges to $0$ whenever $\alpha_\sigma = o(\epsilon^{-1/2})$, and in such case $\mathcal{E}_\epsilon^1(0, T) \rightarrow 0$ in probability as well.
\end{theorem}
%Quite differently from Theorems \ref{theorem-cvg-weak} and \ref{theorem-cvg-strong} about (rSNS$_{\alpha_\sigma}^\epsilon$), here the error between (SNS$_{\alpha_\sigma}^\epsilon$) and (PE$_{\alpha_\sigma}^\epsilon$) is only shown bounded by a term of order $O(\epsilon^2(1+\alpha_\sigma^4))$. We thus interpret Theorems \ref{theorem-cvg-weak} and \ref{theorem-cvg-strong} as (PE$_{\alpha_\sigma}^\epsilon$) being a suitable approximation of (SNS$_{\alpha_\sigma}^\epsilon$) whenever $\alpha_\sigma = o(\epsilon^{-1/2})$. As such, due to the presence of the term $K$-dependent terms, this analysis does not distinguish the orders of convergence of the weak and strong hydrostatic hypotheses. The regularisation kernel $K$ yet adds a degree of freedom which impacts \emph{a priori} the error of convergence.
In Theorem \ref{theorem-cvg-strong}, the divergence of $\tau_\epsilon^{SNS}$ to infinity stands as a generalisation in the stochastic setting of the well-posedness of the Navier-Stokes equations for thin enough domains. Nevertheless, almost sure estimates would typically be needed to obtain a result as strong as in the deterministic context using the methods of \cite{LT2019}, while in our work they only hold in second order moment \emph{a priori}. This is due to the non-boundedness of our noise -- see \cite{chueshov2008random}. \warning{Again, such result is standard in the context of SPDEs.

In addition, Theorem \ref{thm:SNS-strong} can be adapted to the estimation of convergence between the strongly hydrostatic primitive equations and the regularised Navier-Stokes, similarly as for Theorem \ref{thm:SNS-weak}. The proof is omitted since it runs almost identically.}

The two following sections are dedicated to proving Theorems \ref{theorem-cvg-weak} and \ref{theorem-cvg-strong}. The proofs of Theorems \ref{theorem-cvg-weak-VREG} and \ref{theorem-cvg-strong-VREG} are omitted, since they run almost identically as that of Theorems \ref{theorem-cvg-weak} and \ref{theorem-cvg-strong}. In fact, the latter case is slightly more intricate, due to the presence of additional $K$-dependent terms in the final energy estimates. For this reason, these proofs are developed in full details hereafter. \warning{We emphasize that these proofs only cover the bidimensional noise case, as this assumption is needed to establish the well-posedness of the LU primitive equations \cite{DMM2025}.}

\section{Proof of Theorem \ref{theorem-cvg-weak}}\label{sec-thm-cvg-weak}

First, we show the following lemma, which will be useful further in the proof.
\begin{lemma}\label{lemma-estimate}
For all $T>0$, there exists a constant $C$ that is independent of $\epsilon$, such that
\begin{equation*}
     \E\Big[\sup_{[0,T]} \|u_\epsilon\|_{L_\epsilon^2}^4\Big] + \E \int_0^T \|v_\epsilon\|_H^2 \|v_\epsilon\|_V^2 + \epsilon^4 \|w_\epsilon\|_{L^2}^2 \|w_\epsilon\|_V^2 dt \leq C,
\end{equation*}
\end{lemma}
\emph{Proof:} This results follows immediately from the energy estimate \eqref{eq-defWeakSol-energ} of Definition \ref{def-GlobMartWeakSol} with $p=2$, the Burkholder-Davis-Gundy inequality and the stochastic Grönwall lemma of \cite{GHZ_2009}.
\CQFD

\noindent
Let $T,\delta > 0$. Let two stopping times $\eta$ and $\zeta$ be such that $T \wedge \tau_\delta \geq \zeta > \eta \geq 0$, where $\tau_\delta$ is defined as in Theorem \ref{theorem-cvg-weak}. In order to apply the stochastic Grönwall lemma \cite{GHZ_2009}, we seek an estimate of the following quantity,
\begin{align}
    \E\Bigg[\sup_{[\eta,\zeta]} \frac{1}{2}\|U_\epsilon\|_{L_\epsilon^2}^2 + \int_\eta^{\zeta} \|\nabla_3 U_\epsilon\|_{L_\epsilon^2}^2 dt \Bigg] = \E\Bigg[\sup_{[0,T]} \frac{1}{2}\|\mathbbm{1}_{[\eta, \zeta]}(t) U_\epsilon\|_{L_\epsilon^2}^2 + \int_0^{T} \|\mathbbm{1}_{[\eta, \zeta]}(t) \nabla_3 U_\epsilon\|_{L_\epsilon^2}^2 dt \Bigg].
\end{align}
For simplicity, we write the estimates with $(\eta,\zeta) = (0,T)$, the proof being similar with a general couple of stopping times $(\eta,\zeta)$. For this purpose, we develop first the term $\frac{1}{2}\|U_\epsilon\|_{L_\epsilon^2}^2(t)$ as follows, for all $t>0$,
\begin{align}
    \frac{1}{2}\|U_\epsilon\|_{L_\epsilon^2}^2(t)
    &= \frac{1}{2}\|u_\epsilon\|_{L_\epsilon^2}^2(t) + \Big(\frac{1}{2}\|u\|_{L_\epsilon^2}^2(t) - (u_\epsilon,u)_{L_\epsilon^2}(t)\Big). \label{eq-med-equality}
\end{align}
In the following, we estimate the two terms on the RHS of equation \eqref{eq-med-equality}. The first term is estimated directly from the definition of global martingale Leray-Hopf weak solutions,
\begin{multline}
    \frac{1}{2}\|u_\epsilon\|_{L_\epsilon^2}^2(t) + \int_0^{t} \| \nabla_3 u_\epsilon\|_{L_\epsilon^2}^2 dt + \frac{1}{2} \int_0^{t} (a \nabla I_{\alpha_\sigma} u_\epsilon, \nabla I_{\alpha_\sigma} u_\epsilon)_{L_\epsilon^2} dt \\
    \leq \frac{1}{2}\|u_0\|_{L_\epsilon^2}^2 + \frac{1}{2} \int_0^{t} \sum_k \|\mathbf{P}_\epsilon (\phi_k \bcdot \nabla I_{\alpha_\sigma} u_\epsilon) + \Delta I_{ \alpha_\sigma} \phi_k \|^2_{L_\epsilon^2} dt\\
    + \int_0^{t} \big(\Delta I_{\alpha_\sigma} \sigma dW_t, u_\epsilon\big)_{L_\epsilon^2} - \int_0^{t}(\frac{\rho_\epsilon}{\rho_0} g, w_\epsilon)_{L^2} dt. \label{eq-first-term}
\end{multline}
It remains now to estimate the second term on the RHS of \eqref{eq-med-equality}. The proof is split into three steps. The first one is dedicated to the estimate of the second term of equation \eqref{eq-med-equality}. Then, we derive an evolution formula for the quantities $\|U_\epsilon\|_{L_\epsilon^2}^2$ and $\|\Theta_\epsilon\|_{L^2}^2$ in the second step. For readability, we divide it into two substeps, one for each quantity. Eventually, we give the final arguments for the proof of Theorem \ref{theorem-cvg-weak} in the third step.

\bigskip
\noindent
\emph{Step 1: Estimate of the second term on the RHS of \eqref{eq-med-equality}.}
Let $T \in \R_+^*$. In the following, we adopt the convention whereby, for all $T > 0$, for all function $u^\sharp \in L^2([0,T],V)$,
\begin{equation}
    \int_0^T \Big(u^\sharp, d_t u \Big)_{L_\epsilon^2} = \int_0^T \Big(u^\sharp, F_0 dt \Big)_{L_\epsilon^2} + \int_0^T \Big(u^\sharp, G_0 dW_t \Big)_{L_\epsilon^2}, \label{eq-abuse-of-notation}
\end{equation}
where, using the expression of the problem $(\mathcal{P}_\epsilon^{PE})$, we denote by $F_0 dt = (F_0^H \: F_0^z)^\tr dt$ and $G_0 dW_t = (G_0^H \: G_0^z)^\tr dW_t$, with
\begin{multline}
    F_0^H dt + G_0^H dW_t = \mathbf{P}\Big((u^* dt + \sigma dW_t) \bcdot \nabla_3\Big) v -  \mathbf{P} \Delta_3 (v dt + \sigma^H dW_t) \\
    - \mathbf{P} \nabla_H \int_z^1 \partial_z (p\:dt + dp_t^{\sigma}) dz' - \frac{1}{2}  \mathbf{P} \nabla_3 \bcdot(a \nabla_3 v) dt, \label{eq-abuse-of-notation-bis}
\end{multline}
and
\begin{equation}
    F_0^z dt + G_0^z dW_t = \int_z^1 \nabla_H \bcdot F_0^H dt \: dz' + \int_z^1 \nabla_H \bcdot G_0^H dW_t\: dz'.
\end{equation}
It can be checked that all the terms in equations \eqref{eq-abuse-of-notation} and \eqref{eq-abuse-of-notation-bis} make sense. We draw attention to the following abuse of notation,
\begin{equation}
    \Big(u^\sharp, d_t u \Big)_{L_\epsilon^2} = \Big(v^\sharp, d_t v \Big)_{L^2} + \epsilon^2 \Big(w^\sharp, d_t w \Big)_{L^2},
\end{equation}
where the last term reads $\epsilon^2 \Big(w^\sharp, d_t w \Big)_{L^2} \equiv \epsilon^2 \langle F_0^z dt + G_0^z dW_t, w^\sharp \rangle_{(H_{0,z}^1(\mathcal{S}, \R))' \times H_{0,z}^1(\mathcal{S},\R)}$, with $H_{0,z}^1(\mathcal{S}, \R)$ being the space of functions of regularity $H^1(\mathcal{S}, \R)$ that cancel on the upper and bottom boundaries $\Gamma_u$ and $\Gamma_b$. We emphasize that the abuse of notation introduced in equation \eqref{eq-abuse-of-notation} is only a convenient formalism to shorten the expressions of our formulas. Moreover, it is only present in the steps 1 and 2.1 of the current proof.

Denote by $\mathcal{H}_{3D} = \{u=(v \: w)^\tr| v \in H, w = \int_z^1 \nabla_H \bcdot v\}$. By classical arguments, $H$ admits a Hilbert basis in $C^\infty(\mathcal{S},\R^2)$, and then $\mathcal{H}_{3D}$ admits a basis of regularity $C^\infty(\mathcal{S},\R^3)$ by remarking that the condition $w = \int_z^1 \nabla_H \bcdot v$ is closed. Thus, choosing any such Hilbert basis $(e_n) \in C^\infty(\mathcal{S},\R^3)$ of $\mathcal{H}_{3D}$, we denote by $u_n := \mathbf{P}_n u$ the projection of $u$ onto $Span(e_i, 0\leq i\leq n)$, so that $u_n$ is a semi-martingale and $u_n \in C([0,T], C^\infty(\mathcal{S},\R^2))$ a.s. Moreover, $d_t u_n = \mathbf{P}_n[F_0 dt + G_0 dW_t]$ and $(u_n)\vert_{t=0} = \mathbf{P}_n u_0$. In addition, we denote by $u_n = (v_n, w_n)^\tr$, where $w_n = \int_z^1 \nabla_H \bcdot v_n dz'$ by the divergence-free condition. Thus, we can extend the aforementioned abuse of notation to the projection $u_n$ of $u$. Applying It\={o}'s lemma to $(u_\epsilon, u_n)_{L_\epsilon^2}$ -- that is taking $u^\sharp = u_\epsilon$ -- we infer that
\begin{multline}
    - \int_0^T \Big(u_\epsilon, d_t u_n \Big)_{L_\epsilon^2} + \int_0^T \Bigg((u_\epsilon^* dt \bcdot \nabla_3) u_\epsilon + (\sigma dW_t \bcdot \nabla_3) I_{\alpha_\sigma} u_\epsilon,  u_n \Bigg)_{L_\epsilon^2}\\
    - \frac{1}{2} \int_0^T \big(\nabla_3 \bcdot(a \nabla_3 I_{\alpha_\sigma^2} u_\epsilon), u_n \big)_{L_\epsilon^2} - \int_0^T \big(\Delta (u_\epsilon dt + I_{\alpha_\sigma} \sigma dW_t),  u_n \big)_{L_\epsilon^2}  = \|u_0\|_{L_\epsilon^2(\mathcal{S})}^2 - \big(u_\epsilon, u_n \big)_{L_\epsilon^2}(T) \\
    + \int_0^T \Bigg[  \sum_k  (\mathbf{P}(\phi_k \bcdot \nabla v_n) + \Delta \phi_k^H + \nabla_H \tilde{\pi}_k, \mathbf{P}_\epsilon^H(\phi_k \bcdot \nabla I_{\alpha_\sigma} u_\epsilon) + \Delta \phi_k^{H})_{L^2} \\
    + \epsilon^2 \Big( \int_z^1 \nabla_H \bcdot \mathcal{R}(\phi_k \bcdot \nabla v_n) + \Delta \phi_k^z + \int_z^1 \Delta_H \tilde{\pi}_k , \mathbf{P}_\epsilon^z(\phi_k \bcdot \nabla I_{\alpha_\sigma} u_\epsilon) + \alpha_\sigma \Delta \phi_k^{z}\Big)_{L^2} - (\frac{\rho_\epsilon}{\rho_0} g,  w_n)_{L^2}\Bigg] dt, \label{eq-final1.1-withSigma}
\end{multline}
remarking that the bounded variation pressure terms cancel since $u_n$ is divergence-free. The use of It\={o}'s lemma is permitted here since $u_n \in Span(e_i, 0\leq i\leq n)$, so that the inner products of the form $(\bcdot, u_n)_{L_\epsilon^2}$ can be regarded as an inner products over \emph{finite dimension} spaces. To derive a relation involving the term $\int_0^T \big(u_\epsilon, d_t u \big)_{L_\epsilon^2}$, we want to pass to limit when $n \rightarrow \infty$. Since the process $v$ is in $L^2(\Omega, C([0,T], V) \cap L^2([0,T], \mathcal{D}(A)))$, we only need to prove that, a.s.,
$$\int_0^T \Big(u_\epsilon, d_t u_n \Big)_{L_\epsilon^2} \rightarrow \int_0^T \Big(u_\epsilon, d_t u \Big)_{L_\epsilon^2}$$
the other terms being meaningful by adapting the arguments of \cite{LT2019}. Let us write the evolution equation of $u_n$ as $d_t u_n = \mathbf{P}_n F_0 dt + \mathbf{P}_n G_0 dW_t$, so that, using the aforementioned notation,
\begin{equation}
    \int_0^T \Big(u_\epsilon, d_t u_n \Big)_{L_\epsilon^2} = \int_0^T \Big(u_\epsilon, \mathbf{P}_n F_0 dt + \mathbf{P}_n G_0 dW_t \Big)_{L_\epsilon^2}. \label{eq-integral-decomposition}
\end{equation}
By standard arguments, thanks to the regularity of the primitive equation solution $u$ -- namely we have that $u \in C([0,T],V) \cap L^2([0,T],\mathcal{D}(A))$ a.s. -- we infer that, a.s.,
\begin{equation}
    \int_0^T \Big(u_\epsilon, \mathbf{P}_n F_0 dt \Big)_{L_\epsilon^2} \rightarrow \int_0^T \Big(u_\epsilon, F_0 dt \Big)_{L_\epsilon^2}.
\end{equation}
Additionally, by the Burkholder-Davis-Gundy inequality, denoting by $\mathcal{U} := H^{-s}(\mathcal{S},\R^3)$ with $s > \frac{3}{2}$,
\begin{multline}
    \E \Bigg[ \sup_{t\in[0,T]} \bigg| \int_{0}^t \Big(u_\epsilon, (\mathbf{P}_n G_0 - G_0) dW_s \Big)_{L_\epsilon^2} \bigg|^2\Bigg] \leq C \E \int_{0}^T \|(u_\epsilon, (\mathbf{P}_n G_0 - G_0) [\bcdot])_{L_\epsilon^2}\|_{\mathcal{L}_2(\mathcal{U}, \R) }^2 ds\\
    \leq C \Big(\int_{0}^T \|(\mathbf{P}_n G_0 - G_0) [\bcdot]\|_{\mathcal{L}_2(\mathcal{U}, L^2(\mathcal{S},\R^3)) }^2 ds\Big) \E \sup_{[0,T]} \|u_\epsilon\|_{L_\epsilon^2}^2,
\end{multline}
which shows that the LHS converges in quadratic mean, hence a.s. by thinning the sequence. Consequently, by taking the limit of a suitable subsequence of $(u_n)$ in equation \eqref{eq-integral-decomposition}, we infer
\begin{equation}
    \int_0^T \Big(u_\epsilon, d_t u_n \Big)_{L_\epsilon^2} \rightarrow \int_0^T  \Big(u_\epsilon, F_0 dt + G_0 dW_t \Big)_{L_\epsilon^2} = \int_0^T \Big(u_\epsilon,  d_t u \Big)_{L_\epsilon^2}.
\end{equation}
Thus, taking the limit of such a subsequence in equation \eqref{eq-final1.1-withSigma} yields,
\begin{multline}
    - \int_0^T \Big(u_\epsilon, d_t u \Big)_{L_\epsilon^2} + \int_0^T \Bigg((u_\epsilon^* dt \bcdot \nabla_3) u_\epsilon + (\sigma dW_t \bcdot \nabla_3) I_{\alpha_\sigma} u_\epsilon,  u \Bigg)_{L_\epsilon^2}\\
    - \frac{1}{2} \int_0^T \big(\nabla_3 \bcdot(a \nabla_3 I_{\alpha_\sigma^2} u_\epsilon), u \big)_{L_\epsilon^2} - \int_0^T \big(\Delta (u_\epsilon dt + I_{\alpha_\sigma} \sigma dW_t),  u \big)_{L_\epsilon^2}  = \|u_0\|_{L_\epsilon^2(\mathcal{S})}^2 \\
    + \int_0^T \Bigg[  \sum_k  (\mathbf{P}(\phi_k \bcdot \nabla v) + \Delta \phi_k^H + \nabla_H \tilde{\pi}_k, \mathbf{P}_\epsilon^H(\phi_k \bcdot \nabla I_{\alpha_\sigma} u_\epsilon) + \Delta \phi_k^{H})_{L^2} \\
    + \epsilon^2 \Big( \int_z^1 \nabla_H \bcdot \mathcal{R}(\phi_k \bcdot \nabla v) + \Delta \phi_k^z + \int_z^1 \Delta_H \tilde{\pi}_k , \mathbf{P}_\epsilon^z(\phi_k \bcdot \nabla I_{\alpha_\sigma} u_\epsilon) + \alpha_\sigma \Delta \phi_k^{z}\Big)_{L^2} - (\frac{\rho_\epsilon}{\rho_0} g,  w)_{L^2}\Bigg] dt. \label{eq-almost-final-step1}
\end{multline}
Also, using a similar approximation argument as above and applying It\={o}'s lemma, we infer
\begin{align}
     \int_0^{T} \big( u_\epsilon, d_t &u \big)_{L_\epsilon^2} dt = \int_0^{T} ( d_t v, v_\epsilon )_{L^2} + \epsilon^2 \langle d_t w, w_\epsilon \rangle_{(H_{z0}^1)' \times H_{z0}^1} dt \nonumber\\
     &= \int_0^{T} ( d_t v, V_\epsilon )_{L^2} + ( d_t v, v )_{L^2} + \epsilon^2 \langle d_t w, W_\epsilon \rangle_{(H_{z0}^1)' \times H_{z0}^1} + \epsilon^2 \langle d_t w, w \rangle_{(H_{z0}^1)' \times H_{z0}^1} dt \nonumber\\
     &= \int_0^{T} \Big( d_t v, V_\epsilon + \epsilon^2 \int_z^1 \nabla_H W_\epsilon \Big)_{L^2} + \frac{1}{2} \int_0^{T} d_t \Big(\|v\|_{L^2}^2 + \epsilon^2 \|w\|_{L^2}^2\Big) \nonumber\\
     &\qquad \quad - \frac{1}{2} \int_0^{T} \sum_k \|\mathbf{P} (\phi_k \bcdot \nabla v) + \Delta \phi_k^H + \nabla_H \tilde{\pi}_k \|^2_{L^2} dt \nonumber \\
     &\qquad \quad - \frac{\epsilon^2}{2} \int_0^{T} \sum_k \| \int_z^1 \nabla_H \bcdot \mathcal{R} (\phi_k \bcdot \nabla v) +\Delta \phi_k^z + \int_z^1 \Delta_H \tilde{\pi}_k \|^2_{L^2} dt. \label{eq-develop-innerProd-step1}
\end{align}
Gathering equations \eqref{eq-almost-final-step1} and  \eqref{eq-develop-innerProd-step1} eventually yields,
\begin{multline}
    \big( u_\epsilon, u \big)_{L_\epsilon^2} (T) - \frac{1}{2}\|u\|_{L_\epsilon^2}^2(T) + \int_0^{T} \Bigg((u_\epsilon^* dt \bcdot \nabla_3) u_\epsilon + (\sigma dW_t \bcdot \nabla_3) I_{\alpha_\sigma} u_\epsilon, u \Bigg)_{L_\epsilon^2} \\
    - \int_0^{T} \big(\Delta (u_\epsilon dt + I_{\alpha_\sigma} \sigma dW_t), u\big)_{L_\epsilon^2} - \int_0^{T} \frac{1}{2} \big(\nabla_3 \bcdot(a \nabla_3 I_{\alpha_\sigma^2} u_\epsilon), u\big)_{L_\epsilon^2(Q_{T})} dt = \frac{1}{2}\|u_0\|_{L_\epsilon^2}^2 \\
    + \int_0^{T} \Big( d_t v, V_\epsilon + \epsilon^2 \int_z^1 \nabla_H W_\epsilon \Big)_{L^2} + \int_0^{T} \Bigg[ - \frac{1}{2} \sum_k \|\mathbf{P}(\phi_k \bcdot \nabla v) + \Delta \phi_k^H + \nabla_H \tilde{\pi}_k \|^2_{L^2}\\
    - \frac{\epsilon^2}{2} \sum_k \| \int_z^1 \nabla_H \bcdot \mathcal{R}(\phi_k \bcdot \nabla v) + \Delta \phi_k^z\|^2_{L^2}
    + \sum_k  (\mathbf{P}(\phi_k \bcdot \nabla v) + \Delta \phi_k^H + \nabla_H \tilde{\pi}_k , \mathbf{P}_\epsilon^H(\phi_k \bcdot \nabla I_{\alpha_\sigma} u_\epsilon) + \Delta \phi_k^{H})_{L^2}\\
    + \epsilon^2(\int_z^1 \nabla_H \bcdot \mathcal{R}(\phi_k \bcdot \nabla v) + \Delta \phi_k^z + \int_z^1 \Delta_H \tilde{\pi}_k , \mathbf{P}_\epsilon^z(\phi_k \bcdot \nabla I_{\alpha_\sigma} u_\epsilon) + \alpha_\sigma \Delta \phi_k^{z})_{L^2} - (\frac{\rho_\epsilon}{\rho_0} g, w)_{L^2} \Bigg]dt. \label{eq-second-term}
\end{multline}
\emph{Step 2.1: Partial estimate on $U_\epsilon$.} By substracting equation \eqref{eq-first-term} to \eqref{eq-second-term}, and using equation \eqref{eq-med-equality}, we find
\begin{multline}
    -\frac{1}{2}\|U_\epsilon\|_{L_\epsilon^2}^2(T) - \int_0^{T} \|U_\epsilon\|_{L_\epsilon^2}^2 dt+ \int_0^{T} \Bigg((u_\epsilon^* dt \bcdot \nabla_3) u_\epsilon + (\sigma dW_t \bcdot \nabla_3) I_{\alpha_\sigma} u_\epsilon, u \Bigg)_{L_\epsilon^2} \\
    - \frac{1}{2} \big( a \nabla_3 I_{\alpha_\sigma} u_\epsilon, \nabla_3 I_{\alpha_\sigma} U_\epsilon \big)_{L_\epsilon^2(Q_{T})} + \big(\Delta u, U_\epsilon \big)_{L_\epsilon^2(Q_{T})} \geq \int_0^{T} ( d_t v, V_\epsilon + \epsilon^2 \int_z^1 \nabla_H W_\epsilon \Big)_{L^2} dt\\
    - \frac{1}{2} \int_0^{T} \sum_k \|\mathbf{P}(\phi_k \bcdot \nabla v) + \Delta \phi_k^H + \nabla_H \tilde{\pi}_k \|^2_{L^2} + \|\mathbf{P}_\epsilon^H (\phi_k \bcdot \nabla I_{\alpha_\sigma} u_\epsilon) + \Delta \phi_k^{H}\|^2_{L^2} dt \\
    - \frac{\epsilon^2}{2} \int_0^{T} \sum_k \|\int_z^1 \nabla_H \bcdot \mathcal{R}(\phi_k \bcdot \nabla v) + \Delta \phi_k^z + \int_z^1 \Delta_H \tilde{\pi}_k \|^2_{L^2} + \|\mathbf{P}_\epsilon^z (\phi_k \bcdot \nabla I_{\alpha_\sigma} u_\epsilon) + \alpha_\sigma \Delta \phi_k^{z}\|^2_{L^2} dt\\
    + \int_0^{T} \Bigg[\sum_k (\mathbf{P}(\phi_k \bcdot \nabla v) + \Delta \phi_k^H + \nabla_H \tilde{\pi}_k , \mathbf{P}_\epsilon^H(\phi_k \bcdot \nabla I_{\alpha_\sigma} u_\epsilon) + \Delta \phi_k^{H})_{L^2}\\
    + \epsilon^2(\int_z^1 \nabla_H \bcdot \mathcal{R} (\phi_k \bcdot \nabla v) + \Delta \phi_k^z + \int_z^1 \Delta_H \tilde{\pi}_k , \mathbf{P}_\epsilon^z(\phi_k \bcdot \nabla I_{\alpha_\sigma} u_\epsilon) + \alpha_\sigma \Delta \phi_k^{z})_{L^2}\Bigg]dt\\
    - \int_0^{T} \big(\Delta I_{\alpha_\sigma} \sigma dW_t, U_\epsilon\big)_{L_\epsilon^2} + \int_0^{T} (\frac{\rho_\epsilon}{\rho_0} g, W_\epsilon)_{L^2} dt. \label{eq:big-eq-step2.1}
\end{multline}
In addition, by taking the $(\bcdot, \bcdot)_{L_\epsilon^2}$ inner product of equation \eqref{eq:abstract-PE-v} with $V_\epsilon$, we find that
\begin{multline}
    (d_t v, V_\epsilon)_{L^2} + (-\Delta v, V_\epsilon)_{L^2} dt + (-\Delta \sigma dW_t^H, V_\epsilon)_{L^2} - \frac{1}{2} (\nabla \bcdot a \nabla v, V_\epsilon)_{L^2} dt \\
    + (\nabla_H (\tilde{p} dt  + d\tilde{p}_t^\sigma  ),V_\epsilon)_{L^2}  = -((u^* dt + \sigma dW_t) \bcdot \nabla v, V_\epsilon)_{L^2} dt. \label{eq:dtv-Veps}
\end{multline}
Combining equations \eqref{eq:big-eq-step2.1} and \eqref{eq:dtv-Veps}, we infer the following relation,
\begin{multline}
    \frac{1}{2}\|U_\epsilon\|_{L_\epsilon^2}^2(T) + \int_0^{T} \|\nabla_3 U_\epsilon\|_{L_\epsilon^2}^2 dt + \frac{1}{2} \int_0^{T} \big( a \nabla_3 I_{\alpha_\sigma} U_\epsilon, \nabla_3 I_{\alpha_\sigma} U_\epsilon \big)_{L_\epsilon^2} dt \\
    \leq - \epsilon^2 \int_0^T \Big( d_t v, \int_z^1 \nabla_H W_\epsilon \Big)_{L^2} + \frac{1}{2} \int_0^{T} \sum_k \|\mathbf{P}_\epsilon^H (\phi_k \bcdot \nabla I_{\alpha_\sigma} u_\epsilon) - \mathbf{P}(\phi_k \bcdot \nabla v)  - \nabla_H \tilde{\pi}_k \|^2_{L^2} \\
    + \frac{\epsilon^2}{2} \int_0^{T} \sum_k \|\mathbf{P}_\epsilon^z (\phi_k \bcdot \nabla I_{\alpha_\sigma} u_\epsilon) - \int_z^1 \nabla_H \bcdot \mathcal{R} (\phi_k \bcdot \nabla v)  -  \int_z^1 \Delta_H \tilde{\pi}_k\|^2_{L^2}\\
    - \alpha_\sigma^2 \epsilon^2 \frac{1}{2} \Bigg(a \nabla_3 w , \nabla_3 W_\epsilon \Bigg)_{L^2(Q_{T})} + \epsilon^2 \Bigg(\Delta w , W_\epsilon \Bigg)_{L^2(Q_{T})} + \int_0^{T} \Bigg[\alpha_\sigma \epsilon^2 \big(\Delta \sigma dW_t^z, W_\epsilon \big)_{L^2}\\
    + \epsilon^2 \Bigg(\Big((u_\epsilon^* dt + \alpha_\sigma \sigma dW_t) \bcdot \nabla_3\Big) w_\epsilon, w \Bigg)_{L^2} + \Bigg(\Big((u_\epsilon^* dt + \sigma dW_t) \bcdot \nabla_3 \Big) v_\epsilon,v \Bigg)_{L^2} \\
    + \Bigg((u^* dt + \sigma dW_t) \bcdot \nabla v, V_\epsilon \Bigg)_{L^2} - (\partial_z (\tilde{p} dt  + d\tilde{p}_t^\sigma ), W_\epsilon)_{L^2} + (\frac{\rho_\epsilon}{\rho_0} g, W_\epsilon)_{L^2}\Bigg]. \label{eq-bigInequation-beforeEstimates}
\end{multline}
Moreover, we may simplify the following term,
\begin{align*}
    \Bigg(&\Big((u_\epsilon^* dt + \sigma dW_t) \bcdot \nabla_3 \Big) v_\epsilon,v \Bigg)_{L^2} + \Bigg((u^* dt + \sigma dW_t) \bcdot \nabla v, V_\epsilon \Bigg)_{L^2}\\
    &= \Bigg(\Big((u_\epsilon^* dt + \sigma dW_t) \bcdot \nabla_3 \Big) v_\epsilon,v \Bigg)_{L^2} - \Bigg((u^* dt + \sigma dW_t) \bcdot \nabla v_\epsilon, v \Bigg)_{L^2}\\
    &= \Big((U_\epsilon \bcdot \nabla_3) v_\epsilon,v \Big)_{L^2} dt = \Big((U_\epsilon \bcdot \nabla_3) V_\epsilon,v \Big)_{L^2} dt.
\end{align*}
Hence, equation \eqref{eq-bigInequation-beforeEstimates} can be rewritten as,
\begin{multline}
    \frac{1}{2}\|U_\epsilon\|_{L_\epsilon^2}^2(T) + \int_0^{T} \|\nabla_3 U_\epsilon\|_{L_\epsilon^2}^2 dt + \frac{1}{2} \int_0^{T} \big( a \nabla_3 I_{\alpha_\sigma} U_\epsilon, \nabla_3 I_{\alpha_\sigma} U_\epsilon \big)_{L_\epsilon^2} dt \\
    \leq \int_0^T \Bigg(\Big(U_\epsilon \bcdot \nabla_3 \Big) V_\epsilon,v \Bigg)_{L^2} dt + \epsilon^2 \int_0^T \Bigg(\Big(u_\epsilon^* \bcdot \nabla_3\Big) w_\epsilon, w \Bigg)_{L^2} dt\\
    + \int_0^T  -\alpha_\sigma^2 \epsilon^2 \frac{1}{2} (a \nabla_3 w , \nabla_3 W_\epsilon )_{L^2} + (-\partial_z \tilde{p} + \frac{\rho_\epsilon}{\rho} g + \epsilon^2 \Delta w, W_\epsilon)_{L^2}  dt 
    \\%--------------------------------------
    - \epsilon^2 \int_0^T \Big( d_t v, \int_z^1 \nabla_H W_\epsilon \Big)_{L^2} + \frac{1}{2} \int_0^{T} \sum_k \|\mathbf{P}_\epsilon^H (\phi_k \bcdot \nabla I_{\alpha_\sigma} u_\epsilon) - \mathbf{P}(\phi_k \bcdot \nabla v)  - \nabla_H \tilde{\pi}_k \|^2_{L^2} \\
    + \frac{\epsilon^2}{2} \int_0^{T} \sum_k \|\mathbf{P}_\epsilon^z (\phi_k \bcdot \nabla I_{\alpha_\sigma} u_\epsilon) - \int_z^1 \nabla_H \bcdot \mathcal{R} (\phi_k \bcdot \nabla v)  -  \int_z^1 \Delta_H \tilde{\pi}_k\|^2_{L^2}\\
    + \sum_k \int_0^{T} \Big[\alpha_\sigma\epsilon^2 \Big(( \phi_k \bcdot \nabla_3) W_\epsilon, w \Big)_{L^2} + \alpha_\sigma \epsilon^2 \big(\Delta \phi_k^z, W_\epsilon \big)_{L^2} - (\partial_z \pi_k, W_\epsilon)_{L^2} \Big] d\beta^k\\
    := J_1 + J_2 + J_3 + J_4 + J_5 + J_6 + \sum_k J_7^k d\beta^k. \label{eq-bigInequation-reordered}
\end{multline}
In addition, using the same arguments as in \cite{LT2019},
\begin{align*}
    J_1 = &\int_0^{T} \Bigg(\Big(U_\epsilon \bcdot \nabla_3 \Big) V_\epsilon,v \Bigg)_{L^2} dt \leq \int_0^{T}  |((V_\epsilon \bcdot \nabla_H) V_\epsilon, v)_{L^2}| + |(W_\epsilon \partial_z V_\epsilon, v)_{L^2}| dt\\
    &\leq \int_0^{T}  |((V_\epsilon \bcdot \nabla_H) V_\epsilon, v)_{L^2}| + |((\nabla_H \bcdot V_\epsilon) V_\epsilon, v)_{L^2}| + |(W_\epsilon V_\epsilon, \partial_z v)_{L^2}| dt\\
    &\leq C \int_0^{T} \|v\|_{L^6} (\|V_\epsilon\|_{L^3} \|\nabla V_\epsilon\|_{L^2} + \|V_\epsilon\|_{L^3} \|\nabla V_\epsilon\|_{L^2}) + \|V_\epsilon\|_V^{3/2} \|V_\epsilon\|_H^{1/2} \|v\|_V^{1/2} \|v\|_{\mathcal{D}(A)}^{1/2} dt\\
    &\leq \int_0^{T} C_\xi (1 + \|v\|_{V}^2 \|v\|_{\mathcal{D}(A)}^2) (\|V_\epsilon\|_H^2 + \epsilon^2 \|W_\epsilon\|_{L^2}^2) + \xi \|V_\epsilon\|_V^2  dt
\end{align*}
and,
\begin{align*}
    J_2 &= \epsilon^2 \int_0^{T} \Big(\big(u_\epsilon^* \bcdot \nabla_3\big) w_\epsilon, w \Big)_{L^2} dt = \epsilon^2 \int_0^{T} \int_\mathcal{S} \Big(\big(v_\epsilon^* \bcdot \nabla_H\big) W_\epsilon - w_\epsilon^* (\nabla_H \bcdot V_\epsilon) \Big) w  dt \\
    &\leq C_\xi \epsilon^2 \int_0^{T} (\|v_\epsilon^*\|_H^{1/2} \|v_\epsilon^*\|_V^{1/2} \|\nabla_3 W_\epsilon\|_{L^2} + \|w_\epsilon^*\|_H^{1/2} \|w_\epsilon^*\|_V^{1/2} \|\nabla_3 V_\epsilon\|_{L^2}) \|v\|_V^{1/2} \|v\|_{\mathcal{D}(A)}^{1/2} dt\\
    &\leq C_\xi \epsilon^2 \int_0^{T} (\|v_\epsilon^*\|_H^{2} \|v_\epsilon^*\|_V^{2}  + \epsilon^4 \|w_\epsilon^*\|^{2}_{L^2} \|w_\epsilon^*\|^{2}_V + \|v\|_V^{2} \|v\|_{\mathcal{D}(A)}^{2}) dt + \xi \int_0^T \|V_\epsilon\|_V^2 + \epsilon^2 \|\nabla_3 W_\epsilon\|_{L^2}^2 dt\\
    &\leq C_\xi \epsilon^2 \int_0^{T} (1+\|v_\epsilon\|_H^{2} \|v_\epsilon\|_V^{2}  + \epsilon^4 \|w_\epsilon\|^{2}_{L^2} \|w_\epsilon\|^{2}_V + \|v\|_V^{2} \|v\|_{\mathcal{D}(A)}^{2}) dt + \xi \int_0^T \|V_\epsilon\|_V^2 + \epsilon^2 \|\nabla_3 W_\epsilon\|_{L^2}^2 dt.
\end{align*}
Moreover,
\begin{align}
    J_3 &\leq |\int_0^T (\partial_z \tilde{p} +  \alpha_\sigma^2 \epsilon^2 \frac{1}{2} \nabla_3 \bcdot (a \nabla_3 w)  + \frac{\rho_\epsilon}{\rho_0} g, W_\epsilon)_{L^2} dt| \nonumber\\
    &\leq C \int_0^{T} \|\Theta_\epsilon\|_{L^2} \|W_\epsilon\|_{L^2} +  \alpha_\sigma^2 \epsilon^2 \frac{1}{2} |(\nabla_3 \bcdot ((a - a^K) \nabla_3 w), W_\epsilon)_{L^2}|  dt \nonumber\\
    &\leq C_\xi \int_0^{T} (\|\Theta_\epsilon\|_{L^2}^2  + \alpha_\sigma^4 \epsilon^2 \|(a - a^K) \nabla w\|_{L^2}^2 ) dt + \xi \int_0^{T} (\|V_\epsilon\|_{V}^2 + \epsilon^2 \|W_\epsilon\|_{H^1}^2) dt.
\end{align}
Eventually, we estimate the following term,
\begin{multline}
    J_4 = \epsilon^2 \Big|\int_0^T \Big( d_t v, \int_z^1 \nabla_H W_\epsilon \Big)_{L^2}\Big| = \epsilon^2 \Big|\int_0^T \Big( \mathbf{P}\Big((u^* dt + \sigma dW_t) \bcdot \nabla_3\Big) v -  \mathbf{P} \Delta_3 (v dt + \sigma^H dW_t)  \\
    -  \mathbf{P} \nabla_H \int_z^1 \partial_z (p\:dt + dp_t^{\sigma}) dz' - \frac{1}{2}  \mathbf{P} \nabla_3 \bcdot(a \nabla_3 v) dt, \int_z^1 \nabla_H W_\epsilon \Big)_{L^2}\Big|\\
    \leq C \epsilon^2 \int_0^T [(1 + \|a\|_{H^3}) \|v\|_{\mathcal{D}(A)} + \|\theta\|_{H^1}] \|W_\epsilon\|_{H^1} dt \\
    + C \epsilon^2 \int_0^T \Big|\Big( (u^* \bcdot \nabla_3) v, \int_z^1 \nabla_H W_\epsilon \Big)_{L^2} \Big| dt.
\end{multline}
The last term on the RHS can be expressed as
\begin{multline}
    \int_0^T \Big|\Big( (u^* \bcdot \nabla_3) v, \int_z^1 \nabla_H W_\epsilon \Big)_{L^2} \Big| dt = \int_0^T \Big|\Big( (v^* \bcdot \nabla_H) v, \int_z^1 \nabla_H W_\epsilon \Big)_{L^2} \Big| dt\\
    \qquad \qquad \qquad \qquad \qquad + \int_0^T \Big|\Big( w^* \partial_z v, \int_z^1 \nabla_H W_\epsilon \Big)_{L^2} \Big| dt \\
    \leq C \int_0^T [\|v^*\|_{L^3} \|v\|_{\mathcal{D}(A)} + \|v\|_{L^6} (\|w^*\|_{L^3} + \|\partial_z w^*\|_{L^3})] \|\nabla_H W_\epsilon\|_{L^2} dt\\
    \leq C \int_0^T [(1+\|v^*\|_{V}) \|v\|_{\mathcal{D}(A)} + \|v\|_{V} \|v^*\|_{V}] \|W_\epsilon\|_{H^1} dt\\
    \leq C_\xi \int_0^T [1+\|v\|_{V}^2\|v\|_{\mathcal{D}(A)}^2  + \|v\|_{V}^4] dt + \xi \int_0^T \|W_\epsilon\|_{H^1}^2 dt
\end{multline}
Therefore, gathering the previous estimates into equation \eqref{eq-bigInequation-reordered} and taking the supremum then the expectation,
\begin{multline}
    \E\Bigg[\sup_{[0,T]} \frac{1}{2}\|U_\epsilon\|_{L_\epsilon^2}^2 + \int_0^{T} \|\nabla_3 U_\epsilon\|_{L_\epsilon^2}^2 dt \Bigg] \leq C \epsilon^2  \E\Bigg[\int_0^{T} \| v \|_{H^1}^4 + \| v\|_{H^2}^2 (1 + \|a\|_{L^\infty}^2 + \|v\|_{H^2}^{2}) dt \Bigg] \\
    + C \E\Bigg[ \int_0^{T} (\|v\|_{H^1}^{2} \|v\|_{H^2}^2 +1) (\|U_\epsilon\|_{L_\epsilon^2}^2 + \|\Theta_\epsilon\|_{L^2}^2) + \alpha_\sigma^4 \epsilon^2 \|(a - a^K) \nabla w\|_{L^2}^2 dt\Bigg] \\
    - \E\Big[\frac{1}{2} \int_0^{T} \big( a \nabla_3 I_{\alpha_\sigma} U_\epsilon, \nabla_3 I_{\alpha_\sigma} U_\epsilon \big)_{L_\epsilon^2} dt\Big] + \E\Big[J_5 + J_6 + \sup_{[0,T]} |\int_0^{\bcdot} \sum_k J_7^k d\beta^k|\Big]. \label{eq-U-eps-estimate}
\end{multline}
In the next substep, we establish an estimate on $\Theta_\epsilon$. The remaining term $\E\Big[J_5 + J_6 + \sup_{[0,T]} |\int_0^{\bcdot} \sum_k J_7^k d\beta^k|\Big]$, which is more challenging to estimate, will be treated in Step 3.

\bigskip

\noindent
\emph{Step 2.2: Estimate on $\Theta_\epsilon$.} Furthermore, applying It\={o}'s lemma to $\|\Theta_\epsilon\|_{L^2}^2$ and using again an approximation argument, we find
\begin{multline}
    \frac{1}{2} d_t \|\Theta_\epsilon\|_{L^2}^2 + \|\Theta_{\epsilon}\|_{H^1}^2 dt + \frac{1}{2} (a \nabla_3 \Theta_\epsilon, \nabla_3 \Theta_\epsilon)_{L^2} dt\\
    = \frac{1}{2} \sum_k \|\phi_k \bcdot \nabla_3 \Theta_\epsilon \|_{L^2}^2 dt - (U_\epsilon \bcdot \nabla \theta, \Theta_\epsilon)_{L^2} dt,
\end{multline}
since $u^*,u_\epsilon^*$ and $\phi_k$ are divergence-free. Moreover,
\begin{align}
    (U_\epsilon \bcdot \nabla \theta, \Theta_\epsilon)_{L^2} &= ((\nabla_H \bcdot V_\epsilon) \theta, \Theta_\epsilon)_{L^2} + (V_\epsilon \bcdot \nabla_H \Theta_\epsilon, \theta)_{L^2} + (W_\epsilon \partial_z \theta, \Theta_\epsilon)_{L^2}\\
    &\leq \|V_\epsilon\|_V \|\theta\|_{H^2} \|\Theta_\epsilon\|_{L^2} + \|V_\epsilon\|_H \|\theta\|_{H^2} \|\Theta_\epsilon\|_{H^1} +  \|V_\epsilon\|_V \|\Theta_\epsilon\|_{L^3} \|\partial_z \theta\|_{L^3} \nonumber\\
    &\leq C \|\theta\|_{H^2}^2(1+\|\theta\|_{H^1}^2) (\|\Theta_\epsilon\|_{L^2}^2 + \|V_\epsilon\|_H^2 ) + \xi (\|V_\epsilon\|_V^2 + \|\Theta_\epsilon\|_{H^1}^2). \nonumber
\end{align}
Therefore,
\begin{multline}
    \E\Big[ \frac{1}{2} \sup_{[0,T]} \|\Theta_\epsilon\|_{L^2}^2 + \int_0^{T} \|\Theta_{\epsilon}\|_{H^1}^2 dt\Big] \leq C \E\Big[\int_0^{T} \|\theta\|_{H^2}^2(1+\|\theta\|_{H^1}^2) (\|\Theta_\epsilon\|_{L^2}^2 + \|V_\epsilon\|_H^2 ) dt \Big] \\
    + \xi \E\Big[ \int_0^{T} \Big(\sum_k \|\phi_k\|_{H^2}^2\Big) \|\nabla_3 \Theta_\epsilon \|_{L^2}^2 + (\|V_\epsilon\|_V^2 + \|\Theta_\epsilon\|_{H^1}^2) dt \Big],
\end{multline}
and thus,
\begin{equation}
    \E\Big[ \frac{1}{2} \sup_{[0,T]} \|\Theta_\epsilon\|_{L^2}^2 + \int_0^{T} \|\Theta_{\epsilon}\|_{H^1}^2 dt\Big] \leq C \E\Big[\int_0^{T} \|\theta\|_{H^2}^2(1+\|\theta\|_{H^1}^2) (\|\Theta_\epsilon\|_{L^2}^2 + \|V_\epsilon\|_H^2 )dt \Big]. \label{eq-Theta-eps-estimate}
\end{equation}
\emph{Step 3: Proof of Theorem \ref{theorem-cvg-weak}.} Now we combine the previous estimates on $U_\epsilon$ and $\Theta_\epsilon$: by summing equations \eqref{eq-U-eps-estimate} and \eqref{eq-Theta-eps-estimate}, we infer
\begin{multline}
    \E\Bigg[\sup_{[0,T]} \frac{1}{2}\|U_\epsilon\|_{L_\epsilon^2}^2 + \int_0^{T} \|\nabla_3 U_\epsilon\|_{L_\epsilon^2}^2 dt \Bigg] +  \E\Big[ \frac{1}{2} \sup_{[0,T]} \|\Theta_\epsilon\|_{L^2}^2 + \int_0^{T} \|\Theta_{\epsilon}\|_{H^1}^2 dt\Big] \\
    \leq C \epsilon^2  \E\Bigg[\int_0^{T} \|v \|_{H^1}^4 + \| v\|_{H^2}^2 (1 + \|a\|_{L^\infty}^2) dt + \|v\|_{H^1}^{2} \|v\|_{H^2}^{2} + \alpha_\sigma^4 \|(a - a^K) \nabla w\|_{L^2}^2 dt \Bigg]  \\
    + C \E\Bigg[ \int_0^{T} (\|\theta\|_{H^1}^{2} \|\theta\|_{H^2}^2 + \|v\|_{H^1}^{2} \|v\|_{H^2}^2 +1) (\|U_\epsilon\|_{L_\epsilon^2}^2 + \|\Theta_\epsilon\|_{L_\epsilon^2}^2) dt\Bigg] \\
    - \E\Big[\frac{1}{2} \int_0^{T} \big( a \nabla_3 I_{\alpha_\sigma} U_\epsilon, \nabla_3 I_{\alpha_\sigma} U_\epsilon \big)_{L_\epsilon^2} dt\Big] + \E\Big[J_5 + J_6 + \sup_{[0,T]} |\int_0^{\bcdot} \sum_k J_7^k d\beta^k|\Big]. \label{eq-before-cov-estimate}
\end{multline}
The terms left to estimate are $J_5$, $J_6$ and $\sup_{[0,T]} |\int_0^{\bcdot} \sum_k J_7^k d\beta^k|$. We begin with the latter, since we can readily bound it using the Burkholder-Davis-Gundy inequality. Thus, we find
\begin{multline}
    \E \Bigg[ \sup_{[0,T]} |\int_0^{\bcdot} \sum_k J_7^k d\beta^k| \Bigg] \\
    \leq C \E \Big( \sum_k \int_0^{T} \Big(\alpha_\sigma\epsilon^2 \Big(( \phi_k \bcdot \nabla_3) W_\epsilon, w \Big)_{L^2} + \alpha_\sigma \epsilon^2 \big(\Delta \phi_k^z, W_\epsilon \big)_{L^2} - (\partial_z \pi_k, W_\epsilon)_{L^2}\Big)^2 dt \Big)^{1/2}\\
    \leq C \alpha_\sigma \epsilon^2 \E \Big(\sum_k \int_0^{T} \Big( (I- K*)( \phi_k \bcdot \nabla_3) w, W_\epsilon \Big)_{L^2}^2 dt \Big)^{1/2}\\
    \leq C \alpha_\sigma \epsilon^2 \E \Big(\int_0^{T} \Big(\sum_k \|(I- K*)( \phi_k \bcdot \nabla_3) w \|_{L^2}^2 \Big) \| W_\epsilon \|_{L^2}^2 dt \Big)^{1/2}\\
    \leq C \alpha_\sigma \epsilon^2 \E  \Bigg[\sup_{[0,T]} \| W_\epsilon \|_{L^2} \Big( \int_0^{T} \Big(\sum_k \|(I- K*)( \phi_k \bcdot \nabla_3) w \|_{L^2}^2 \Big)  dt \Big)^{1/2}\Bigg]\\
    \leq C \epsilon^2 \E \sup_{[0,T]} \| W_\epsilon \|_{L^2}^2 + C \alpha_\sigma^2 \epsilon^2 \E \int_0^{T} \Big(\sum_k \|(I- K*)( \phi_k \bcdot \nabla_3) w \|_{L^2}^2 \Big)  dt,
\end{multline}
where we also used Young's inequality. Finally, we estimate $\E[J_5 + J_6]$, which can be interpreted as a covariation term stemming from It\={o}'s lemma. First we rewrite $J_5 + J_6$ as follows,
\begin{align*}
    2&(J_5 + J_6) = \sum_k \|\mathbf{P}_\epsilon^H (\phi_k \bcdot \nabla I_{\alpha_\sigma} u_\epsilon) - \mathbf{P}(\phi_k \bcdot \nabla v)  - \nabla_H \tilde{\pi}_k \|^2_{L^2}\\
    &+ \epsilon^2 \sum_k \|\mathbf{P}_\epsilon^z (\phi_k \bcdot \nabla I_{\alpha_\sigma} u_\epsilon) - \int_z^1 \nabla_H \bcdot (\reallywidetilde{\phi_k \bcdot \nabla v})  - \int_z^1 \Delta_H \tilde{\pi}_k \|^2_{L^2}\\
    &= \sum_k \Big\|\mathbf{P}_\epsilon\Big[
        \phi_k \bcdot \nabla I_{\alpha_\sigma} u_\epsilon - \begin{pmatrix}
        \mathbf{P}(\phi_k \bcdot \nabla v)\\
        \int_z^1 \nabla_H \bcdot (\reallywidetilde{\phi_k \bcdot \nabla v})
    \end{pmatrix} - \begin{pmatrix}
        \nabla_H \tilde{\pi}_k\\
        \int_z^1 \Delta_H \tilde{\pi}_k
    \end{pmatrix} \Big] \Big\|^2_{L_\epsilon^2}\\
    &= \sum_k \Big\|\mathbf{P}_\epsilon\Big[
        \phi_k \bcdot \nabla I_{\alpha_\sigma} U_\epsilon + \begin{pmatrix}
        \phi_k \bcdot \nabla v - \mathbf{P}(\phi_k \bcdot \nabla v)\\
        \alpha_\sigma \phi_k \bcdot \nabla w - \int_z^1 \nabla_H \bcdot (\reallywidetilde{\phi_k \bcdot \nabla v})
    \end{pmatrix} - \begin{pmatrix}
        \nabla_H \tilde{\pi}_k\\
        \int_z^1 \Delta_H \tilde{\pi}_k
    \end{pmatrix} \Big] \Big\|^2_{L_\epsilon^2}\\
    &= \sum_k \Big\|\mathbf{P}_\epsilon\Big[
        \phi_k \bcdot \nabla I_{\alpha_\sigma} U_\epsilon + \begin{pmatrix}
        (\mathbf{I}^{2D} - \mathbf{P}^{2D}) \overline{\phi_k \bcdot \nabla v}\\
        \alpha_\sigma \phi_k \bcdot \nabla w - w (\reallywidetilde{\phi_k \bcdot \nabla v})
    \end{pmatrix} - \begin{pmatrix}
        \nabla_H \tilde{\pi}_k\\
        \int_z^1 \Delta_H \tilde{\pi}_k
    \end{pmatrix} \Big] \Big\|^2_{L_\epsilon^2}\\
    &= \sum_k \Big\|\mathbf{P}_\epsilon\Big[
        \phi_k \bcdot \nabla I_{\alpha_\sigma} U_\epsilon + \begin{pmatrix}
        - \nabla_H \tilde{\pi}_k\\
        \alpha_\sigma \phi_k \bcdot \nabla w - w (\reallywidetilde{\phi_k \bcdot \nabla v}) - \int_z^1 \Delta_H \tilde{\pi}_k
    \end{pmatrix} \Big] \Big\|^2_{L_\epsilon^2}\\
    &\leq (a \nabla I_{\alpha_\sigma} U_\epsilon, \nabla I_{\alpha_\sigma} U_\epsilon)_{L_\epsilon^2} + \sum_k \|\mathbf{P}_\epsilon \begin{pmatrix}
        - \nabla_H \tilde{\pi}_k\\
        \alpha_\sigma \phi_k \bcdot \nabla w - w (\reallywidetilde{\phi_k \bcdot \nabla v}) - \int_z^1 \Delta_H \tilde{\pi}_k
    \end{pmatrix} \|_{L_\epsilon^2}^2\\
    &+ \Big(\phi_k \bcdot \nabla V_\epsilon, \mathbf{P}_\epsilon^H[- \nabla_H \tilde{\pi}_k]\Big)_{L^2} + \alpha_\sigma \epsilon^2 \Big(\phi_k \bcdot \nabla W_\epsilon, \mathbf{P}_\epsilon^z[\alpha_\sigma \phi_k \bcdot \nabla w - w (\reallywidetilde{\phi_k \bcdot \nabla v}) - \int_z^1 \Delta_H \tilde{\pi}_k]\Big)_{L^2} \\
    &\leq (a \nabla I_{\alpha_\sigma} U_\epsilon, \nabla I_{\alpha_\sigma} U_\epsilon)_{L_\epsilon^2} + \xi \| V_\epsilon\|_V^2 + C_\xi \sum_k \|\phi_k\|_{L^\infty}^2 \|\mathbf{P}_\epsilon^H \nabla_H \tilde{\pi}_k\|_{L^2}^2\\
    & + \xi \epsilon^2 \| W_\epsilon\|_V^2 + C_\xi \alpha_\sigma^2 \epsilon^2 \sum_k \|\phi_k\|_{L^\infty}^2 \|\mathbf{P}_\epsilon^z[\alpha_\sigma \phi_k \bcdot \nabla w - w (\reallywidetilde{\phi_k \bcdot \nabla v}) - \int_z^1 \Delta_H \tilde{\pi}_k]\|_{L^2}^2 \\
    & + \sum_k \|\mathbf{P}_\epsilon \begin{pmatrix}
        - \nabla_H \tilde{\pi}_k\\
        \alpha_\sigma \phi_k \bcdot \nabla w - w (\reallywidetilde{\phi_k \bcdot \nabla v}) - \int_z^1 \Delta_H \tilde{\pi}_k
    \end{pmatrix} \|_{L_\epsilon^2}^2.
\end{align*}
Using the noise regularity assumption \ref{smoothness-noise}, we infer
\begin{multline}
    (J_5 + J_6) \leq \frac{1}{2}(a \nabla I_{\alpha_\sigma} U_\epsilon, \nabla I_{\alpha_\sigma} U_\epsilon)_{L_\epsilon^2} + \xi \|U_\epsilon\|_{V_\epsilon}^2\\
    + C_\xi (\sup_k \|\phi_k\|_{L^\infty}^2) (1+\alpha_\sigma^2) \sum_k \|\mathbf{P}_\epsilon \begin{pmatrix}
        - \nabla_H \tilde{\pi}_k\\
        \alpha_\sigma \phi_k \bcdot \nabla w - w (\reallywidetilde{\phi_k \bcdot \nabla v}) - \int_z^1 \Delta_H \tilde{\pi}_k
    \end{pmatrix} \|_{L_\epsilon^2}^2
\end{multline}
Now, we estimate the last term on the RHS,
\begin{align*}
    \sum_k &\|\mathbf{P}_\epsilon \begin{pmatrix}
        - \nabla_H \tilde{\pi}_k\\
        \alpha_\sigma \phi_k \bcdot \nabla w - w (\reallywidetilde{\phi_k \bcdot \nabla v}) - \int_z^1 \Delta_H \tilde{\pi}_k
    \end{pmatrix}\|_{L_\epsilon^2}^2\\
    &= \sum_k \|\mathbf{P}_\epsilon \begin{pmatrix}
        - \nabla_H \tilde{\pi}_k\\
        (I-K*)[\alpha_\sigma \phi_k \bcdot \nabla w - w (\reallywidetilde{\phi_k \bcdot \nabla v})] - \int_z^1 \Delta_H \tilde{\pi}_k -\frac{1}{\epsilon^2} \partial_z \tilde{\pi}_k
    \end{pmatrix} \|_{L_\epsilon^2}^2\\
    &= \sum_k \|\mathbf{P}_\epsilon \begin{pmatrix}
        0\\
        (I-K*)[\alpha_\sigma \phi_k \bcdot \nabla w - w (\reallywidetilde{\phi_k \bcdot \nabla v})] - \int_z^1 \Delta_H \tilde{\pi}_k
    \end{pmatrix} \|_{L_\epsilon^2}^2\\
    &\leq C\epsilon^2 \sum_k \|\tilde{\pi}_k\|_{H^2}^2 + C \epsilon^2 \sum_k \|(I-K*)[\alpha_\sigma \phi_k \bcdot \nabla w - w (\reallywidetilde{\phi_k \bcdot \nabla v})]\|_{L^2}^2\\
    &\leq C \epsilon^2 (1+\|v\|_{H^2}^2 +\alpha_\sigma^2 \sum_k \|(I-K*)[ \phi_k \bcdot \nabla w ]\|_{L^2}^2),
\end{align*}
where we used the fact that $\mathbf{P}_\epsilon (\nabla_H \tilde{\pi}_k \: \frac{1}{\epsilon^2} \partial_z \tilde{\pi}_k)^\tr = \mathbf{P}_\epsilon \nabla_\epsilon \tilde{\pi}_k = 0$. Hence,
\begin{multline}
    \E\Bigg[\sup_{[0,T]} \frac{1}{2}\|U_\epsilon\|_{L_\epsilon^2}^2 + \int_0^{T} \|\nabla_3 U_\epsilon\|_{L_\epsilon^2}^2 dt \Bigg] +  \E\Big[ \frac{1}{2} \sup_{[0,T]} \|\Theta_\epsilon\|_{L^2}^2 + \int_0^{T} \|\Theta_{\epsilon}\|_{H^1}^2 dt\Big] \\
    \leq C \epsilon^2  \E\Bigg[\int_0^{T} \|v \|_{H^1}^4 + \| v\|_{H^2}^2 (1 + \|a\|_{L^\infty}^2) dt + \|v\|_{H^1}^{2} \|v\|_{H^2}^{2} dt \Bigg]  \\
    + C \E\Bigg[ \int_0^{T} (\|\theta\|_{H^1}^{2} \|\theta\|_{H^2}^2 + \|v\|_{H^1}^{2} \|v\|_{H^2}^2 +1) (\|U_\epsilon\|_{L_\epsilon^2}^2 + \|\Theta_\epsilon\|_{L_\epsilon^2}^2) dt\Bigg] \\
    + C (1+\alpha_\sigma^4) \epsilon^2 \E\Bigg[ \int_0^{T} \sum_k \|(I-K*)[ \phi_k \bcdot \nabla w ]\|_{L^2}^2 + \|(a - a^K) \nabla w\|_{L^2}^2 dt \Bigg]. \label{eq-estimateForThm1}
\end{multline}
As mentioned at the beginning of the proof, we can rewrite the estimate \eqref{eq-estimateForThm1} in terms of $\eta$ and $\zeta$, which yields,
\begin{multline}
    \E\big[\mathcal{E}_\epsilon^0(\eta,\zeta)\big]
    \leq C \epsilon^2  \E\Bigg[\int_\eta^\zeta \|v \|_{H^1}^4 + \| v\|_{H^2}^2 (1 + \|a\|_{L^\infty}^2) dt + \|v\|_{H^1}^{2} \|v\|_{H^2}^{2} dt \Bigg]  \\
    + C \E\Bigg[ \int_\eta^\zeta (\|\theta\|_{H^1}^{2} \|\theta\|_{H^2}^2 + \|v\|_{H^1}^{2} \|v\|_{H^2}^2 +1) (\|U_\epsilon\|_{L_\epsilon^2}^2 + \|\Theta_\epsilon\|_{L_\epsilon^2}^2) dt\Bigg] \\
    + C (1+\alpha_\sigma^4) \epsilon^2 \E\Bigg[ \int_0^{T} \sum_k \|(I-K*)[ \phi_k \bcdot \nabla w ]\|_{L^2}^2 + \|(a - a^K) \nabla w\|_{L^2}^2 dt \Bigg].
\end{multline}
Thus, we conclude by the use of the stochastic Grönwall lemma of \cite{GHZ_2009}, since the following holds,
\begin{equation}
    \E \int_{0}^{T \wedge \tau_\delta} \|v\|_{H^1}^4 + \|v\|_{H^1}^2 \|v\|_{H^2}^2 + \|\theta\|_{H^1}^4 + \|\theta\|_{H^1}^2 \|\theta\|_{H^2}^2 dt \leq C_\delta. \label{eq-mom4-H1-H2-bound}
\end{equation}
Furthermore, $\E\big[\mathcal{E}_\epsilon^0(0,T \wedge \tau_\delta) \big]$ converges to $0$ at order $\alpha_\sigma \epsilon$ whenever $\alpha_\sigma = o(\epsilon^{-1/2})$. In such case, $\mathcal{E}_\epsilon^0(0, T \wedge \tau_\delta)$ converges to $0$ in probability. By taking the limit $\delta \rightarrow \infty$, we deduce that $\mathcal{E}_\epsilon^0(0, T) \rightarrow 0$ in probability.

\CQFD

\section{Proof of Theorem \ref{theorem-cvg-strong}} \label{sec-thm-cvg-strong}

In this section, we commute to "fully periodic" boundary conditions to prove Theorem \ref{theorem-cvg-strong}. Notice that all the results presented previously hold in this case, the proof being identical. In the following, we use the same notations as in Section \ref{sec-thm-cvg-weak}. The proof of Theorem \ref{theorem-cvg-strong} is split into three steps: first we derive an energy estimate on the primitive equations (Lemma \ref{lemma-energy-estimate-PE}), then we establish the energy estimate stated in Theorem \ref{theorem-cvg-strong}.
\subsection{Estimate in $H^2-H^3$ for the primitive equations}
In this subsection, we prove the following lemma, which stands as an improved energy estimate on the primitive equations in periodic boundary conditions.
\begin{lemma}\label{lemma-energy-estimate-PE}
    Under the assumptions of Theorem \ref{theorem-cvg-strong}, let $T,\delta>0$. There exists a constant $C>0$, such that 
    \begin{equation}
        \E \sup_{[0, \tau_\delta \wedge T]} \|v\|_{\mathcal{D}(A)}^2 + \E \int_0^{\tau_\delta \wedge T} \|v\|_{\mathcal{D}(A^{3/2})}^2 dt \leq C, \label{eq:improved-estimate-PE}
    \end{equation}
    where $\tau_\delta$ is defined as in Theorem \ref{theorem-cvg-weak}.
\end{lemma}
\emph{Proof:} Essentially, we adapt the method proposed in \cite{LT2019} to our stochastic setting. Remind that, in the problem ($\Pi_\epsilon^{PE}$), the horizontal momentum equation reads,
\begin{align}
    d_t v + \mathbf{P}\Big[\Big((v^* dt + \sigma dW_t^H) &\bcdot \nabla_H\Big) v + (w^* dt + \sigma dW_t^z) \partial_z v\Big] + A (v dt + \sigma dW_t^H) \nonumber\\
    &+ \nabla_H (\tilde{p}\:dt + d\tilde{p}_t^{\sigma}) - \frac{1}{2} \mathbf{P}[\nabla_3 \bcdot(a \nabla_3 v)] dt= 0,
\end{align}
where $\tilde{p}$ and $d\tilde{p}_t^\sigma$ denote the baroclinic components of $p$ and $dp_t^\sigma$, respectively. By It\={o}'s lemma, and using the fact that, with periodic boundary conditions, $A^2v = (-\Delta)^2 v$, we infer that
\begin{align}
    d_t \|v\|_{\mathcal{D}(A)}^2 &+ \Bigg(\Big((v^* dt + \sigma dW_t^H) \bcdot \nabla_H\Big) v + (w^* dt + \sigma dW_t^z) \partial_z v, (-\Delta )^2 v \Bigg)_{L^2} \nonumber\\
    & + \underbrace{\Big((-\Delta ) (v dt + \sigma dW_t^H), (-\Delta )^2 v \Big)_H}_{= \|v\|_{\mathcal{D}(A^{3/2})}^2 dt + ((-\Delta)^{3/2} \sigma dW_t^H, (-\Delta )^{3/2} v)_H} + \Big(\nabla_H (\tilde{p}\:dt + d\tilde{p}_t^{\sigma}) - \frac{1}{2} \nabla_3 \bcdot(a \nabla_3 v) dt,  (-\Delta )^2 v \Big)_{L^2} \nonumber\\
    &\leq \frac{1}{2} \sum_k \|\mathbf{P}[(\phi_k \bcdot \nabla_3) v] + \nabla_H \tilde{\pi}_k - \Delta \phi_k^H \|_{\mathcal{D}(A)}^2.
\end{align}
Moreover, we have
\begin{equation}
    \Bigg( \Big(v^* \bcdot \nabla_H\Big) v + w^* \partial_z v, (-\Delta )^2 v \Bigg)_{L^2} \leq - \Big( \nabla [(v^* \bcdot \nabla_H) v + w^* \partial_z v], \nabla \Delta v \Big)_{L^2},
\end{equation}
and, for $i \in \{x,y,z\}$ and $k \in \{x,y\}$, using Einstein's notation (with $j \in \{x,y,z\}$),
\begin{equation*}
    \Big[\nabla [(u^* \bcdot \nabla) v]\Big]_{ik} = \partial_i (u_j^* \partial_j v_k) = (\partial_i u_j^*) (\partial_j v_k) + u_j^* \partial_{ij} v_k.
\end{equation*}
Hence, by the Ladyzhenskaya-type relation of \cite{LT2019}, and Young's inequality,
\begin{align}
    \Big( \nabla [(v^* \bcdot \nabla_H) v + w^* \partial_z v], \nabla \Delta v \Big)_{L^2} &\leq C(1 + \|v\|_{H^1}^{1/2} \|v\|_{H^2} \|v\|_{H^3}^{3/2})\\
    &\leq C_\xi(1 + \|v\|_{H^1}^2 \|v\|_{H^2}^4) + \xi \|v\|_{H^3}^2 \nonumber
\end{align}
In addition,
\begin{align}
    \Big(\nabla_H \tilde{p}, (-\Delta )^2 v \Big)_{L^2} &= \Big(\nabla_H \int_z^1 [\theta - \frac{\alpha_\sigma^2 \epsilon^2}{2} \nabla_3 \bcdot (a^K \nabla_3 w)]dz', (-\Delta )^2 v \Big)_{L^2}\\
    &\leq (\|\theta\|_{H^2} + \|\nabla_3\bcdot(a^K \nabla w)\|_{H^2}) \|v\|_{\mathcal{D}(A^{3/2})}\\
    &\leq C_\xi (\|\theta\|_{H^2}^2 + \|v\|_{H^2}^2) + \xi \|v\|_{\mathcal{D}(A^{3/2})}^2,
\end{align}
using the fact that $K \in H^3(\mathcal{S},\R^3)$ and $a^K[\bcdot] = \sum_{k=1}^\infty \phi_k \mathcal{C}_K^* \mathcal{C}_K [\phi_k^\tr \bcdot]$, with $\mathcal{C}_K[\bcdot] = K* [\bcdot]$. Furthermore, using Einstein's notation again with $i,j,k \in \{x,y\}$ and $l,m \in \{x,y,z\}$,
\begin{align}
    -\frac{1}{2}\Big(\nabla_3 \bcdot(a \nabla_3 v) dt, \mathbf{P} (-\Delta )^2 v \Big)_{L^2} &= -\frac{1}{2} \Big(\partial_i (a_{ij} \partial_j v_k), \partial_{ll} \partial_{mm} v_k \Big)_{L^2} \\
    &= \frac{1}{2} \Big(\partial_{lm} (a_{ij} \partial_j v_k), \partial_{ilm} v_k \Big)_{L^2} \nonumber\\
    &\leq \frac{1}{2} \Big(a_{ij} \partial_{jlm} v_k, \partial_{ilm} v_k \Big)_{L^2} + C \|v\|_{\mathcal{D}(A)} \|v\|_{\mathcal{D}(A^{3/2})}, \nonumber
\end{align}
where $a$ is sufficiently smooth by \eqref{smoothness-noise}. For the covariation term, we remark that $A^v \mathbf{P} = \mathbf{P} (-\Delta)$ on $L^2(\mathbbm{T}_3,\R^2)$. Thus,
\begin{align}
    \sum_k &\|\mathbf{P}[(\phi_k \bcdot \nabla_3) v] + \nabla_H \tilde{\pi}_k - \Delta \phi_k^H \|_{\mathcal{D}(A)}^2 \nonumber\\
    &\leq \sum_k \|\mathbf{P} (-\Delta)[(\phi_k \bcdot \nabla_3) v + \nabla_H \tilde{\pi}_k - \Delta \phi_k^H]\|_H^2 \nonumber\\
    &\leq \sum_k \|A^v[(\phi_k \bcdot \nabla_3) v + \nabla_H \tilde{\pi}_k - \Delta \phi_k^H]\|_H^2 \nonumber\\
    &\leq C_\xi \Bigg(\|v\|_{\mathcal{D}(A)}^2 + \sum_k (\|\phi_k^H\|_{H^4}^2 + \|\phi_k\|_{L^\infty}^2 \|v\|_V^2 + \|\nabla_H \tilde{\pi}_k\|_{\mathcal{D}(A)}^2)\Bigg) \nonumber\\
    & \qquad \qquad \qquad \qquad \qquad \qquad \qquad \qquad + (a_{ij} \partial_{jlm} v_k, \partial_{ilm} v_k)_{L^2} + \xi \|v\|_{\mathcal{D}(A^{3/2})}^2,
\end{align}
for $\xi > 0$. Gathering the previous estimates, integrating over time, then taking the supremum and the expectation, we get, for all stopping times $\eta,\zeta \leq \tau_\delta \wedge T$ such that $\eta<\zeta$,
\begin{multline}
    \E \sup_{[\eta, \zeta]} \|v\|_{\mathcal{D}(A)}^2 + \E \int_\eta^\zeta \|v\|_{\mathcal{D}(A^{3/2})}^2 dt \leq C \E  \int_\eta^\zeta \Big(1+\|v\|_{V}^2\|v\|_{\mathcal{D}(A)}^4 + \|\theta\|_{H^2}^2 + \sum_k \|\nabla_H \tilde{\pi}_k\|_{\mathcal{D}(A)}^2\Big) dt \\
    + C \E \sup_{[\eta, \zeta]} \Big|\int_\eta^\zeta ((\sigma^H dW_t \bcdot \nabla) v - \Delta \sigma^H dW_t, \mathbf{P}(-\Delta)^2 v)_{L^2}\Big|.
\end{multline}
By the Burkholder-Davis-Gundy inequality, and using interpolation inequalities,
\begin{align}
    \E \sup_{[\eta, \zeta]} \Big|&\int_\eta^\zeta ((\sigma^H dW_t \bcdot \nabla_H) v - \Delta \sigma^H dW_t + \nabla_H d\tilde{p}_t^{\sigma}, \mathbf{P}(-\Delta)^2 v)_{L^2}\Big| \nonumber\\
    &\leq \E \Bigg[\Big(\sum_k \int_\eta^\zeta ((\phi_k^H \bcdot \nabla_H) v - \Delta \phi_k^H + \nabla_H \tilde{\pi}_k, \mathbf{P} (-\Delta)^2 v)_{L^2}^2 \Big)^{1/2}\Bigg] dt \nonumber\\
    &\leq C \E \Bigg[\Big(\sum_k \int_\eta^\zeta (1+\|\nabla[(\phi_k^H \bcdot \nabla_H) v - \Delta \phi_k^H + \nabla_H \tilde{\pi}_k]\|_{L^2}^2) \|v\|_{\mathcal{D}(A^{3/2})}^2 \Big)^{1/2}\Bigg] dt \nonumber\\
    &\leq C \E \Bigg[\Big(\int_\eta^\zeta (1 + \|v\|_{\mathcal{D}(A)}^2 + \sum_k \|\nabla_H \tilde{\pi}_k\|_{\mathcal{D}(A)}^2) \|v\|_{\mathcal{D}(A^{3/2})}^2 \Big)^{1/2}\Bigg] dt \nonumber\\
    &\leq C \E \Bigg[\Big(1+\sup_{[\eta,\zeta]} (\|v\|_{\mathcal{D}(A)}^2 + \sum_k \|\nabla_H \tilde{\pi}_k\|_{\mathcal{D}(A)}^2)\Big)^{1/2} \Big(\int_\eta^\zeta \|v\|_{\mathcal{D}(A^{3/2})}^2 \Big)^{1/2}\Bigg] dt \nonumber\\
    &\leq C_\xi \E \Bigg[1+\sup_{[\eta,\zeta]} (\|v\|_{\mathcal{D}(A)}^2 + \sum_k \|\nabla_H \tilde{\pi}_k\|_{\mathcal{D}(A)}^2)\Bigg]dt + \xi \E \Bigg[\int_\eta^\zeta \|v\|_{\mathcal{D}(A^{3/2})}^2\Bigg] dt.
\end{align}
In addition, since $K \in H^3(\mathcal{S},\R)$, we infer that
\begin{equation}
    \sum_k \|\nabla_H \tilde{\pi}_k\|_{\mathcal{D}(A)}^2 \leq C (1+ \|v\|_{\mathcal{D}(A)}^2).
\end{equation}
Hence,
\begin{equation}
    \E \sup_{[\eta, \zeta]} \|v\|_{\mathcal{D}(A)}^2 + \E \int_\eta^\zeta \|v\|_{\mathcal{D}(A^{3/2})}^2 dt \leq C \E \int_\eta^\zeta \Big(1+\|v\|_{V}^2\|v\|_{\mathcal{D}(A)}^2 \Big) \|v\|_{\mathcal{D}(A)}^2 + \|\theta\|_{H^2}^2 dt.
\end{equation}
This last result allows to conclude by the stochastic Grönwall lemma.
\CQFD

\subsection{Estimate in $H^1-H^2$ for the error}

In this subsection, we establish the results stated in Theorem \ref{theorem-cvg-strong}, with an emphasis on the estimate \eqref{eq-bound-thm2}.

In the following, we write $\tau_\epsilon$ in place of $\tau_\epsilon^{SNS}$ to economize notation. Let $\tau$ a stopping time such that $0 < \tau < \tau_\epsilon$. We further make an abuse of notation by writing $U_\epsilon$, $V_\epsilon$ and $W_\epsilon$ in place of $(U_\epsilon)_{\bcdot \wedge \tau}$, $(V_\epsilon)_{\bcdot \wedge \tau}$ and $(W_\epsilon)_{\bcdot \wedge \tau}$, respectively. First, we state the evolution equations of $V_\epsilon$ and $W_\epsilon$, which are meaningful since $u_\epsilon$ is a local-in-time strong solution to ($\Pi_\epsilon^{SNS}$),
\begin{multline*}
    d_t V_\epsilon + \Big((u_\epsilon^* dt + \sigma dW_t) \bcdot \nabla_H\Big) V_\epsilon + \Big(U_\epsilon \bcdot \nabla\Big) v dt - \frac{1}{2} \nabla_3 \bcdot (a \nabla_3 V_\epsilon) dt \\
    - \Delta V_\epsilon dt +\nabla_H(P_\epsilon dt + dP_t^{\sigma, \epsilon}) = 0,
\end{multline*}
and
\begin{multline*}
    d_t W_\epsilon + \Big((u_\epsilon^* dt + \alpha_\sigma \sigma dW_t) \bcdot \nabla \Big) W_\epsilon + \Big(U_\epsilon \bcdot \nabla\Big) w dt - \frac{\alpha_\sigma^2}{2} \nabla_3 \bcdot (a \nabla_3 W_\epsilon) dt - \Delta W_\epsilon \: dt\\
    + \frac{1}{\epsilon^2} \partial_z(P_\epsilon dt + dP_t^{\sigma, \epsilon}) + \frac{1}{\epsilon^2} \frac{\rho_\epsilon - \rho}{\rho_0}g dt = - (I-K*)(\alpha_\sigma \sigma dW_t \bcdot \nabla) w +  \frac{\alpha_\sigma^2}{2} \nabla_3 \bcdot ((a-a^K) \nabla w) dt.
\end{multline*}
Applying It\={o}'s lemma, and using an approximation argument that is similar to the one of Section \ref{sec-thm-cvg-weak}, with $\|U_\epsilon\|_{V_\epsilon}^2$, we deduce that
\begin{multline}
    \frac{1}{2} d_t \|U_\epsilon\|_{V_\epsilon}^2 + \|U_\epsilon\|_{\mathcal{D}_\epsilon(A)}^2 dt - \frac{1}{2} \Big(\nabla_3 \bcdot (a \nabla_3 I_{\alpha_\sigma} U_\epsilon), A I_{\alpha_\sigma} U_\epsilon \Big)_{L_\epsilon^2} dt \leq - \Bigg(\Big(U_\epsilon dt \bcdot \nabla_3\Big) U_\epsilon, A U_\epsilon \Bigg)_{L_\epsilon^2} \\
    - \Bigg(\Big((u^* dt + U_\epsilon dt) \bcdot \nabla_3\Big) U_\epsilon, A U_\epsilon \Bigg)_{L_\epsilon^2} - (g \Theta_\epsilon, A W_\epsilon)_{L^2} dt - \frac{\alpha_\sigma^2 \epsilon^2 }{2} \Big( \nabla_3 \bcdot ((a-a^K) \nabla w) dt, A W_\epsilon \Big)_{L^2}\\
    + \frac{1}{2} \sum_k \Big\| \mathbf{P}_\epsilon \Big[\phi_k \bcdot \nabla_3 I_{\alpha_\sigma} U_\epsilon + \begin{pmatrix}
        0 \\ \alpha_\sigma (I-K*)(\phi_k \bcdot \nabla_3)w
    \end{pmatrix}\Big] \Big\|_{V_\epsilon}^2 dt\\
    - \alpha_\sigma \epsilon^2 \Big((I-K*)( \sigma dW_t \bcdot \nabla) w\Big)_{L^2} - \Bigg(\Big(\sigma dW_t \bcdot \nabla_3\Big) I_{\alpha_\sigma} u, A U_\epsilon \Bigg)_{L_\epsilon^2}\\
    =: (J_1 + J_2 + J_3 + J_4 + J_5) dt + \sum_k (J_6^k + J_7^k) d\beta^k. \label{eq:big-equation-h1h2ErrorEstimate}
\end{multline}
First, we remark the following about the last term on the LHS of \eqref{eq:big-equation-h1h2ErrorEstimate},
\begin{multline}
    \frac{1}{2} \Big(\nabla_3 \bcdot (a \nabla_3 I_{\alpha_\sigma} U_\epsilon), A I_{\alpha_\sigma} U_\epsilon \Big)_{L_\epsilon^2} \leq - \frac{1}{2} \Big(a \nabla_3 \nabla_3 I_{\alpha_\sigma} U_\epsilon, \nabla_3 \nabla_3 I_{\alpha_\sigma} U_\epsilon \Big)_{L_\epsilon^2} \\
    + C_\xi (1+\alpha_\sigma^2)^2 \|U_\epsilon\|_{V_\epsilon}^2 + \xi \| U_\epsilon\|_{\mathcal{D}_\epsilon(A)}^2.
\end{multline}
Next, we develop the expression of $J_1 + J_2$,
\begin{multline}
    J_1 + J_2 = -\Bigg(\Big(U_\epsilon \bcdot \nabla_3\Big) U_\epsilon + \Big(u^* \bcdot \nabla_3\Big) U_\epsilon + \Big(U_\epsilon \bcdot \nabla_3\Big) u, A U_\epsilon \Bigg)_{L_\epsilon^2} \\
    = -\Bigg(\Big(U_\epsilon \bcdot \nabla_3\Big) V_\epsilon + \Big(u^* \bcdot \nabla_3\Big) V_\epsilon + \Big(U_\epsilon \bcdot \nabla_3\Big) v, A V_\epsilon \Bigg)_{L^2} \\
    - \epsilon^2 \Bigg(\Big(U_\epsilon \bcdot \nabla_3\Big) W_\epsilon + \Big(u^* \bcdot \nabla_3\Big) W_\epsilon + \Big(U_\epsilon \bcdot \nabla_3\Big) w, A W_\epsilon \Bigg)_{L^2}.
\end{multline}
Remark that, using again the Ladyzhenskaya-type relation of \cite{LT2019},
\begin{multline}
    \Bigg(\Big(U_\epsilon \bcdot \nabla_3\Big) V_\epsilon + \Big(u^* \bcdot \nabla_3\Big) V_\epsilon + \Big(U_\epsilon \bcdot \nabla_3\Big) v, A V_\epsilon \Bigg)_{L^2}    \leq \Big[\|\nabla_3 V_\epsilon\|_{L^2} \|\Delta_3 V_\epsilon\|_{L^2} \\
    + \|\nabla_3 V_\epsilon\|_{L^2}^{1/2} \|\Delta_3 V_\epsilon\|_{L^2}^{1/2} \Big(\|\nabla_3 v^*\|_{L^2}^{1/2} \|\Delta_3 v^*\|_{L^2}^{1/2} +  \|\nabla_3 v\|_{L^2}^{1/2} \|\Delta_3 v\|_{L^2}^{1/2}\Big)\Big] \|V_\epsilon\|_{\mathcal{D}(A)},
\end{multline}
and
\begin{multline}
    \Bigg(\Big(U_\epsilon \bcdot \nabla_3\Big) W_\epsilon + \Big(u^* \bcdot \nabla_3\Big) W_\epsilon + \Big(U_\epsilon \bcdot \nabla_3\Big) w, A W_\epsilon \Bigg)_{L^2} \leq \Big[\|\nabla_3 V_\epsilon\|_{L^2}^{1/2} \|\Delta_3 V_\epsilon\|_{L^2}^{1/2} \|\nabla_3 W_\epsilon\|_{L^2}^{1/2} \|\Delta_3 W_\epsilon\|_{L^2}^{1/2}\\
    + \|\nabla_3 v^*\|_{L^2}^{1/2} \|\Delta_3 v^*\|_{L^2}^{1/2} \|\nabla_3 W_\epsilon\|_{L^2}^{1/2} \|\Delta_3 W_\epsilon\|_{L^2}^{1/2} + \|\nabla_3 V_\epsilon\|_{L^2}^{1/2} \|\Delta_3 V_\epsilon\|_{L^2}^{1/2} \|\nabla_3 w\|_{L^2}^{1/2} \|\Delta_3 w\|_{L^2}^{1/2}\Big] \|W_\epsilon\|_{H^2}.
\end{multline}
Hence,
\begin{multline}
    J_1 + J_2 \leq C_\xi \|V_\epsilon\|_{V}^2 \|V_\epsilon\|_{\mathcal{D}(A)}^2 + C_\xi (1+\|v\|_{V}^2)(1+ \|v\|_{\mathcal{D}(A)}^2) \|V_\epsilon\|_V^2\\
    + C_\xi \epsilon^2 (1+\|U_\epsilon\|_{V_\epsilon}^2)(1+ \|v\|_{H^3}^2 + \|V_\epsilon\|_{\mathcal{D}(A)}^2) +  \xi \|\nabla U_\epsilon\|_{V_\epsilon}^2.
\end{multline}
Additionally, for $J_3$,
\begin{equation}
    J_3 = - (g \Theta_\epsilon, A W_\epsilon)_{L^2} \leq \|\Theta_\epsilon\|_{V} \|W_\epsilon\|_{V} \leq \|\Theta_\epsilon\|_{V} \|V_\epsilon\|_{\mathcal{D}(A)} \leq C\|\Theta_\epsilon\|_{V}^2 + \xi \|V_\epsilon\|_{\mathcal{D}(A)}^2.
\end{equation}
Moreover, for $J_4$,
\begin{align}
    J_4 = - \frac{\alpha_\sigma^2 \epsilon^2}{2} \Big(\nabla_3 \bcdot ((a-a^K) \nabla &w), AW_\epsilon\Big)_{L_\epsilon^2} = \frac{\alpha_\sigma^2 \epsilon^2}{2} \|(a - a^K)\nabla w\|_{H^1} \|W_\epsilon\|_{\mathcal{D}(A)} \nonumber\\
    &\leq C_\xi \frac{\alpha_\sigma^4 \epsilon^2}{2} \|(a - a^K)\nabla w\|_{H^1}^2  + \xi \epsilon^2 \|W_\epsilon\|_{\mathcal{D}(A)}^2. 
\end{align}
Furthermore, for $J_5$,
\begin{align}
    J_5 &= \frac{1}{2} \sum_k \Big\| \mathbf{P}_\epsilon \Big[\phi_k \bcdot \nabla_3 I_{\alpha_\sigma} U_\epsilon + \begin{pmatrix}
        0 \\ \alpha_\sigma (I-K*)[(\phi_k \bcdot \nabla_3)w]
    \end{pmatrix} \Big] \Big\|_{V_\epsilon}^2\\
    &\leq C_\xi \alpha_\sigma^4 \epsilon^2\|v\|_{H^3}^2 (1+\|U_\epsilon\|_{V_\epsilon}^2) + \xi \|U_\epsilon\|_{\mathcal{D}_\epsilon(A)}^2 + \sum_k \frac{1}{2} (\phi_k \bcdot \nabla_3 \nabla_3I_{\alpha_\sigma} U_\epsilon, \phi_k \bcdot \nabla_3 \nabla_3 I_{\alpha_\sigma} U_\epsilon)_{L_\epsilon^2} \nonumber\\
    &\qquad \qquad \qquad \qquad \qquad \qquad \qquad \qquad \qquad + C_\xi \alpha_\sigma^2 \epsilon^2 \sum_k \|(I-K*)[(\phi_k \bcdot \nabla_3)w]\|_{H^1}^2. \nonumber
\end{align}
Hence, taking the integral, the supremum then the expectation of \eqref{eq:big-equation-h1h2ErrorEstimate} and gathering the previous estimates yields,
\begin{multline}
    \E \sup_{[\eta, \zeta]} \|U_\epsilon\|_{V_\epsilon}^2 + \E \int_{\eta}^{\zeta} \| U_\epsilon\|_{\mathcal{D}_\epsilon(A)}^2 dt \leq C \E \int_\eta^\zeta  \|U_\epsilon\|_{V_\epsilon}^2 \|U_\epsilon\|_{\mathcal{D}_\epsilon(A)}^2 dt + \xi \E\int_\eta^\zeta \|U_\epsilon\|_{\mathcal{D}_\epsilon(A)}^2 dt\\
    + C\epsilon^2 (1+ \alpha_\sigma^4)\E\int_\eta^\zeta (1+\|v\|_{H^3}^2)(1+\|(a - a^K)\nabla w\|_{H^1}^2 + \sum_k \|(I-K*)[(\phi_k \bcdot \nabla_3)w]\|_{H^1}^2)(1+\|U_\epsilon\|_{V_\epsilon}^2+  \|\Theta_\epsilon\|_{V}^2) dt \\
    + C \sum_k \E \sup_{[\eta, \zeta]} \Bigg|\int_\eta^\zeta (J_6^k + J_7^k) d\beta_k\Bigg|.
\end{multline}
In addition, by the Burkholder-Davis-Gundy inequality,
\begin{align}
    \E &\sup_{[\eta, \zeta]} \Bigg|\int_\eta^\zeta \Bigg( \mathbf{P}_\epsilon \Big[\phi_k \bcdot \nabla_3 I_{\alpha_\sigma} U_\epsilon + \begin{pmatrix}
        0 \\ \alpha_\sigma (I-K*)(\phi_k \bcdot \nabla_3)w
    \end{pmatrix}\Big], A I_{\alpha_\sigma} U_\epsilon \Bigg)_{L_\epsilon^2} d\beta_k\Bigg| \\
    &\leq \E \Bigg| \int_\eta^\zeta \Bigg( \mathbf{P}_\epsilon \Big[\phi_k \bcdot \nabla_3 I_{\alpha_\sigma} U_\epsilon + \begin{pmatrix}
        0 \\ \alpha_\sigma (I-K*)(\phi_k \bcdot \nabla_3)w
    \end{pmatrix}\Big], A I_{\alpha_\sigma} U_\epsilon \Bigg)_{L_\epsilon^2}^2 \Bigg|^{1/2} dt \nonumber\\
    &\leq \sqrt{1+\alpha_\sigma^2} \E \Bigg| \int_\eta^\zeta \|\phi_k\|_{H^4}^2 (1+ \|U_\epsilon\|_{V_\epsilon}^2 + \|(I-K*)(\phi_k \bcdot \nabla_3)w\|_{L^2}^2) \|A U_\epsilon\|_{L_\epsilon^2}^2 \Bigg|^{1/2} dt \nonumber\\
    &\leq C_\xi \epsilon^2 (1+\alpha_\sigma^4) \E \|\phi_k\|_{H^4}^2 \sup_{[\eta, \zeta]}  (1+ \|U_\epsilon\|_{V_\epsilon}^2 + \|(I-K*)(\phi_k \bcdot \nabla_3)w\|_{L^2}^2) + \xi \E \int_\eta^\zeta \|A U_\epsilon\|_{L_\epsilon^2}^2 dt. \nonumber
\end{align}
Therefore,
\begin{multline}
    \E \sup_{[\eta, \zeta]} \|U_\epsilon\|_{V_\epsilon}^2 + \E \int_{\eta}^{\zeta} \| U_\epsilon\|_{\mathcal{D}_\epsilon(A)}^2 dt \leq C \E \int_\eta^\zeta  \|U_\epsilon\|_{V_\epsilon}^2 \|U_\epsilon\|_{\mathcal{D}_\epsilon(A)}^2 dt + \xi \E\int_\eta^\zeta \|U_\epsilon\|_{\mathcal{D}_\epsilon(A)}^2 dt\\
    + C\epsilon^2 (1+ \alpha_\sigma^4)\E\int_\eta^\zeta \sum_k (1+\|v\|_{H^3}^2)(1+\|(a - a^K)\nabla w\|_{H^1}^2 + \|(I-K*)[(\phi_k \bcdot \nabla_3)w]\|_{H^1}^2)(1+\|U_\epsilon\|_{V_\epsilon}^2+  \|\Theta_\epsilon\|_{V}^2) dt.
\end{multline}
Then we deduce that, for some constant $C_1 >0$,
\begin{multline}
    \E \sup_{[\eta, \zeta]} (\|U_\epsilon\|_{V_\epsilon}^2 + \|\Theta_\epsilon\|_{V}^2) + \E \int_\eta^\zeta(\|U_\epsilon\|_{\mathcal{D}_\epsilon(A)}^2 + \|\Theta_\epsilon\|_{\mathcal{D}(A)}^2) dt \leq C_1 \E \int_\eta^\zeta  \|U_\epsilon\|_{V_\epsilon}^2 \|U_\epsilon\|_{\mathcal{D}_\epsilon(A)}^2 dt \\
    + C_1\epsilon^2 (1+ \alpha_\sigma^4)\E\int_\eta^\zeta (1+\|v\|_{H^3}^2)(1+\|(a - a^K)\nabla w\|_{H^1}^2 + \sum_k \|(I-K*)[(\phi_k \bcdot \nabla_3)w]\|_{H^1}^2)(1+\|U_\epsilon\|_{V_\epsilon}^2+  \|\Theta_\epsilon\|_{V}^2) dt.
\end{multline}
Denote by
\begin{multline}
    \check{\tau}_\epsilon = \inf \Big\{t \in \R_+ \Big| \sup_{[0, t]} (\|U_\epsilon\|_{V_\epsilon}^2 + \|\Theta_\epsilon\|_{V}^2) + \int_0^t(\| U_\epsilon\|_{\mathcal{D}_\epsilon(A)}^2 + \|\Theta_\epsilon\|_{\mathcal{D}(A)}^2) dt > \frac{1}{4C_1} \Big\} \leq \tau_\epsilon.
\end{multline}
Therefore, by the stochastic Grönwall lemma and Lemma \ref{eq:improved-estimate-PE}, if $\epsilon < \frac{1}{\sqrt{2C_1}}$,
\begin{multline}
    \E \sup_{[0, \check{\tau}_{\epsilon} \wedge \tau_\delta \wedge T]} (\|U_\epsilon\|_{V_\epsilon}^2 + \|\Theta_\epsilon\|_{V}^2) + \E \int_{0}^{\check{\tau}_{\epsilon} \wedge \tau_\delta \wedge T} (\|U_\epsilon\|_{\mathcal{D}_\epsilon(A)}^2 + \| \Theta_\epsilon\|_{\mathcal{D}(A)}^2) dt \\
    \leq C_1\epsilon^2 (1+ \alpha_\sigma^4)\E\int_{0}^{\check{\tau}_{\epsilon} \wedge \tau_\delta \wedge T} (1+\|v\|_{H^3}^2)(1+\|(a - a^K)\nabla w\|_{H^1}^2 + \sum_k \|(I-K*)[(\phi_k \bcdot \nabla_3)w]\|_{H^1}^2) dt.
\end{multline}
Concerning the asymptotic behaviour of $\tau_\epsilon$, the last equation implies that the following convergence happens in probability, whenever $\alpha_\sigma = o(\epsilon^{-1/2})$,
\begin{equation}
    \sup_{[0, \check{\tau}_{\epsilon} \wedge \tau_\delta \wedge T]} (\|U_{\epsilon}\|_{V_\epsilon}^2 + \|\Theta_{\epsilon}\|_{V}^2) + \int_{0}^{\check{\tau}_{\epsilon} \wedge \tau_\delta \wedge T} \Big(\|U_{\epsilon}\|_{\mathcal{D}_{\epsilon}(A)}^2 + \|\Theta_{\epsilon}\|_{\mathcal{D}(A)}^2\Big) dt \rightarrow 0. \label{eq-convergenceTo0-E1eps}
\end{equation}
Then, $\check{\tau}_{\epsilon} \wedge \tau_\delta \rightarrow \tau_\delta$ in probability. Since $\check{\tau}_\epsilon \leq \tau_\epsilon$, we infer that $\tau_\epsilon \wedge \tau_\delta \rightarrow \tau_\delta$ in probability as well. In addition, remind that, for all stopping times $\eta<\zeta$, $\mathcal{E}_\epsilon^1(\eta,\zeta)$ is defined by
\begin{equation}
    \mathcal{E}_\epsilon^1(\eta,\zeta) := \left\{
    \begin{array}{ll}
        \sup_{[\eta, \zeta]} \frac{1}{2}\Big(\|U_\epsilon\|_{V_\epsilon^2}^2+\|\Theta_\epsilon\|_{V}^2\Big) + \int_{\eta}^{\zeta} \Big(\|U_\epsilon\|_{D_\epsilon(A)^2}^2 + \|\Theta_{\epsilon}\|_{\mathcal{D}(A)}^2\Big) dt \text{ if $\zeta \geq \tau_\epsilon$},\\
        \infty \text{ otherwise}.
    \end{array}
    \right.
\end{equation}
With this notation, equation \eqref{eq-convergenceTo0-E1eps} and the fact that $\check{\tau}_{\epsilon} \wedge \tau_\delta \rightarrow \tau_\delta$ yield,
\begin{equation}
    \mathcal{E}_\epsilon^1(0, \tau_\delta \wedge T) \rightarrow 0, \text{ in probability.}
\end{equation}
Hence, by taking the limit $\delta \rightarrow \infty$, we deduce that $\mathcal{E}_\epsilon^1(0, T) \rightarrow 0$ in probability.

\CQFD

\noindent
\textbf{Acknowledgements:}
The authors acknowledge the support of the ERC EU project 856408-STUOD and benefit from the support of the French government ``Investissements d'Avenir'' program ANR-11-LABX-0020-01.

\bigskip

\noindent
\textbf{Conflict of interests:}
We declare we have no competing interests.

\bigskip

\noindent
\textbf{Data availability:}
The data that support this study are openly available.

\printbibliography[title=References]

@misc{AV2024-NS-tranport-noise,
      title={Stochastic Navier-Stokes equations for turbulent flows in critical spaces}, 
      author={Antonio Agresti and Mark Veraar},
      year={2023},
      eprint={2107.03953},
      archivePrefix={arXiv},
      primaryClass={math.AP},
      url={https://arxiv.org/abs/2107.03953}, 
}

@misc{DMM2025,
      title={Stochastic interpretations of the oceanic primitive equations with relaxed hydrostatic assumptions}, 
      author={Arnaud Debussche and Étienne Mémin and Antoine Moneyron},
      year={2025},
      eprint={2502.14946},
      archivePrefix={arXiv},
      primaryClass={math.AP},
      url={https://arxiv.org/abs/2502.14946}, 
}

@book{pedlosky2013geophysical,
  title={Geophysical fluid dynamics},
  author={Pedlosky, Joseph},
  year={2013},
  publisher={Springer Science \& Business Media}
}

@article{LT2019,
title = {The primitive equations as the small aspect ratio limit of the Navier–Stokes equations: Rigorous justification of the hydrostatic approximation},
journal = {Journal de Mathématiques Pures et Appliquées},
volume = {124},
pages = {30-58},
year = {2019},
issn = {0021-7824},
doi = {https://doi.org/10.1016/j.matpur.2018.04.006},
url = {https://www.sciencedirect.com/science/article/pii/S0021782418300503},
author = {Jinkai Li and Edriss S. Titi}
}

@article{PF2001,
author = {Az\'{e}rad, Pascal and Guill\'{e}n, Francisco},
title = {Mathematical Justification of the Hydrostatic Approximation in the Primitive Equations of Geophysical Fluid Dynamics},
journal = {SIAM Journal on Mathematical Analysis},
volume = {33},
number = {4},
pages = {847-859},
year = {2001},
doi = {10.1137/S0036141000375962},
URL = {https://doi.org/10.1137/S0036141000375962},
eprint = {https://doi.org/10.1137/S0036141000375962}
}

@article{DHM_2023,
	author = {Debussche, Arnaud and Hug, B{\'e}renger and M{\'e}min, Etienne},
	doi = {10.1007/s00021-023-00764-0},
	id = {Debussche2023},
	isbn = {1422-6952},
	journal = {Journal of Mathematical Fluid Mechanics},
	number = {1},
	pages = {19},
	title = {A Consistent Stochastic Large-Scale Representation of the Navier--Stokes Equations},
	url = {https://doi.org/10.1007/s00021-023-00764-0},
	volume = {25},
	year = {2023},
    month ={1}}

@article{DGHT_2011,
	doi = {10.1016/j.physd.2011.03.009},
  
	url = {https://doi.org/10.1016%2Fj.physd.2011.03.009},
  
	year = 2011,
	month = {7},
  
	publisher = {Elsevier {BV}},
  
	volume = {240},
  
	number = {14-15},
  
	pages = {1123--1144},
  
	author = {Arnaud Debussche and Nathan Glatt-Holtz and Roger Temam},
  
	title = {Local martingale and pathwise solutions for an abstract fluids model},
  
	journal = {Physica D: Nonlinear Phenomena}
}

@article{DGHTZ_2012,
	doi = {10.1088/0951-7715/25/7/2093},
  
	url = {https://doi.org/10.1088%2F0951-7715%2F25%2F7%2F2093},
  
	year = 2012,
	month = {6},
  
	publisher = {{IOP} Publishing},
  
	volume = {25},
  
	number = {7},
  
	pages = {2093--2118},
  
	author = {Debussche, Arnaud and Glatt-Holtz, Nathan and Temam, Roger and Ziane, Mohammed},
  
	title = {Global existence and regularity for the 3D stochastic primitive equations of the ocean and atmosphere with multiplicative white noise},
  
	journal = {Nonlinearity}
}

@article{BS_2021,
	author = {Zdzislaw Brze{{\'z}}niak and Jakub Slav{\'i}k},
	doi = {https://doi.org/10.1016/j.jde.2021.05.049},
	issn = {0022-0396},
	journal = {Journal of Differential Equations},
	pages = {617-676},
	title = {Well-posedness of the 3D stochastic primitive equations with multiplicative and transport noise},
	url = {https://www.sciencedirect.com/science/article/pii/S0022039621003521},
	volume = {296},
	year = {2021}}

@article{FG_1995,
	title = {Martingale and stationary solutions for stochastic {Navier}-{Stokes} equations},
	volume = {102},
	issn = {1432-2064},
	url = {https://doi.org/10.1007/BF01192467},
	doi = {10.1007/BF01192467},
	number = {3},
	journal = {Probability Theory and Related Fields},
	author = {Flandoli, Franco and Gatarek, Dariusz},
	month = sep,
	year = {1995},
	pages = {367--391}
}

@article{AHHS_2022,
	author = {Antonio Agresti and Matthias Hieber and Amru Hussein and Martin Saal},
	doi = {10.1007/s40072-022-00277-3},
	id = {Agresti2022},
	isbn = {2194-041X},
	journal = {Stochastics and Partial Differential Equations: Analysis and Computations},
	title = {The stochastic primitive equations with transport noise and turbulent pressure},
	url = {https://doi.org/10.1007/s40072-022-00277-3},
	year = {2022},
    month= {10}}

@misc{AHHS_2022_preprint,
      title={The stochastic primitive equations with non-isothermal turbulent pressure}, 
      author={Antonio Agresti and Matthias Hieber and Amru Hussein and Martin Saal},
      year={2023},
      eprint={2210.05973},
      archivePrefix={arXiv},
      primaryClass={math.AP}
}

@article{Mémin_2014,
  title={Fluid flow dynamics under location uncertainty},
  author={M{\'e}min, Etienne},
  journal={Geophysical \& Astrophysical Fluid Dynamics},
  volume={108},
  number={2},
  pages={119--146},
  year={2014},
  publisher={Taylor \& Francis}
}

@inproceedings{TML_2023,
	address = {Cham},
	title = {Primitive {Equations} {Under} {Location} {Uncertainty}: {Analytical} {Description} and {Model} {Development}},
	isbn = {978-3-031-18988-3},
	booktitle = {Stochastic {Transport} in {Upper} {Ocean} {Dynamics}},
	publisher = {Springer International Publishing},
	author = {Tucciarone, Francesco L and Mémin, Etienne and Li, Long},
	editor = {Chapron, Bertrand and Crisan, Dan and Holm, Darryl and Mémin, Etienne and Radomska, Anna},
	year = {2023},
	pages = {287--300}
}

@article{CT_2007,
  title={Global well-posedness of the three-dimensional viscous primitive equations of large scale ocean and atmosphere dynamics},
  author={Cao, Chongsheng and Titi, Edriss S},
  journal={Annals of Mathematics},
  pages={245--267},
  year={2007},
  publisher={JSTOR}
}

@article{Berner-Al_2017,
  title={Stochastic parameterization: Toward a new view of weather and climate models},
  author={Berner, Judith and Achatz, Ulrich and Batte, Lauriane and Bengtsson, Lisa and De La Camara, Alvaro and Christensen, Hannah M and Colangeli, Matteo and Coleman, Danielle RB and Crommelin, Daan and Dolaptchiev, Stamen I and others},
  journal={Bulletin of the American Meteorological Society},
  volume={98},
  number={3},
  pages={565--588},
  year={2017},
  publisher={American Meteorological Society}
}

@article{FOBWL_2015,
  title={Stochastic climate theory and modeling},
  author={Franzke, Christian LE and O'Kane, Terence J and Berner, Judith and Williams, Paul D and Lucarini, Valerio},
  journal={Wiley Interdisciplinary Reviews: Climate Change},
  volume={6},
  number={1},
  pages={63--78},
  year={2015},
  publisher={Wiley Online Library}
}

@book{FO_2017,
  title={Nonlinear and stochastic climate dynamics},
  author={Franzke, Christian LE and O'Kane, Terence J},
  year={2017},
  publisher={Cambridge University Press}
}

@article{MTV_1999,
  title={Models for stochastic climate prediction},
  author={Majda, Andrew J and Timofeyev, Ilya and Vanden Eijnden, Eric},
  journal={Proceedings of the National Academy of Sciences},
  volume={96},
  number={26},
  pages={14687--14691},
  year={1999},
  publisher={National Acad Sciences}
}

@article{Leith_1990,
  title={Stochastic backscatter in a subgrid-scale model: Plane shear mixing layer},
  author={Leith, CE},
  journal={Physics of Fluids A: Fluid Dynamics},
  volume={2},
  number={3},
  pages={297--299},
  year={1990},
  publisher={American Institute of Physics}
}

@article{MT_1992,
  title={Stochastic backscatter in large-eddy simulations of boundary layers},
  author={Mason, Paul J and Thomson, David J},
  journal={Journal of Fluid Mechanics},
  volume={242},
  pages={51--78},
  year={1992},
  publisher={Cambridge University Press}
}

@article{BMP_1999,
  title={Stochastic representation of model uncertainties in the ECMWF ensemble prediction system},
  author={Buizza, Roberto and Miller, M and Palmer, Tim N},
  journal={Quarterly Journal of the Royal Meteorological Society},
  volume={125},
  number={560},
  pages={2887--2908},
  year={1999},
  publisher={Wiley Online Library}
}

@article{Shutts_2005,
  title={A kinetic energy backscatter algorithm for use in ensemble prediction systems},
  author={Shutts, Glenn},
  journal={Quarterly Journal of the Royal Meteorological Society: A journal of the atmospheric sciences, applied meteorology and physical oceanography},
  volume={131},
  number={612},
  pages={3079--3102},
  year={2005},
  publisher={Wiley Online Library}
}

@article{Schmitt_2007,
  title={About Boussinesq's turbulent viscosity hypothesis: historical remarks and a direct evaluation of its validity},
  author={Schmitt, Fran{\c{c}}ois G},
  journal={Comptes Rendus M{\'e}canique},
  volume={335},
  number={9-10},
  pages={617--627},
  year={2007},
  publisher={Elsevier}
}

@article{CDMR_2018,
  title={Large-scale flows under location uncertainty: a consistent stochastic framework},
  author={Chapron, Bertrand and D{\'e}rian, Pierre and M{\'e}min, Etienne and Resseguier, Valentin},
  journal={Quarterly Journal of the Royal Meteorological Society},
  volume={144},
  number={710},
  pages={251--260},
  year={2018},
  publisher={Wiley Online Library}
}

@article{BCCLM_2020,
  title={Deciphering the role of small-scale inhomogeneity on geophysical flow structuration: a stochastic approach},
  author={Bauer, Werner and Chandramouli, Pranav and Chapron, Bertrand and Li, Long and M{\'e}min, Etienne},
  journal={Journal of Physical Oceanography},
  volume={50},
  number={4},
  pages={983--1003},
  year={2020},
  publisher={American Meteorological Society}
}

@article{RMC_2017_Pt1,
  title={Geophysical flows under location uncertainty, Part I Random transport and general models},
  author={Resseguier, Valentin and M{\'e}min, Etienne and Chapron, Bertrand},
  journal={Geophysical \& Astrophysical Fluid Dynamics},
  volume={111},
  number={3},
  pages={149--176},
  year={2017},
  publisher={Taylor \& Francis}
}

@article{RMC_2017_Pt2,
  title={Geophysical flows under location uncertainty, part II quasi-geostrophy and efficient ensemble spreading},
  author={Resseguier, Valentin and M{\'e}min, Etienne and Chapron, Bertrand},
  journal={Geophysical \& Astrophysical Fluid Dynamics},
  volume={111},
  number={3},
  pages={177--208},
  year={2017},
  publisher={Taylor \& Francis}
}

@article{RMC_2017_Pt3,
  title={Geophysical flows under location uncertainty, Part III SQG and frontal dynamics under strong turbulence conditions},
  author={Resseguier, Valentin and M{\'e}min, Etienne and Chapron, Bertrand},
  journal={Geophysical \& Astrophysical Fluid Dynamics},
  volume={111},
  number={3},
  pages={209--227},
  year={2017},
  publisher={Taylor \& Francis}
}

@article{RMHC_2017,
  title={Stochastic modelling and diffusion modes for proper orthogonal decomposition models and small-scale flow analysis},
  author={Resseguier, Valentin and M{\'e}min, Etienne and Heitz, Dominique and Chapron, Bertrand},
  journal={Journal of Fluid Mechanics},
  volume={826},
  pages={888--917},
  year={2017},
  publisher={Cambridge University Press}
}

@article{RPMC_2021,
  title={Quantifying truncation-related uncertainties in unsteady fluid dynamics reduced order models},
  author={Resseguier, Valentin and Picard, Agustin M and M{\'e}min, Etienne and Chapron, Bertrand},
  journal={SIAM/ASA Journal on Uncertainty Quantification},
  volume={9},
  number={3},
  pages={1152--1183},
  year={2021},
  publisher={SIAM}
}

@article{TCM_2021,
  title={Stochastic linear modes in a turbulent channel flow},
  author={Tissot, Gilles and Cavalieri, Andr{\'e} VG and M{\'e}min, Etienne},
  journal={Journal of Fluid Mechanics},
  volume={912},
  pages={A51},
  year={2021},
  publisher={Cambridge University Press}
}

@article{CHLM_2018,
  title={Coarse large-eddy simulations in a transitional wake flow with flow models under location uncertainty},
  author={Chandramouli, Pranav and Heitz, Dominique and Laizet, Sylvain and M{\'e}min, Etienne},
  journal={Computers \& Fluids},
  volume={168},
  pages={170--189},
  year={2018},
  publisher={Elsevier}
}

@article{CMH_2020,
  title={4D large scale variational data assimilation of a turbulent flow with a dynamics error model},
  author={Chandramouli, Pranav and M{\'e}min, Etienne and Heitz, Dominique},
  journal={Journal of Computational Physics},
  volume={412},
  pages={109446},
  year={2020},
  publisher={Elsevier}
}

@article{HM_2017,
  title={Stochastic representation of the Reynolds transport theorem: revisiting large-scale modeling},
  author={Harouna, Souleymane K and M{\'e}min, Etienne},
  journal={Computers \& Fluids},
  volume={156},
  pages={456--469},
  year={2017},
  publisher={Elsevier}
}

@article{BLBM_2021,
  title={Rotating shallow water flow under location uncertainty with a structure-preserving discretization},
  author={Brecht, R{\"u}diger and Li, Long and Bauer, Werner and M{\'e}min, Etienne},
  journal={Journal of Advances in Modeling Earth Systems},
  volume={13},
  number={12},
  pages={e2021MS002492},
  year={2021},
  publisher={Wiley Online Library}
}

@article{CL_1976,
  title={A rational model for Langmuir circulations},
  author={Craik, Alex DD and Leibovich, Sidney},
  journal={Journal of Fluid Mechanics},
  volume={73},
  number={3},
  pages={401--426},
  year={1976},
  publisher={Cambridge University Press}
}

@article{MSM_1997,
  title={Langmuir turbulence in the ocean},
  author={McWilliams, James C and Sullivan, Peter P and Moeng, Chin-Hoh},
  journal={Journal of Fluid Mechanics},
  volume={334},
  pages={1--30},
  year={1997},
  publisher={Cambridge University Press}
}

@article{GHZ_2009,
author = {Nathan Glatt-Holtz and Mohammed Ziane},
title = {{Strong pathwise solutions of the stochastic Navier-Stokes system}},
volume = {14},
journal = {Advances in Differential Equations},
number = {5/6},
publisher = {Khayyam Publishing, Inc.},
pages = {567 -- 600},
year = {2009},
doi = {10.57262/ade/1355867260},
URL = {https://doi.org/10.57262/ade/1355867260}
}

@book{Vallis2017,
  title={Atmospheric and oceanic fluid dynamics},
  author={Vallis, Geoffrey K},
  year={2017},
  publisher={Cambridge University Press}
}

@article{CFH_2019,
	author = {Crisan, Dan and Flandoli, Franco and Holm, Darryl},
	journal = {Journal of Nonlinear Science},
	number = {3},
	pages = {813--870},
	publisher = {Springer},
	title = {Solution properties of a 3{D} stochastic {E}uler fluid equation},
	volume = {29},
	year = {2019}}

@article{GCL_2023,
  title={Existence and uniqueness of maximal solutions to SPDEs with applications to viscous fluid equations},
  author={Goodair, Daniel and Crisan, Dan and Lang, Oana},
  journal={Stochastics and Partial Differential Equations: Analysis and Computations},
  pages={1--64},
  year={2023},
  publisher={Springer}
}

@article{LCM_2023,
	author = {Lang, Oana and Crisan, Dan and M{\'e}min, Etienne},
	date = {2023-02-20},
	doi = {10.1007/s00021-023-00769-9},
	id = {Lang2023},
	isbn = {1422-6952},
	journal = {Journal of Mathematical Fluid Mechanics},
	number = {2},
	pages = {29},
	title = {Analytical Properties for a Stochastic Rotating Shallow Water Model Under Location Uncertainty},
	url = {https://doi.org/10.1007/s00021-023-00769-9},
	volume = {25},
	year = {2023}}

@article{MR_2005,
	author = {Remigijus Mikulevicius and Boris L Rozovskii},
	date = {2005-01-01},
	doi = {10.1214/009117904000000630},
	journal = {The Annals of Probability},
	journal1 = {The Annals of Probability},
	journal2 = {The Annals of Probability},
	month = {1},
	number = {1},
	pages = {137--176},
	title = {Global {L}$_{2}$-solutions of stochastic {N}avier--{S}tokes equations},
	url = {https://doi.org/10.1214/009117904000000630},
	volume = {33},
	year = {2005}}

@article{FGL_2022,
	author = {Flandoli, Franco and Galeati, Lucio and Luo, Dejun},
	date = {2022-01-31},
	doi = {10.1098/rsta.2021.0096},
	journal = {Philosophical Transactions of the Royal Society A: Mathematical, Physical and Engineering Sciences},
	journal1 = {Philosophical Transactions of the Royal Society A: Mathematical, Physical and Engineering Sciences},
	journal2 = {Philosophical Transactions of the Royal Society A: Mathematical, Physical and Engineering Sciences},
	number = {2219},
	pages = {20210096},
	publisher = {Royal Society},
	title = {Eddy heat exchange at the boundary under white noise turbulence},
	type = {doi: 10.1098/rsta.2021.0096},
	url = {https://doi.org/10.1098/rsta.2021.0096},
	volume = {380},
	year = {2022},
	year1 = {2022}}

@article{FGP_10,
	author = {Flandoli, Franco and Gubinelli, Massimiliano and Priola, Enrico },
	date = {2010-04-01},
	doi = {10.1007/s00222-009-0224-4},
	id = {Flandoli2010},
	isbn = {1432-1297},
	journal = {Inventiones mathematicae},
	number = {1},
	pages = {1--53},
	title = {Well-posedness of the transport equation by stochastic perturbation},
	url = {https://doi.org/10.1007/s00222-009-0224-4},
	volume = {180},
	year = {2010}}

@article{FR_2023,
	archiveprefix = {arXiv},
	author = {Franco Flandoli and Francesco Russo},
	eprint = {2305.19293},
	journal = {ArXiv},
	primaryclass = {math.PR},
	title = {Reduced dissipation effect in stochastic transport by Gaussian noise with regularity greater than 1/2},
	volume = {2305.19293},
	year = {2023}}

@article{FL_2021,
	author = {Flandoli, Franco and Luo, Dejun},
	date = {2021-06-01},
	doi = {10.1007/s00440-021-01037-5},
	id = {Flandoli2021},
	isbn = {1432-2064},
	journal = {Probability Theory and Related Fields},
	number = {1},
	pages = {309--363},
	title = {High mode transport noise improves vorticity blow-up control in 3D Navier--Stokes equations},
	url = {https://doi.org/10.1007/s00440-021-01037-5},
	volume = {180},
	year = {2021}}

@article{FGL_2021,
	author = {Franco Flandoli and Lucio Galeati and Dejun Luo},
	doi = {10.1080/03605302.2021.1893748},
	eprint = {https://doi.org/10.1080/03605302.2021.1893748},
	journal = {Communications in Partial Differential Equations},
	number = {9},
	pages = {1757-1788},
	publisher = {Taylor & Francis},
	title = {Delayed blow-up by transport noise},
	url = {https://doi.org/10.1080/03605302.2021.1893748},
	volume = {46},
	year = {2021}}

@article{chueshov2008random,
  title={Random kick-forced 3D Navier--Stokes equations in a thin domain},
  author={Chueshov, Igor and Kuksin, Sergei},
  journal={Archive for Rational Mechanics and Analysis},
  volume={188},
  number={1},
  pages={117--153},
  year={2008},
  publisher={Springer}
}

\end{document}